\documentclass[a4paper, reqno]{amsart}
\usepackage{amssymb, amsmath, amscd}
\usepackage{amssymb, amsmath, amscd}
\usepackage[final]{graphicx}
\usepackage{color}
\def\red{\color{red}}

\newtheorem{theorem}{Theorem}[section]
\newtheorem{lemma}[theorem]{Lemma}
\newtheorem{assertion}[theorem]{Assertion}
\newtheorem{remark}[theorem]{Remark}
\newtheorem{definition}[theorem]{Definition}
\newtheorem{corollary}[theorem]{Corollary}

\newtheorem{note}[theorem]{Note}
\newtheorem{proposition}[theorem]{Proposition}
\newtheorem{claim}[theorem]{Claim}
\makeatletter

\@addtoreset{equation}{section}
\makeatother
\begin{document}
\title[Hessian]{Hessian of Busemann functions and rank of  Hadamard manifolds
}
\author{Mitsuhiro Itoh}
\address{University of Tsukuba,
1-1-1 Tennodai, Tsukuba, Ibaraki 305-8577 Japan}
\email{itohm@math.tsukuba.ac.jp}
\author{Sinwhi Kim}
\address{Graduate school of Mathematics,\, Sungkyunkwan University, Suwon 16419, Korea}
\email{kimsinhwi@skku.edu}
\author{JeongHyeong\  Park}
\address{Department of Mathematics,\, Sungkyunkwan University, Suwon 16419, Korea}
\email{parkj@skku.edu}
\author{Hiroyasu Satoh}
\address{Nippon Institute of Technology, Japan}
\email{hiroyasu@nit.ac.jp}

\begin{abstract}In this article we show that every geodesic is rank one and the Hessian of Busemann functions is positive definite for a harmonic Damek-Ricci space, a two step solvable Lie group with a left invariant metric. Moreover, the eigenspace of the Hessian of Busemann functions on a Hadamard manifold $(M,g)$ corresponding to eigenvalue zero is investigated with respect to rank of geodesics.  On a harmonic Hadamard manifold which is of purely exponential volume growth, or of hypergeometric type it is shown that every Busemann function admits positive definite Hessian. A criterion for $(M,g)$ fulfilling  visibility axiom is presented in terms of positive definiteness of the Hessian of Busemann functions.
\end{abstract}
\maketitle

\section{Introduction and main theorems}

Let $(M,g)$ be a Hadamard manifold, namely, a simply
connected, complete Riemannian manifold of non-positive sectional
curvature. A geometric notion important for a Hadamard
manifold is a Busemann function
$b_{\gamma}\hspace{0.5mm};\hspace{0.5mm}M \rightarrow {\Bbb R}$,
associated with a unit speed geodesic $\gamma :
{\Bbb R} \rightarrow M$. Each Busemann function $b_{\gamma}$ which is
$C^2$, convex over $M$, 
induces level hypersurfaces $b_{\gamma}^{-1}(p)$, $p\in M$ called horospheres, in $M$ and admits positive semi-definite Hessian
$\nabla d b_{\gamma}$. 

For a Hadamard manifold an ideal boundary $\partial M$ is defined as the set of 
geodesic rays on $M$ modulo asymptotical equivalence relation.
The Busemann function 
 is well-defined associated with an ideal boundary point $\theta$ of $\partial M$ represented by a geodesic ray $\gamma$, denoted by $b_{\theta}$, up to additive constants{\red{.}}


 If a Hadamard manifold $(M,g)$ is a rank one symmetric space of non-compact type, then,
$(M,g)$ is either a real, complex, quaternionic hyperbolic space, or 16 dimensional Cayley hyperbolic space, for which the
Hessian of the Busemann function $\nabla d b_{\theta}$ is
described as
\begin{eqnarray*}
\nabla d b_{\theta}(u,v) = \langle u, v \rangle - \langle \nabla b_{\theta},u\rangle\, \langle \nabla b_{\theta}, v\rangle + \sum_{k=1}^{d-1} \langle \nabla b_{\theta}, J_k u\rangle\, \langle \nabla b_{\theta}, J_k v \rangle,
\end{eqnarray*}with respect to the corresponding structures $\{J_1,\cdots, J_{d-1}\}$ where $d$ is the dimension of the field ${\Bbb F}$ defining the hyperbolic space. Here, sectional curvature $K$ is  normalized as constant  $ -1$ for ${\Bbb F} = {\Bbb R}$, and $-4 \leq K \leq -1$ for ${\Bbb F} = {\Bbb C}$, ${\Bbb H}$ and ${\Bbb O}$, respectively. So, except for the real hyperbolic space each hyperbolic space admits the eigenspace $\oplus_k\hspace{0.5mm} {\Bbb R}\hspace{0.5mm}  J_k \nabla b_{\gamma}$ with eigenvalue $\lambda_1 = 2$ and the eigenspace $\left(\oplus_k\hspace{0.5mm}
{\Bbb R}\hspace{0.5mm}  J_k \nabla b_{\gamma}\right)^{\perp}$  
with eigenvalue $\lambda_2=1$.
Therefore, for every
hyperbolic space the Hessian of the Busemann function is
positive definite over $\nabla b_{\theta}^{\perp}$, 
 and the constant eigenvalues
$\lambda_1$, $\lambda_2$ with constant multiplicities $\mu_1$,
$\mu_2$. In case of the real hyperbolic space the Hessian of
the Busemann function admits  only single eigenvalue $1$.

These hyperbolic spaces 
are non-compact harmonic Hadamard manifolds.  A Riemannian manifold is called harmonic, if it admits a solution of Laplace equation $\Delta f = 0$ which
depends only on the distance function. However, there are non-symmetric, harmonic Riemannian homogeneous Hadamard manifolds, called Damek-Ricci spaces.
It is interesting to investigate rank one geodesic property and the positive definiteness for  the Hessian of Busemann functions of a
Damek-Ricci space, compared with  those hyperbolic spaces.
The rank of a geodesic $\gamma$ is dimension of the parallel Jacobi fields along $\gamma$.  
The rank
$r(M)$ of $(M,g)$ is defined by $r(M) =
\inf_{\gamma} {\rm rank} \gamma$, a generalization of the rank
of symmetric space.
See \cite{BBE} for the notion
of rank\hspace{0.5mm}. The rank of any geodesic is one for a rank one symmetric space of non-compact type.

This article is comprised of two parts, {\sc{Part}} I and {\sc{Part}} II. In {\sc{Part}} I we present for a Damek-Ricci space $S$ the rank one geodesic property and the positive definiteness of the Hessian of an arbitrary Busemann function by exploiting the algebraic structure of the Lie algebra ${\mathfrak s}$ of $S$. In {\sc Part II} the main results are presented,
namely Theorem 1.4 and Corollary 1.5.


A Damek-Ricci space $S$ is a two step solvable Lie group with a left invariant metric which is a harmonic Hadamard manifold, a particularly important manifold
which plays an significant role in the counter example of the Lichnerowicz conjecture. For basic properties of a Damek-Ricci space refer to \cite{BTV}.

The following are our main results of {\sc Part I}.

\begin{theorem}\label{damekricci-2}{\rm
Any geodesic of a Damek-Ricci space is rank one.
}\end{theorem}

\begin{theorem}\label{damekricci-2a}{\rm
For any Busemann function $b_{\gamma}$ of a Damek-Ricci space the Hessian is positive definite over the orthogonal complement $\nabla b_{\gamma}^{\hspace{1mm}\perp}$ to the gradient field $\nabla b_{\theta}$.
}
\end{theorem}
These theorems are apparently independent. However, the rank of a geodesic and the eigenspace of the Hessian of the Busemann function corresponding to eigenvalue zero are essentially related, as illustrated in Part II. They are obtained by using arguments of the Lie algebra structure of a Damek-Ricci space.

We will give two proofs to Theorem \ref{damekricci-2}, one  by an argument based on a subtle relation on eigenvalues of the Jacobi operator together with an explicit formula of a geodesic on a Damek-Ricci space $S$ at \ref{kimproof}, in Section 3, Part I and the other one by using flat plane section field along a geodesic in Appendix 1, Part I.
Theorem \ref{damekricci-2a} is shown by getting formulae of the Hessian $\nabla d b_{\gamma}$ of the Busemann function at the identity $e_S$ and
computing the components of  $\nabla d b_{\gamma}$ by the aid of the Lie algebra structure of ${\mathfrak s}$.
Theorem \ref{damekricci-2a} will be given in Section 4, Part I. Refer to \cite{Satoh} for another proof. Moreover we have the following.
\begin{theorem}\label{positivedefiniteness}
\rm Let $S$ be a Damek-Ricci space. Then  
there exists a positive constant $C_0$ such that for any $\theta\in \partial S$ and $p\in S$
\begin{eqnarray*}\nabla d b_{\theta}(u,u) \geq C_0 \vert u\vert^2,\hspace{2mm}\forall u\in T_p {\mathcal H}_{(\theta,p)}.
\end{eqnarray*}
\end{theorem} 



 In {\sc Part II}  the following are investigated with respect to  Hadamard manifolds including Damek-Ricci spaces.
\begin{theorem}\label{positivehessian}\hspace{2mm}{\rm
Let $(M,g)$ be an analytic Hadamard manifold.

Let $\theta\in\partial M$ and $\gamma : {\Bbb R} \rightarrow M$ be a geodesic, parametrized by arc-length and with $[\gamma]=\theta$. Let $b_{\theta}$ be the Busemann function associated with $\theta$.  Assume $b_{\theta}$ is an analytic function. Then,
\begin{eqnarray}
{\rm rank}\, \gamma = \dim ({\Bbb E}_{b_{\theta}})_{\gamma(t)} + 1
\end{eqnarray}for any $t$. Here, $({\Bbb E}_{b_{\theta}})_{\gamma(t)} := \{ u\in (\nabla b_{\theta}^{\perp})_{\gamma(t)}\, \vert\, (\nabla d b_{\theta})_{\gamma(t)}(u,u) = 0 \}$ denotes the eigenspace of $\nabla d b_{\theta}$ over $\nabla b_{\theta}^{\hspace{1mm}\perp}$ corresponding to eigenvalue zero.}
\end{theorem}
\vspace{2mm}
Consequently, we get the following.
\begin{corollary}\rm Let $(M,g)$ be an analytic Hadamard manifold. Assume each Busemann function $b_{\theta}$ is analytic in $M$.  Let $\gamma$ be a geodesic on $M$. Then, $\gamma$ has rank one if and only if  $\nabla d b_{\theta}$, $\theta=[\gamma]$ 
is positive definite
over $b_{\theta}^{\hspace{1mm}\perp}$.
\end{corollary}

Theorem \ref{positivehessian} is obtained by exploiting the Riccati equation together with a double induction argument. 
The Riccati equation,
given at (\ref{riccati-equation}), in \S 2 describes a time-variation of the shape operator ${\mathcal S}(t)$ 
of horospheres 
along a geodesic $\gamma$ of $[\gamma]=\theta\in\partial M$.

\vspace{2mm}
Now let $(M,g)$ be a Hadamard manifold which is harmonic.
Then $(M,g)$ is Einstein and hence analytic.   Every Busemann function on such $M$ is analytic (cf. \cite{RShah} for this). It is well known that a
Damek-Ricci space is harmonic. It is expected that every geodesic is
rank one on such $(M,g)$ as in \cite{Kn} where G.  Knieper conjectured
for a harmonic non-compact manifold  and obtained the following by
assuming that $(M,g)$ is of purely exponential volume growth (see
\cite{Kn, Kn-2}).

\begin{proposition}\label{prop1.4}\rm
 Let $(M,g)$ be a non-compact harmonic manifold of volume entropy $Q > 0$. If $(M,g)$ is of purely exponential volume growth. Then each geodesic of $(M,g)$ is rank one.

\end{proposition}
Here, a non-compact, harmonic manifold of volume entropy $Q > 0$ is said to be of purely exponential volume growth if there exist constants $0 < a \leq b$ such that
$\displaystyle{ a\hspace{0.5mm}e^{Q\hspace{0.5mm}r} \leq \Theta(r) \leq b\hspace{0.5mm} e^{Q\hspace{0.5mm}r}
}$ 
 for all $r \geq 1$, with respect to the volume density function $\Theta(r)$ of a geodesic sphere of radius $r > 0$ (see \cite{Kn} for the details).

Since a Hadamard manifold is non-compact, we have from Theorem \ref{positivehessian} and
Proposition \ref{prop1.4}.
\begin{proposition}\label{prop1.5}\rm
Let $(M,g)$ be a harmonic Hadamard manifold of purely exponential volume growth. Then the Hessian $\nabla d b_{\theta}$ of  each Busemann function $b_{\theta}$ is positive definite over $\nabla b_{\theta}^{\perp}$.

\end{proposition}

 Theorem \ref{damekricci-2a} is then obtained from Proposition {\ref{prop1.5}\, not using the Lie algebra argument, since any Damek-Ricci space $S$ is of purely exponential volume growth. In fact, the volume density of a geodesic sphere of radius $r$ in $S$ has the form \,

  \noindent $\displaystyle{\Theta(r) = c\cdot\sinh^{n-1} r/2\cdot \cosh^{2Q-(n-1)} r/2}$\hspace{0.5mm}, where $c > 0$,\, $n = \dim S$ and the volume entropy $Q$ coincides with the homogeneous dimension of $S$. Refer to  \cite{ADY}.

\vspace{2mm}

Now, let $(M,g)$ be a harmonic Hadamard manifold of the hypergeometric type, namely, the radial part
of the Laplace-Beltrami operator of $M$, $\displaystyle{ -\left(\frac{\partial^2}{\partial r^2} + \sigma(r)\frac{\partial}{\partial r} \right)}$
 is transformed into a second order, differential equation of Gauss hypergeometric type by the transformation $z = - \sinh^2\frac{r}{2}$ of the radial variable $r > 0$. Here $\sigma(r)$ is the mean curvature of the geodesic sphere $\Sigma(x;r)$. See \cite{IKPS, ISGauss} for the precise
definition of the hypergeometric type.  A harmonic Hadamard manifold of a hypergeometric type is of purely exponential volume growth. In fact, the following is verified in \cite{ISGauss}.
\begin{proposition}\label{hypergeometric}\rm
Let $(M,g)$ be a harmonic Hadamard manifold of volume entropy $Q > 0$. Then, $(M,g)$ is of the hypergeometric type if and only if the volume density function
$\Theta(r)$ of the geodesic sphere $\Sigma(x;r)$ has the form\,
$\displaystyle{
\Theta(r) = c_1 \sinh^{2c_2}\frac{r}{2}\cosh^{2c_3}\frac{r}{2}
}$,\, $r>0$  
for some constants $c_i > 0$,\ $i=1,2,3$.
\end{proposition}
 A harmonic Hadamard manifold of the hypergeometric type is of importance from harmonic analysis, since the spherical Fourier-Helgason transform on such a manifold is well defined and admits an inversion formula, as shown in \cite{ISGauss}.
From Propositions \ref{prop1.5} and \ref{hypergeometric}, we obtain the following.
\begin{corollary}
\hspace{2mm}{\rm Let $(M,g)$ be a harmonic Hadamard manifold of the hypergeometric type. Then, each Busemann function of $(M,g)$ has positive definite Hessian over $\nabla b_{\theta}^{\perp}$.
}
\end{corollary}

This  asserts again that every Busemann function on a Damek-Ricci space has positive definite Hessian. 

\vspace{2mm}

A notion of visibility axiom is closely related to the hyperbolicity of
a Hadamard manifold.
\begin{definition}{\rm
A Hadamard manifold $(M,g)$ satisfies {\it the visibility
axiom}, when for any distinct ideal points $\theta$, $\theta'$ of $\partial
M$ there exists a geodesic $\gamma : {\Bbb R}\rightarrow M$ such that $[\gamma] = \theta$ and $[\gamma^-] = \theta'$.
Here $\gamma^-$ is the reversed one of $\gamma$.
}
\end{definition}The visibility axiom is a geometric feature of the ideal boundary $\partial M$. Refer to \cite{BGS, EO}. There are several geometric criterions equivalent with the visibility axiom, one of which is stated as follows. 
 \begin{proposition}\label{critervisibility} {\rm \cite{BGS}
 Let $(M,g)$ be a Hadamard manifold. Then, $(M,g)$ satisfies the visibility axiom if and only if it holds for any $\theta\in \partial M$ \begin{eqnarray}
 \lim_{t\rightarrow\infty} b_{\theta}(\gamma_1) = +\infty,
  \end{eqnarray}with respect to  the  Busemann function $b_{\theta}$ centered at $\theta\in \partial M$. Here $\gamma_1$ is any geodesic parametrized by arc-length with $[\gamma_1] \not=\theta$.
}
 \end{proposition}
In \cite{BGS} Proposition \ref{critervisibility} is stated in terms of horofunction.  However, from Lemma 3.4, \cite{BGS} any Busemann function is a synonym for horofunction with respect to continuous functions. The proposition means that for any geodesic $\sigma$
satisfying $[\sigma]\not=\theta$ and  $[\sigma^{-}]\not=\theta$
 the convex function $b_{\theta}(\sigma(t))$ has necessarily a minimum at some $t_0$. Refer to $\S$1.6.5, \cite{Eberlein}.



The positive definiteness of the Hessian of the Busemann functions is related with the visibility axiom.

\begin{proposition} \label{visibilitycriterion}{\rm
Let $(M,g)$ be a Hadamard manifold.
Then $(M,g)$ satisfies the visibility axiom, if and only if 
$(M,g)$ fulfills the following;  let $\theta\in \partial M$ be an arbitrary ideal boundary point and $\gamma$ be an arbitrary geodesic of $M$ parametrized by arc-length and satisfying $[\gamma] \not= \theta$, and let $t\in {\Bbb R}$ be arbitrary. Then, there exists  a $T\in {\Bbb R}$ ($t < T$) satisfying
\begin{eqnarray}\label{integral}
 \int_{t}^{T} (\nabla d b_{\theta})_{\gamma(s)}(\gamma'(s),\gamma'(s)) ds\hspace{0.5mm} \geq 1.
 \end{eqnarray}}
\end{proposition}

The positive definiteness of the Hessian of all Busemann functions is still not sufficient for $(M,g)$ to satisfy the visibility axiom.  However, we exhibit the following in terms of the least eigenvalue of the Hessian.
\begin{proposition}\label{leasteigenvalue}{\rm
Let $(M,g)$ be a Hadamard manifold. Assume that (i) any
Busemann function $b_{\theta}$ admits positive definite Hessian
$\nabla d b_{\theta}$ over  $\nabla
b_{\theta}^{\perp}$ and (ii) let  $\lambda = \lambda(p)$ be the
least eigenvalue of the Hessian $(\nabla d
b_{\theta})_p$ and let $t_1$ be arbitrary. If 
along an arbitrary geodesic $\gamma_1$
with $[\gamma_1] \not= \pm \theta$ 
\begin{eqnarray}\label{tendinfty}
 \int_{t_1}^t \lambda(\gamma_1(s)) ds  \to +\infty\hspace{2mm}{\rm as}\hspace{2mm}t\rightarrow\infty,
  \end{eqnarray} then, $(M,g)$ fulfills the visibility axiom. }
\end{proposition}

\vspace{2mm}The content of this article is as follows.
In Section 2 we begin with preliminaries of Hadamard manifolds as a basic reference to the sequel. In Section 3, Part I we give definition of a Damek-Ricci space and then present a proof of
Theorem \ref{damekricci-2}. 
 Section 4 is devoted to proving the positive definiteness of the Hessian of an arbitrary Busemann function on a Damek-Ricci space. Appendix I, Part I another proof of Theorem \ref{damekricci-2}
 is provided 
 and in Appendix II, Part I a proof to the existence of  the admissible decomposition of the Lie algebra of a Damek-Ricci space is given. In Section 5, Part II
 we deal with an inductive proof of Theorem \ref{positivedefiniteness}, using Riccati equation. 
  Section 6, Part II  offers  criterions for is the visibility axiom in terms of the Hessian of the Busemann function.

\section{Preliminaries for Hadamard manifolds}

Let $(M,g)$ be a simply connected, $n$--dimensional complete Riemannian manifold of non-positive sectional curvature\hspace{0.5mm}($n \geq 2$). We call such a $(M,g)$ a Caratn-Hadamard manifold, or briefly a Hadamard manifold. By the Cartan-Hadamard theorem $(M,g)$ is diffeomorphic to an $n-$dimensional Euclidean space.

From the non-positive sectional curvature condition the distance function $d : M\times M \rightarrow {\Bbb R}$ is a convex function \cite{Sakai}.  A function $f: M \rightarrow {\Bbb R}$ is convex (strictly convex), if the restriction $f(\gamma(t))$ is convex (strictly convex) for any geodesic $\gamma$ of $M$. If a continous function is $C^2$, then $f$ is convex if and only if the Hessian $\nabla d f$ is positive semi-definite for any point. A $C^2$ function $h$ is strictly convex if $\nabla d h$ is positive definite at any point. Refer  to 1.6.4, \cite{Eberlein} for the convexity.

In what follows, a geodesic on $M$ is assumed parametrized by arc-length.
  Associated with a geodesic $\gamma : {\Bbb R} \rightarrow M$ we define a function, called Busemann function $b_{\gamma} : M \rightarrow {\Bbb R}$ by
  \begin{eqnarray*}
  b_{\gamma}(p) := \lim_{t\rightarrow\infty} \{d(p,\gamma(t)) - t\}. \end{eqnarray*}
  The Busemann function associated with $\gamma$ thus defined, is $C^2$ and convex, with gradient field $\nabla b_{\gamma}$ of unit norm. Refer to \cite{HI} for Busemann function being of $C^2$. Note $b_{\gamma}(\gamma(t)) = - t$, $t\in{\Bbb R}.$  Since $b_{\gamma}$ is convex, the Hessian $\nabla d b_{\gamma}$ is positive semi-definite at any $p\in M$ and satisfies $(\nabla d b_{\gamma})_p(\nabla b_{\gamma}, u) = 0$ for any $u\in T_pM$ and $p\in M$.
  Here, for a $C^2$--function $f$ on $M$, the Hessian of $f$ is defined by $\displaystyle{(\nabla d f)_p(u,v) := u (df({\tilde v})) - df(\nabla_u{\tilde v})
  }$, $u,v\in T_pM$ where ${\tilde v}$ is a smooth vector field around $p$, an extension of $v$.

  In order to develop geometry of a Hadamard manifold $(M,g)$ we define an ideal boundary $\partial M$ for $M$. We denote by ${\mathcal G}_M$ the set of all geodesic rays; $\gamma : [0,\infty) \rightarrow M$. Then, $\gamma$, $\gamma_1\in {\mathcal G}_M$ are  asymptotically equivalent, denoted by $\gamma \sim_a \gamma_1$, if there exists a constant $C > 0$ such that $d(\gamma(t),\gamma_1(t)) \leq C$ for $t \geq 0$. The relation $\sim_a$ is an equivalence relation so we obtain the quotient space ${\mathcal G}_M/\sim_a$ denoted by $\partial M$, and called an ideal boundary of $(M,g)$. We denote by $[\gamma]$ an equivalence class represented by $\gamma$ and usually use a symbol $\theta$ for $[\gamma]$. Let $p\in M$ be an arbitrary point. Then, for each $\theta$ there exists a unique geodesic $\gamma : {\Bbb R} \rightarrow M$ such that $\gamma(0) = p$ and $[\gamma] = \theta$ so that the ideal boundary $\partial M$ is identified with the unit tangent sphere $U_pM$ at $p$, by  the map $U_pM \ni v \mapsto [\gamma_v]\in\partial M$, where $\gamma_v$ is a geodesic defined by $\gamma_v(t) := \exp_p t v$.  It is noticed that $\gamma \sim_a \gamma_1$ if and only if the difference $b_{\gamma_1}(\cdot) - b_{\gamma}(\cdot)$ is a constant function on $M$.
  Now fix $p\in M$ as a reference point such that $U_pM \sim \partial M$ being identified, and choose for each $\theta\in\partial M$ a geodesic $\gamma = \gamma_v$,\, $v\in U_pM$,\, $[\gamma_v] = \theta$ and set $b_{\theta} := b_{\gamma_v}$. We call the $C^2$--function $b_{\theta}$ the Busemann function associated with $\theta$.  Denote by $\displaystyle{ {\mathcal H}_{(\theta,p)} : = \{ x \in M\, \vert\, b_{\theta}(x) = b_{\theta}(p) \}}$ a horosphere centered at $\theta$ passing $p$, a level hypersurface of $b_{\theta}$ which passes a point $p$.
 The minus signed Hessian of $b_{\theta}$ restricted to a tangent space of a horosphere ${\mathcal H}_{(\theta,p)}$ yields the second fundamental form of the horosphere. Indeed, the gradient field $\nu := \nabla b_{\theta}$ gives a unit normal to ${\mathcal H}_{(\theta,p)}$ so that one defines the shape operator ${\mathcal S} : T_x {\mathcal H}_{(\theta,p)} \rightarrow T_x {\mathcal H}_{(\theta,p)}$ at $x\in {\mathcal H}_{(\theta,p)}$ by  ${\mathcal S}\hspace{0.5mm}v := - \nabla_v \nu$ and the second fundamental form $h\,:\,T_x {\mathcal H}_{(\theta,p)}\times T_x {\mathcal H}_{(\theta,p)} \rightarrow {\Bbb R}$ by $h(v,w) := \langle \nabla_v {\tilde w}, \nu \rangle$, respectively (\hspace{0.5mm}${\tilde w}$ is an extension of $w$). Then, one sees easily
   \begin{eqnarray}\label{secondff}
   h(v,w) = \langle {\mathcal S}v, w \rangle = - (\nabla d b_{\theta})_x(v,w).
   \end{eqnarray}

   An eigenvector of $\nabla d b_{\theta}$ corresponding to  eigenvalue $\lambda$ is geometrically a principal direction of a horosphere with principal curvature $- \lambda$.
 Let $x\in {\mathcal H}_{(\theta,p)}$ and $\gamma : {\Bbb R} \rightarrow M$ a geodesic of $M$ satisfying $\gamma(0) = x$ and $\gamma'(0) = -(\nabla b_{\theta})_x$. Then, one observes  that $[\gamma] = \theta$ in $\partial M$ and  a family of horospheres $\{{\mathcal H}_{(\theta, \gamma(t))}\,\vert\, t\in {\Bbb R} \}$ along $\gamma$ foliates $M$.


 Let 
  $\Sigma(q;r)$ be a geodesic sphere centered at a point $q\in M$ and of radius $r$. Then, one admits a family $\{ \Sigma (q;t)\, \vert\, t > 0\}$ of geodesic spheres centered at $q$
  which foliates $M\setminus \{q\}$.  Let $\sigma : {\Bbb R}\rightarrow M$ be a geodesic such that $\sigma(0) = q$, $\sigma'(0)= v \in U_qM$. Then,  one can define a Jacobi vector field $J(t)$ along $\sigma$, perpendicular to $\sigma'(t)$ by $\displaystyle{J(t) := (d \exp_q)_{tv} t u
 }$, $u\in U_qM$, $u\perp v$. See for this \cite{DoCarmo}.  One sees easily $J(0) = 0$ and $J'(0) = u$. Now let $\{ e_j(t)\, \vert\, j = 1,\dots, n \}$ be a set of  parallel vector fields along $\sigma$, orthonormal at any $t$ such that $e_1(0) = v$ and define the Jacobi fields $J_j(t)$, $j = 2,\dots,n$ by $J_j(t) := (d \exp_q)_{tv} t e_j$ which induce an endomorphism $A(t) : T_{\sigma(t)} \Sigma(q;t) \rightarrow T_{\sigma(t)} \Sigma(q;t)$ along $\sigma$, called a Jacobi tensor field, by $A(t) e_j(t) := J_j(t)$.   $\{A(t)\, \vert\, t>0\}$ is a family of endomorphisms satisfying the equation
 $A''(t) + R(t) \circ A(t) = 0$,
where $R(t)$ is a Jacobi operator associated with $\sigma'(t)$; $R(t)v := R(v,\sigma'(t))\sigma'(t)$.  Here the prime $'$ denotes the covariant differentiation along $\sigma$. The endomorphism $A(t)$ satisfies $A(0) = 0$ and $A(t)$ is invertible at any $t \not= 0$. 
   Therefore, for $t > 0$, $A(t)$ induces an endomorphism $A'(t) \circ A^{-1}(t)$, which yields the shape operator of $\Sigma (q; t)$ at $\sigma(t)$.

   Let $\gamma$ be a geodesic of $M$ and $b_{\theta}$ the Busemann function on $M$ associated with $\theta=[\gamma]$. The shape operator ${\mathcal S} = {\mathcal S}(t)$ of the horosphere ${\mathcal H}_{(\theta,\gamma(t))}$ along the geodesic $\gamma$ of $[\gamma]=\theta$ at $\gamma(t)$ is defined similarly by using a stable Jacobi tensor $B(t)$ as ${\mathcal S}(t) := B'(t)\circ B^{-1}(t) $. A Jacobi tensor field $B(t)$ along $\gamma$ is called stable, if $B(t)$ satisfies $\vert B(t)\vert \leq C$, for a constant $C>0$,\, $t\geq 0$. For a precise definition and the construction
   of a stable Jacobi tensor refer to \cite{ISKyushu, Kn2002, Kn-2}. Notice that $\{{\mathcal S}(t)\, \vert\, -\infty <t<\infty\}$ satisfies the Riccati equation;
   \begin{eqnarray}\label{riccati-equation}
   {\mathcal S}'(t) + {\mathcal S}(t)^2 + R(t) = 0.
   \end{eqnarray}which plays a significant role in Sections 7 and 8 of Part II.


\vspace{3mm}
\begin{center}{\sc Part I}
\end{center}

\section{A Damek-Ricci space and rank one  geodesic property}
\subsection{Damek-Ricci spaces}
Let $({\mathfrak n},[\cdot,\cdot]_{\mathfrak n})$ be a two-step nilpotent algebra with an inner product $\langle\cdot,\cdot\rangle_{\mathfrak n}$. Denote by ${\mathfrak z}$ the center of ${\mathfrak n}$ and by ${\mathfrak v}$ the orthogonal complement of ${\mathfrak z}$. Then, $[{\mathfrak v},{\mathfrak v}] \subset {\mathfrak z}$ and each $Z\in {\mathfrak z}$ defines an endomorphism $J_Z: {\mathfrak v}\rightarrow {\mathfrak v}$;
\begin{eqnarray}
 \langle J_Z V, V_1\rangle := \langle Z,[V,V_1]_{\mathfrak n}\rangle_{\mathfrak n}\hspace{2mm}V, V_1\in {\mathfrak v}.
\end{eqnarray}
When $J_Z$ satisfies
\begin{eqnarray}\label{genheisalg}
   (J_Z)^2 = - \vert Z\vert^2 {\rm id}_{\mathfrak v}
\end{eqnarray}for $Z\in {\mathfrak z}$, we call $({\mathfrak n},[\cdot,\cdot]_{\mathfrak n})$
together with $\langle\cdot,\cdot\rangle_{\mathfrak n}$ a {\it generalized Heisenberg algebra}. Remark that polarization of (\ref{genheisalg}) yields
\begin{eqnarray}\label{innerproduct}
 J_Z J_{Z_1} + J_{Z_1} J_Z = - 2 \langle Z, Z_1\rangle_{\mathfrak n}\, {\rm id}_{\mathfrak v}.
\end{eqnarray}The endomorphism $J_Z$ fulfills the basic formulae;
\begin{eqnarray}\label{algebrastr-a}
& &\langle J_Z U, J_{Z_1}V\rangle_{\mathfrak n} + \langle J_{Z_1} U, J_{Z}V\rangle_{\mathfrak n} = 2 \langle U,V\rangle_{\mathfrak n} \langle Z, Z_1 \rangle_{\mathfrak n}\\ 
& &
\label{bracket}
 [J_ZU, V] - [U, J_ZV] = - 2 \langle U,V\rangle_{\mathfrak n} Z \\ 
& & [J_ZU, J_{Z}V] = - \vert Z\vert^2[U,V] - 2 \langle U, J_ZV\rangle_{\mathfrak n} Z
 \end{eqnarray} and then
 \begin{eqnarray}
  [V, J_Z V] = \vert V\vert^2 Z.
 \end{eqnarray}
 Let $({\mathfrak n},[\cdot,\cdot]_{\mathfrak n})$ with $\langle\cdot,\cdot\rangle_{\mathfrak n}$ be a generalized Heisenberg algebra and $N$ be a simply connected Lie group with  Lie algebra $({\mathfrak n},[\cdot,\cdot]_{\mathfrak n})$. Via the exponential map $\exp_{\mathfrak n} : {\mathfrak n} = {\mathfrak v} \oplus {\mathfrak z} \rightarrow N$; $V + Z \mapsto \exp_{\mathfrak n}(V+Z)$\,  the group structure on $N$ is described by the Campbell-Hausdorff formula;
\begin{eqnarray*}
  \exp_{\mathfrak n}(V+Z)\cdot  \exp_{\mathfrak n}(V_1+Z_1) =  \exp_{\mathfrak n}(V+V_1+Z+Z_1+ \frac{1}{2}[V,V_1]).
 \end{eqnarray*}

 Let ${\mathfrak s} = {\mathfrak n} \oplus {\mathfrak a} = {\mathfrak v}\oplus{\mathfrak z}\oplus {\mathfrak a}$ be a one dimensional extension of a generalized Heisenberg algebra ${\mathfrak n}= {\mathfrak v}\oplus{\mathfrak z}$. Here ${\mathfrak a}$ is a one-dimensional vector space with a non-zero vector $A$. Define a bracket structure $[\cdot,\cdot]_{\mathfrak s}$ on ${\mathfrak s}$ by
 \begin{eqnarray*}
  [V+Z+ \ell A, V_1+Z_1+\ell_1 A]_{\mathfrak s} = \left(\frac{\ell}{2} V_1 - \frac{\ell_1}{2} V\right) + \left(\ell Z_1- \ell_1 Z + [V,V_1]_{\mathfrak n}\right) + 0 A
 \end{eqnarray*}for $V, V_1\in {\mathfrak v}$, $Z, Z_1\in {\mathfrak z}$, $\ell, \ell_1\in {\Bbb R}$ and an inner product $\langle\cdot,\cdot\rangle$ on ${\mathfrak s}$ by\begin{eqnarray*}
  \langle V+Z+\ell A, V_1+Z_1+\ell_1 A\rangle = \langle V,V_1\rangle_{\mathfrak n}+ \langle Z, Z_1\rangle_{\mathfrak n} + \ell\ell_1.
 \end{eqnarray*} Consider the product $S := N \times {\Bbb R}$ of $N$ with ${\Bbb R}$.
  A group structure on $S$ is given by, for $\left(\exp_{\mathfrak n}(V+Z), a\right), \left(\exp_{\mathfrak n}(V_1+Z_1), a_1\right)\in S$
 \begin{eqnarray*}
  & &\left(\exp_{\mathfrak n}(V+Z), a\right)\cdot \left(\exp_{\mathfrak n}(V_1+Z_1), a_1\right)\\ \nonumber
  &=& \left(\exp_{\mathfrak n}(V+ e^{a/2} V_1 + Z+ e^a Z_1 + \frac{1}{2} e^{a/2}[V,V_1]), a+a_1 \right).
 \end{eqnarray*}  See \cite{BTV, Re} for this.
 We call the simply connected Lie group $S$ whose Lie algebra is ${\mathfrak s}$ a {\it Damek-Ricci space}. $S$ is equipped with the left invariant Riemannian metric  induced from $\langle\cdot,\cdot\rangle$, denoted by the same symbol $\langle\cdot,\cdot\rangle$. For later use, we adopt another semi-direct product representation of $S$ as
\begin{eqnarray*}\label{upperhalfmodel}
 S \cong N \ltimes {\Bbb R}_+ = \{ (\exp_{\mathfrak n}(U+X), e^{\lambda})\, \vert\, U+X\in {\mathfrak n}, \lambda \in {\Bbb R} \}.
\end{eqnarray*}This representation corresponds to Poincar\'e upper half space model for a real hyperbolic space. For this refer to \cite{ADY} also. Note that the identity element is $e_S =(\exp_{\mathfrak n}(0_{\mathfrak v}+0_{\mathfrak z}), 1)$.  The exponential map \begin{eqnarray*}\exp_{\mathfrak n}\times \exp_{\mathfrak a} : {\mathfrak n}\oplus {\mathfrak a}  &\rightarrow& S = N \times {\Bbb R}; \\ \nonumber
U+X+ s A &\mapsto& (\exp_{\mathfrak n}(U+X), s)
\end{eqnarray*}is a diffeomorphism.

\begin{proposition}\rm \cite{ADY, BTV, Re}
A Damek-Ricci  space is a harmonic, Einstein Hadamard manifold.
\end{proposition}

\subsection{Coordinates}
We introduce global coordinates on a Damek-Ricci space $S$ by the aid of the exponential map.

Let $\{ V_1,\dots, V_m, Y_1,\dots, Y_k, A\}$ be a basis of $(\mathfrak{s}, \langle\cdot,\cdot\rangle)$ such that $\{ V_i\}$ and $\{Y_{\alpha}\}$ are 
bases of ${\mathfrak v}$ and ${\mathfrak z}$, respectively. Here $m = \dim \mathfrak{v}$, $k = \dim \mathfrak{z}$. Denote by $\{{\tilde v}_1, \dots, {\tilde v}_m, {\tilde y}_1, \dots, {\tilde y}_k, {\tilde \lambda}\}$  the corresponding coordinate functions on ${\mathfrak{s}}$.

Via the exponential map
 $\exp_{\mathfrak{n}}\times\exp_{\mathfrak{a}}$ we define coordinates $\{ v_1,\dots, v_m$, $y_1, \dots, y_k, \lambda\}$ on $S$ by
 \begin{eqnarray*}
(v_1,\dots, v_m, y_1,\dots, y_k, \lambda)(p) := ({\tilde v}_1, \dots, {\tilde v}_m, {\tilde y}_1, \dots, {\tilde y}_k, {\tilde \lambda})(\exp_{\mathfrak{n}}\times\exp_{\mathfrak{a}})^{-1}(p).
\end{eqnarray*}
For a later convenience we set $a = e^{\lambda}$ with respect to which $\{v_i, y_{\alpha}, a\}$ give coordinates of $S$.
The basis $\{ V_1,\dots, V_m, Y_1,\dots Y_k, A\}$ of ${\mathfrak{s}}$ induces left invariant vector fields on $S$ denoted by the same symbol, as $V_i(p) := (L_p)_{\ast e} V_i$ for instance.

Then we get the following.
\begin{lemma}\label{coordinatesons}\rm \cite{BTV}\hspace{2mm}At a point $p =(U,X,e^r)\in S$ of coordinates $\{v_i, y_{\alpha}, a = e^{\lambda}\}$
\begin{eqnarray}\label{localfields}
V_i &=&  \sqrt{a}\, \frac{\partial}{\partial v_i} - \frac{1}{2}\, \sqrt{a}\, \sum_{j,\alpha}\, \langle [V_i,V_j], Y_{\alpha}\rangle \, v_j\ \frac{\partial}{\partial y_{\alpha}}, \\ \nonumber
Y_{\alpha} &=& a\, \frac{\partial}{\partial y_{\alpha}}, \hspace{2mm}
A = \frac{\partial}{\partial \lambda} = a \frac{\partial}{\partial a}.
\end{eqnarray}
Conversely
\begin{eqnarray}
\frac{\partial}{\partial v_i} &=& \frac{1}{\sqrt{a}}\, V_i + \frac{1}{2a}\, \sum_{j,\alpha}\, \langle [V_i,V_j], Y_{\alpha}\rangle\, v_j\ Y_{\alpha},
\\ \nonumber
\frac{\partial}{\partial y_{\alpha}} &=& \frac{1}{a}\, Y_{\alpha}, \hspace{2mm}
\frac{\partial}{\partial \lambda} = a \frac{\partial}{\partial a} = A.
\end{eqnarray}
\end{lemma}

\subsection{Levi-Civita connection and Riemannian curvature tensor}

Let $\nabla$ be the Levi-Civita connection of a Damek-Ricci space $S$.
The connection $\nabla$ is determined by its values on the left invariant vector fields in $\mathfrak{s}$.

\begin{lemma}\label{connection}\rm \cite{BTV} For left invariant vector fields $V+Y+sA$ and $U+X+ r A$
\begin{eqnarray*}
\nabla_{(V+Y+sA)}\, (U+X+ r A) &=& - \frac{1}{2}\, J_X V - \frac{1}{2}\, J_Y U - \frac{1}{2}\, r V \\ \nonumber
&-& \frac{1}{2} [U,V] - r Y + \frac{1}{2} \langle U,V\rangle A +  \langle X,Y\rangle A.
\end{eqnarray*} Namely
\begin{eqnarray}\label{levicivita}
\nabla_{V}U &=& -\frac{1}{2} [U,V] + \frac{1}{2} \langle U,V\rangle A,\\ \nonumber
\nabla_{V} X &=&  \nabla_{X} V = - \frac{1}{2} J_{X}V,
\hspace{2mm} \nabla_{V} A =  -\frac{1}{2}V, \\ \nonumber
 \nabla_{Y}X &=& \nabla_XY =\langle X,Y\rangle A, \hspace{2mm}  \nabla_{Y} A = - Y, \\ \nonumber
 \nabla_A V &=& 0,\hspace{2mm}\nabla_A Y = 0,\hspace{2mm}\nabla_A A = 0.
  \end{eqnarray}
\end{lemma}

  The Riemannian curvature tensor $R$ is defined by
  \begin{eqnarray*}R(u,v)w := \left(\nabla_{\tilde u}\nabla_{\tilde v}{\tilde w}-\nabla_{\tilde v}\nabla_{\tilde u}{\tilde w}- \nabla_{[{\tilde u},{\tilde v}]}{\tilde w}\right)(p),\hspace{2mm}u,v,w\in T_pM,
  \end{eqnarray*} where ${\tilde u}$, ${\tilde v}$ and ${\tilde w}$ are extensions of  $u,v,w$.
   For the formula of $R$ in terms of the algebra ${\mathfrak s}$ refer to pp. \hspace{-1mm}84,\, 85, \cite{BTV}.  The Jacobi operator $R_u$, $u\in T_xS$, $\vert u\vert=1$ is an endomorphism,  defined by $R_u : u^{\perp} \rightarrow u^{\perp}$, $R_u(v) := R(v,u)u$.  Since $S$ is of non-positive sectional curvature, $R_u$ is negative semi-definite for any $u$.

    \subsection{The endomorphism $K_{V,Y}$}\label{endoj2}
 Let $V\in {\mathfrak v}$ and $Y\in{\mathfrak z}$ be non-zero vectors with the normalized unit vectors  $\displaystyle{{\hat V}=\frac{V}{\vert V\vert}}$, $\displaystyle{{\hat Y}=\frac{Y}{\vert Y\vert}}$. Denote by $Y^{\perp}$ a linear subspace in ${\mathfrak z}$;
 \begin{eqnarray}Y^{\perp}:=\{Z\in {\mathfrak z}\,\vert\, \langle Z,Y\rangle = 0\}.
 \end{eqnarray}
  \begin{definition}\cite{BTV}\label{defk}\,  \rm An endomorphism $K = K_{V,Y} : Y^{\perp} \rightarrow Y^{\perp}$ is defined by
\begin{eqnarray}
   Z \mapsto\, K(Z) := [{\hat V}, J_ZJ_{\hat Y}{\hat V}].
\end{eqnarray}
\end{definition}
\begin{proposition}\cite{BTV}\label{k1}\, \rm
 The endomorphism $K$ is skew-adjoint with respect to the inner product $\langle\cdot,\cdot\rangle$.
Thus $K^2$ is self-adjoint. 
\end{proposition}

 We decompose $Y^{\perp}$ orthogonally into a direct sum
 \begin{eqnarray}\label{eigenspacesofell}
 Y^{\perp} = L_0 \oplus \cdots\oplus L_{\ell}, \hspace{2mm} 
  L_j := {\rm Ker}\, (K^2 - \mu_j {\rm id}_{Y^{\perp}}),\, j = 0,\cdots,\ell
 \end{eqnarray}with the distinct eigenvalues
 \begin{eqnarray*}
 0\geq\mu_0 \geq \mu_1 > \cdots > \mu_{\ell} \geq -1
 \end{eqnarray*}of $K^2$. It is observed since $K^2\circ K = K\circ K^2$ that
 \begin{eqnarray*}
     K(L_j) \subset L_j,\hspace{2mm}j=0,\dots,\ell
 \end{eqnarray*}and from \cite{BTV} that when $\mu_0=0$
 \begin{eqnarray*}\label{charactofl}
 X\in L_0\, \Leftrightarrow \, J_XJ_Y V \in {\rm Ker\, ad}\ (V)
 \end{eqnarray*}
 and when $\mu_{\ell}= -1$
 \begin{eqnarray*}
 X\in L_{\ell}\, \Leftrightarrow \, J_XJ_Y V \in {\rm Ker\, ad}\ (V)^{\perp}.
 \end{eqnarray*}
 Define a subspace of ${\mathfrak v}\oplus{\mathfrak z}$
 \begin{eqnarray}
   {\mathfrak q}_j := {\rm Span}\, \{L_j, J_{L_j}V, J_{L_j}J_Y V\}, \, j= 0,\cdots, \ell,\end{eqnarray}
   when $\mu_{\ell} \not=-1$ and
   \begin{eqnarray}
   {\mathfrak q}_{\ell} := {\rm Span}\, \{L_{\ell}, J_{L_{\ell}}V\},\, {\rm when}\, \mu_{\ell} = -1.
 \end{eqnarray}Then, we define as an orthogonal direct sum
 \begin{eqnarray}\label{decompoq}
 {\mathfrak q} :=  {\mathfrak q}_0\oplus \dots \oplus  {\mathfrak q}_{\ell}
 \end{eqnarray}
 which has the form \begin{eqnarray}\label{spaceq}{\mathfrak q} = {\rm Span}\, \{L_{Y^{\perp}}V, L_{Y^{\perp}}L_YV, Y^{\perp}\}.
 \end{eqnarray} Refer for this to p.97,\ \cite{BTV}.

 We define a linear subspace  $(Y^{\perp})_J$ in $Y^{\perp}$ as 
 \begin{eqnarray}\label{J2condition-1}
 (Y^{\perp})_J := \{ Z\in Y^{\perp}\, \vert\, \exists Z'\in {\mathfrak z}\hspace{2mm}{\rm such \, that}\, J_ZJ_YV = J_{Z'}V \}
 \end{eqnarray}and its orthogonal complement $(Y^{\perp})_J^{\perp}$  so that
 \begin{eqnarray}\label{decompy}
 Y^{\perp} = (Y^{\perp})_J \oplus (Y^{\perp})_J^{\perp}.
 \end{eqnarray}  We remark that in (\ref{J2condition-1}) $Z'$ belongs to $(Y^{\perp})_J$. Actually from (\ref{J2condition-1}) the $Z'$ must satisfy $J_{Z'}J_YV = J_{(-\vert Y\vert^2 Z)}J_Y V$, since from $J_Y(J_ZJ_YV) = J_Y( J_{Z'}V)$,\,  one has from (\ref{algebrastr-a}) $- J_Z (J_Y J_Y V) = - J_{Z'}J_YV$ and hence $- J_Z(- \vert Y\vert^2)V) = - J_{Z'}J_YV$.
   \begin{note}\rm If ${\mathfrak s}$ satisfies the $J^2$-condition, then, one has exactly $Y^{\perp}= (Y^{\perp})_J$.
   \end{note}
       \begin{lemma}\label{k}\rm
        $K((Y^{\perp})_J) \subset (Y^{\perp})_J$ and $K((Y^{\perp})_J^{\perp}) \subset (Y^{\perp})_J^{\perp}$.
        \end{lemma}

                \begin{lemma}\label{ksquare}\rm
    (i)\, The endomorphism $K^2 =K^2_{V,Y}$ is equal to $- id_{(Y^{\perp})_J}$, when restricted to $(Y^{\perp})_J$.

    (ii)\, For $Z\in Y^{\perp}$\, it holds $K^2(Z) = - Z$ if and only if $Z\in (Y^{\perp})_J$.

       \vspace{2mm}
     (iii)\, Any eigenvalue $\mu$ of $K^2\vert_{(Y^{\perp})_J^{\perp}}$ restricted to $(Y^{\perp})_J^{\perp}$ satisfies\, $-1 < \mu \leq 0$.
     \end{lemma}

 Therefore, we have from (\ref{charactofl})
 \begin{eqnarray*}
  (Y^{\perp})_J = L_{\ell},  \hspace{3mm}
  (Y^{\perp})_J^{\perp} = L_0 \oplus \dots \oplus L_{\ell-1},
 \end{eqnarray*}provided $\mu_{\ell} = -1$, since $J_{\mathfrak z}V = {\rm Ker\, ad}(V)^{\perp}$.

 \begin{lemma}\label{orthogonal-2}\rm Assume $\mu_{\ell} = -1$. Let $j, j' = 0,\dots, \ell-1$. Then
 \begin{eqnarray*}
{\rm (i)}\hspace{2mm}J_{L_j}V &\perp& J_{L_{j'}}V,\hspace{2mm}j\not= j', \\ \nonumber
{\rm (ii)}\hspace{2mm}J_{L_j}J_YV &\perp& J_{L_{j'}}J_YV,\hspace{2mm}j\not= j',\\ \nonumber
{\rm (iii)}\hspace{2mm}J_{L_j}V &\perp& J_{L_{j'}} J_Y V,\hspace{2mm}j\not= j',\\ \nonumber
{\rm (iv)}\hspace{2mm}J_{L_j}V &\cap& J_{L_{j}} J_Y V = \{0\}.
 \end{eqnarray*}
 \end{lemma}
 \begin{remark}\label{notorthogonal}\rm  It does not hold
 \begin{eqnarray*}J_{L_j} V \perp J_{L_{j}}J_Y V, \hspace{2mm}j= 1,\dots, \ell-1,
 \end{eqnarray*}except for  $j = 0$ with $\mu_0 = 0$ as
 \begin{eqnarray}\label{perpendicular}
 J_{L_0}V &\perp& J_{L_{0}} J_Y V.
 \end{eqnarray}
In fact, for  vectors $W_1 := J_{Z_1}{\hat V}\in J_{L_j} V$, $W_2:= J_{Z_2}J_{\hat Y}{\hat V} \in  J_{L_{j}}J_Y V$, $0< j\leq \ell-1$, where $Z_1$ and $\displaystyle{Z_2 =
\frac{1}{\sqrt{-\mu_j}} K(Z_1)}$  are unit vectors in $L_j$. We have then
\begin{eqnarray}\label{notorthogonal} 
\langle W_1,W_2\rangle = - \sqrt{-\mu_j}.
\end{eqnarray}Actually the inner product is
\begin{eqnarray*}
\langle W_1,W_2\rangle = 
 \frac{1}{\sqrt{-\mu_j}}\langle [{\hat V}, J_{KZ_1}J_{\hat Y}{\hat V}], Z_1\rangle
= \frac{1}{\sqrt{-\mu_j}}\langle K(KZ_1), Z_1\rangle. 
\end{eqnarray*}
 \end{remark}

\subsection{Geodesics of a Damek-Ricci space}

In what follows, we use customarily symbols $U, V$ and $W$ for elements of ${\mathfrak v}$ and $X,Y$ and $Z$ for elements of ${\mathfrak z}$.


\begin{lemma}\label{geod} \rm \cite{BTV, IStokyo, Re}\hspace{2mm}Let $V+Y+ s A \in {\mathfrak s}$ be a unit vector and $\gamma$ be a geodesic in $S$ of $\gamma(0) = e_S$, $\gamma'(0)=V+Y+sA$. Then $\gamma(t)$ can be represented by
\begin{eqnarray}
\gamma(t) = \Big(\exp_{\mathfrak n}\big(\frac{2\theta(t)(1-s\theta(t))}{\chi(t)}V + \frac{2\theta^2(t)}{\chi(t)} J_YV + \frac{2\theta(t)}{\chi(t)} Y\big), \big(\frac{1-\theta^2(t)}{\chi(t)}\big) A\Big).
 \end{eqnarray}
 Here
 \begin{eqnarray}
  \theta(t) := \tanh \frac{t}{2},\, \chi(t) := (1-s\theta(t))^2 + \vert Y\vert^2 \theta^2(t).
 \end{eqnarray}
 \end{lemma}

 \vspace{2mm}
\subsection{Rank one  geodesic property}\label{kimproof}

Let $\gamma$ be a geodesic in a Damek-Ricci space $S$ of $\gamma(0) = e_S$ and of $\gamma'(0) = V + Y + s A$, a unit tangent vector in $T_{e_S}S$. Let $u=u(t)$ be a parallel Jacobi field along $\gamma$, that is,
\begin{eqnarray}
u''(t) + R(u(t),\gamma'(t))\gamma'(t) = 0\hspace{2mm} {\rm and}\hspace{2mm} u'(t) = 0,
\end{eqnarray}where $R(\cdot,\cdot)$ is the Riemannian curvature tensor of $S$. Therefore it holds
with respect to the Jacobi operator $R_{\gamma'(t)}$ defined by
$R_{\gamma'(t)} : \gamma'(t)^{\perp} \rightarrow \gamma'(t)^{\perp}\,;\, v \mapsto R(v,\gamma'(t))\gamma'(t)$ it holds $R_{\gamma'(t)} u(t) = 0$ for all $t$.   In particular, $u(t)$ is an element of ${\rm Ker}(R_{\gamma'(t)})$ in $\gamma'(t)^{\perp}\subset T_{\gamma(t)}S$.

Now, using several propositions, we will prove that if there exists a vector
field perpendicular to $\gamma'$ which is an element of ${\rm
Ker}(R_{\gamma'(t)})$ at any $t$, then there is a contradiction.

\vspace{2mm}
Eigenvalue zero of  $R_{V+Y+sA}$ and the corresponding eigenspaces are presented as

\begin{proposition}\label{zeroeigenvector} \rm (Theorem 1, p. 96-98, \cite{BTV}) we have

\vspace{2mm}
Case (i):\, $V = 0$ or $Y = 0$.

\vspace{1mm}
The eigenspace of $R_{V+Y+sA}$ corresponding to zero is ${\Bbb R}(V+Y+sA)$.

\vspace{2mm}
Case (ii):\, $V \not= 0$ and $Y \not= 0$.

\vspace{1mm}In this case we can decompose from (\ref{eigenspacesofell})
the subspace $Y^{\perp}$ orthogonally into $Y^{\perp}
= L_0\oplus \cdots \oplus L_{\ell}$ in terms of the eigenspaces $L_j$ of $K^2$ with eigenvalues

$\displaystyle{ 
0 \geq \mu_0 > \mu_1 > \cdots > \mu_{\ell} \geq -1}$. Then, each space  ${\mathfrak q}_j$ of (\ref{decompoq})
 is invariant under $R_{V+Y+sA}$. Refer  for this to \cite{BTV}.

We have further a decomposition of ${\mathfrak s}$ as ${\mathfrak s} = {\mathfrak s}_4 \oplus {\mathfrak p} \oplus {\mathfrak q}$ orthogonally for some subspaces ${\mathfrak s}_4$, ${\mathfrak p}$ which are invariant under  $R_{V+Y+sA}$. 
The eigenspace of $(R_{V+Y+sA})_{\vert({\mathfrak s}_4\oplus{\mathfrak p})}$ corresponding to eigenvalue zero is  ${\Bbb R}(V+Y+s A)$.

If $j = \ell$ and $\mu_{\ell} = -1$, then the eigenvalues of $(R_{V+Y+sA})_{\vert{\mathfrak q}_{\ell}}$ are $- \frac{1}{4}$ and $-1$.

Otherwise, if not, $(R_{V+Y+sA})_{\vert{\mathfrak q}_{j}}$ has two
or three distinct eigenvalues $\kappa_1$, $\kappa_2$, $\kappa_3$ satisfying
\begin{eqnarray}\label{k1k2k3}
 - 1 < \kappa_1 \leq - \frac{3}{4} \leq \kappa_2 < - \frac{1}{4} < \kappa_3 \leq 0.
\end{eqnarray}They are the solutions of the equation
\begin{eqnarray}\label{equation}
 (\kappa +1)(\kappa + \frac{1}{4})^2 = \frac{27}{64}\ \vert V \vert^4\ \vert Y \vert^2 (1 + \mu_j).
\end{eqnarray}
\end{proposition}

Notice for detail of the decomposition ${\mathfrak s} = {\mathfrak s}_4 \oplus {\mathfrak p} \oplus {\mathfrak q}$ of ${\mathfrak s}$ refer to Proposition \ref{decompositionofs}. A proof is given in {\sc Appendix II}.

\begin{proposition}\label{notparallel}\rm
Let $V + Y+ sA$ be a unit vector in ${\mathfrak s}$ of the
Damek-Ricci space $S$. If there exists an eigenvector of  $R_{V+Y+sA}$ corresponding to eigenvalue zero which is not proportional to the vector $V + Y+ sA$, then $V$ and $Y$ must satisfy
\begin{eqnarray}
  \vert V \vert^2 = \frac{2}{3}\hspace{2mm}{\rm and}\hspace{2mm} \vert Y\vert^2 = \frac{1}{3}.
\end{eqnarray}

\end{proposition}
\noindent
\begin{proof}\rm Suppose that there exists an eigenvector of $R_{V+Y+sA}$ corresponding to zero, not proportional to $V+Y+sA$. By Proposition \ref{zeroeigenvector} there exists some $j \in \{1,\cdots,\ell\}$ such that $(R_{V+Y+sA})_{\vert{\mathfrak q}_j}$ has two or three distinct eigenvalues $\kappa_1$, $\kappa_2$, $\kappa_3$ satisfying the conditions in (\ref{k1k2k3}) and (\ref{equation}). We have here from (\ref{k1k2k3}) the eigenvalue $\kappa_3 = 0$, since others $\kappa_1$, $\kappa_2$ are negative. Thus, we have, substituting $\kappa = 0$ into (\ref{equation})
\begin{eqnarray*}
  \frac{1}{16} = \frac{27}{64}\, \vert V\vert^4\ \vert Y\vert^2 (1 + \mu_j).
\end{eqnarray*}
Since $\vert V \vert^2 + \vert Y\vert^2 + s^2 = 1$,
\begin{eqnarray*}
  \vert V\vert^4\ \vert Y\vert^2 \leq (1 - \vert Y\vert^2)^2 \vert Y \vert^2 \leq \frac{4}{27}.
\end{eqnarray*}The last inequality comes from\, ${\rm max}_{0\leq t\leq 1}\, (1-t)^2 t = 4/27$\, and the equality $\vert V\vert^4\vert Y\vert^2 = 4/27$ attains  if and only if
\begin{eqnarray*}
\vert V \vert^2 = \frac{2}{3}, \,  \vert Y \vert^2 = \frac{1}{3}\hspace{2mm}
{\rm and}\hspace{2mm}s = 0.
\end{eqnarray*}
Since $0 \leq 1 + \mu_j \leq 1$,
\begin{eqnarray*}
\frac{27}{64}\, \vert V \vert^4 \vert Y \vert^2 (1+ \mu_j) \leq \frac{27}{64}\, \vert V \vert^4 \vert Y \vert^2 \leq \frac{1}{16}
\end{eqnarray*}and hence
\begin{eqnarray*}\frac{1}{16} =
\vert V\vert^4\ \vert Y\vert^2(1+\mu_j)\, \Leftrightarrow\, \vert V \vert^2 = \frac{2}{3}, \,  \vert Y \vert^2 = \frac{1}{3},\, s = 0\hspace{2mm} {\rm and}\hspace{2mm} \mu_0 = 0,\, j=0.
\end{eqnarray*}
In particular,
\begin{eqnarray*}
  \vert V \vert^2 = \frac{2}{3},\,  \vert Y \vert^2 = \frac{1}{3}.
\end{eqnarray*}
\end{proof}
\begin{proposition}\label{geodesic}{\rm(Theorem 2, p. 94, \cite{BTV})
Let $V + Y + s A$ be a unit vector of ${\mathfrak s}$ and $\gamma$ the geodesic in $S$ of $\gamma(0) = e_S$, $\gamma'(0) = V + Y+ s A$. Decompose $\gamma'(t)$ into
\begin{eqnarray}\label{geodesicdecompo}
  \gamma'(t) = V(t) + Y(t) + A(t).
\end{eqnarray}Then $\vert V(t) \vert^2 = \vert V \vert^2 h(t)$ and $\vert Y(t) \vert^2 = \vert Y \vert^2 h(t)$ where
\begin{eqnarray}\label{hfunction}
h(t) = \frac{1-\theta^2(t)}{\chi(t)}.  
\end{eqnarray}
 }
 \end{proposition}

\begin{theorem} {\rm
There exists no parallel Jacobi vector field which is orthogonal to $\gamma'$ along a geodesic $\gamma$ in a Damek-Ricci space $S$.
}
\end{theorem}
\noindent
\begin{proof}\rm Suppose that there exists a vector field $u(t)$ which is parallel, Jacobi and orthogonal to $\gamma'$ along a geodesic $\gamma$.

Decompose $\gamma'(t)$ into (\ref{geodesicdecompo})
 for each $t$.  Fix $t_0$. Then $\gamma'(t_0)$ is a unit vector and $R_{(V(t)+Y(t)+A(t))} u(t) = 0$ at $t_0$. Thus by Proposition \ref{notparallel}, $\vert V(t_0)\vert^2 = \frac{2}{3}$ and $\vert Y(t_0)\vert^2 = \frac{1}{3}$ and  moreover $A(t_0) = 0$.  Therefore,  $\vert V(t)\vert^2 = \frac{2}{3}$ and $\vert Y(t)\vert^2 = \frac{1}{3}$ for any $t$.

By Proposition \ref{geodesic} $h(t) \equiv 1$ for any $t$ and then, from
(\ref{hfunction}) $\theta(t) = 0$ for any $t$. However, $\theta(t) =
0$ if and only if $t = 0$. So, this is a contradiction.
\end{proof}

\section{Hessian of Busemann function}

\subsection{The distance function of a Damek-Ricci space}

\vspace{2mm}Let $p\in S$ be a point and represent $p$ in a view of (\ref{upperhalfmodel}) as $p = (\exp_{\mathfrak n}(U+X), e^r)$.
From Lemma \ref{geod} we obtain
\begin{lemma}\label{distance-1}\rm  \cite{IStokyo}\, The distance of $p = (\exp_{\mathfrak n}(U+X), e^r)$ from $e_S$ is given by
\begin{eqnarray}
d(e_S,p) = \log \frac{\lambda(p) -2 + \sqrt{\lambda^2(p) - 4 \lambda(p)}}{2},
\end{eqnarray}where, 
$a = e^r$ and $\lambda : S \rightarrow {\Bbb R}$ is a function of $S$ given by
\begin{eqnarray}
\lambda(p) := \frac{1}{a}\left\{(1+a+\frac{1}{4}\vert U\vert^2)^2 + \vert X \vert^2 \right\}.
\end{eqnarray}
\end{lemma}

Now, let $\gamma$ be a geodesic of $\gamma(0) = e_S$, $\gamma'(0) = V+Y+ sA$.

\begin{lemma}\rm \label{distance-2} \cite{IStokyo}
 \hspace{2mm}The distance $d(p,\gamma(t))$ 
  is described by
 \begin{eqnarray}\label{distance}
 d(p,\gamma(t)) &=& d(e_S,p^{-1}\, \gamma(t))\\ \nonumber
 &=& \log \frac{\lambda - 2 + \sqrt{\lambda^2 - 4 \lambda}}{2},
 \end{eqnarray}
 where \begin{eqnarray}
 \lambda = \lambda(p^{-1}\ \gamma(t))= \frac{1}{a(t)}\big\{ \big(1+ a(t)+ \frac{1}{4}\vert {\hat U}(t)\vert^2\big)^2 + \vert{\hat X}(t)\vert^2\big\},
 \end{eqnarray}
 \begin{eqnarray}
   {\hat U}(t) &=& - e^{-r/2} U + e^{-r/2}{\tilde V}(t), \\ \nonumber
   {\hat X}(t) &=& - e^{-r} X+ e^{-r} {\tilde Y}(t),  \hspace{2mm}
   a(t) =  e^{-r} h(t).
 \end{eqnarray}
 \end{lemma}

If we write $\gamma(t) = ({\tilde V}(t), {\tilde Y}(t), h(t))$,
\begin{eqnarray*}
{\tilde V}(t) =
\frac{2\theta(t)(1-s\theta(t))}{\chi(t)}V + \frac{2\theta^2(t)}{\chi(t)} J_YV, \hspace{2mm}
\label{eqn-3}
{\tilde Y}(t) = \frac{2\theta(t)}{\chi(t)} Y. 
\end{eqnarray*}
Then
 \begin{eqnarray*}
 p^{-1}\ \gamma(t) &=& (- e^{-r/2} U, - e^{-r} X, e^{-r})({\tilde V}(t),{\tilde Y}(t), h(t)) \\ \nonumber
 &=& \big( - e^{-r/2} U + e^{-r/2}{\tilde V}(t), - e^{-r} X+ e^{-r} {\tilde Y}(t),
 e^{-r} h(t)
 \big)\\ \nonumber
 &=:& ({\hat U}(t),{\hat X}(t), a(t)).
 \end{eqnarray*}
Therefore, Lemma \ref{distance-2} is obtained from Lemma \ref{distance-1}.

 \subsection{Busemann function on a Damek-Ricci space}
 Let $b_{\gamma}$ be the Buseman function on $S$ associated with a geodesic $\gamma$;
 \begin{eqnarray*}
 b_{\gamma}(p) := \lim_{t\rightarrow \infty} \{d(p,\gamma(t)) - t\} = \lim \{d(e_S, p^{-1} \gamma(t))-t\}, \hspace{2mm}p\in S.
 \end{eqnarray*}
 Then
  \begin{proposition}\label{busemann}\rm \cite{IStokyo}\hspace{2mm}Let $p = (\exp_{\mathfrak n}(U+X), e^r)\in S$ and $\gamma$ be a geodesic of $\gamma(0) = e_S$, $\gamma'(0) = V+Y+ s A$, a unit vector. Then the Busemann function $b_{\gamma}$ on $S$ associated with $\gamma$ is described by
 \begin{eqnarray}\label{busefuncexpression}
 b_{\gamma}(p) &=& -\log \Big(\frac{e^r\cdot \big((1+\frac{1}{4}\vert v\vert^2)^2 + \vert y\vert^2\big)}{\big(e^r+\frac{1}{4}\vert U - v\vert^2 \big)^2 + \vert X-y + \frac{1}{2}[U,v]\vert^2}  \Big),
   \hspace{2mm}s\not= 1, \\ \nonumber
  &=& - \log e^r,\hspace{2mm}s = 1.\end{eqnarray}
 Here $(v,y)\in {\mathfrak v}\times {\mathfrak z}$
 is defined by 
 \begin{eqnarray}\label{vy}
 v
   := \frac{2}{\chi_{\infty}}\{(1-s)V + J_YV\}, \hspace{2mm}
 y:= \frac{2}{\chi_{\infty}} Y
 \end{eqnarray}
 ($\chi_{\infty} := \lim_{t\rightarrow\infty} \chi(t) = (1-s)^2 + \vert Y\vert^2$). $(v,y)$ defines an ideal boundary point representing $[\gamma]\in \partial S$, relative to a cone topology of $S\cup \partial S$.

 \end{proposition}

  \subsection{Hessian of Buseman function}

 The aim of this section and the subsequent sections is to investigate the positive definiteness of the Hessian of the Busemann function over $\nabla b_{\gamma}^{\perp}$
 at any point of $S$.


 Let $b_{\gamma}$ be the Busemann function associated with a geodesic $\gamma$ of $\gamma(0) = e_S$, $\gamma'(0) \in T_eS$, a unit vector. 
 Notice that the gradient $\nabla b_{\gamma}$ gives a null vector of the Hessian. 
  Let $p\in S$ be an arbitrary point and restrict  the Hessian  to $\nabla b_{\gamma,\,p}^{\perp}$, the subspace of $T_pS$ orthogonal to $\nabla b_{\gamma,\,p}$;
 \begin{eqnarray*}
 \nabla d b_{\gamma,\, p} : \nabla b_{\gamma,\,p}^{\perp} \times  \nabla b_{\gamma,\,p}^{\perp} \rightarrow {\Bbb R}.
 \end{eqnarray*}
  Since the metric of $S$ is left invariant, any left translation is an isometry. Therefore, as Proposition \ref{equivariance} shows, in order to assert the positive definiteness of the Hessian $\nabla d b_{\gamma}$ at an arbitrary point it suffices to investigate the positive definiteness of the Hessian 
 at $e_S$.  Just like the equivariance of the Jacobi operator $R_u$ with respect to an isometry, the Hessian $\nabla d b_{\gamma}$ satisfies the following for the action of left translation.

 \begin{proposition}\label{equivariance}\rm  Let $\gamma$ be a geodesic of $S$ and $y$ an arbitrary point of $S$. Then, for any $x\in S$
 \begin{eqnarray*}
 (\nabla d b_{\gamma})_y(u,v) =  (\nabla d b_{L_x(\gamma)})_{L_x(y)}(L_{x\ast}u,L_{x\ast}v),\hspace{2mm}u,v\in T_yS.
 \end{eqnarray*}Here $L_x(\gamma)$ is a geodesic defined by $L_x(\gamma)(t) := L_x(\gamma(t))$.
 \end{proposition}

 A proof is derived by $\displaystyle{(\nabla d b_{\gamma})_y(u,v) = \langle \nabla_u \nabla b_{\gamma}, v\rangle_y}$.

 We may assume in Proposition \ref{equivariance} that $\gamma$ satisfies $\gamma(0) = y$, since $\nabla d b_{\gamma} = \nabla d b_{\gamma_1}$ for geodesics $\gamma$ and $\gamma_1$ such that $\gamma_1 \sim \gamma$, so that we choose $x= y^{-1}$. Therefore,  $L_x(\gamma)(0) = L_xy =e_S$ for the geodesic $L_x(\gamma)$ appeared on the right hand side.


 \subsection{The formula of the Hessian.}

We now investigate some of the properties of the Hessian.

 \begin{theorem}\label{hessianatidentity}\rm Let $V+Y+sA\in {\mathfrak s}$ be an arbitrary unit vector and $\gamma$ be the geodesic of $\gamma(0) = e_S$, $\gamma'(0) = V+Y+sA$. Then the Hessian $\nabla d b_{\theta}$ is positive definite over $\gamma'(0)^{\perp}$, where $\theta =[\gamma]\in \partial S$.

 \end{theorem}

 Note $\nabla d b_{\theta} \equiv \nabla d b_{\gamma}$ on $S$.  Theorem \ref{hessianatidentity} therefore implies
\begin{proposition}\label{uniformbound}\rm The Hessian $\nabla d b_{\theta}$ is uniformly bounded from below by a positive constant. Namely, there exists $C_0>0$ such that  $\nabla d b_{\theta}(X,X) \geq C_0$,\hspace{2mm}for any $X\in T_p {\mathcal H}_{(\theta, p)}$ of $\vert X\vert = 1$ and any $p\in S$, $\theta\in \partial S${\red{.}}
\end{proposition}
In fact, let $\Theta : = \{ (V+Y+sA, V_1+Y_1+s_1A) \in {\mathfrak s}\times{\mathfrak s}\, \vert\, \vert V+Y+sA\vert = \vert V_1+Y_1+s_1A\vert= 1,\, \langle V+Y+sA, V_1+Y_1+s_1A\rangle = 0\}$ and define a map $\Psi : \Theta \rightarrow {\Bbb R}$ by $\Psi(V+Y+sA,V_1+Y_1+s_1A ) = (\nabla d b_{\gamma})_{e_S}(V_1+Y_1+s_1A, V_1+Y_1+s_1A)$,
 where $\gamma$ is the geodesic of $\gamma(0)=e_S$, $\gamma'(0)= V+Y+sA$. The map $\Psi$ is a continuous function on a compact set $\Theta$ so that from Theorem
 \ref{hessianatidentity} there exists a constant $C_0 > 0$ such that $\Psi(V+Y+sA,V_1+Y_1+s_1A )\geq C_0$, namely
 $\nabla d b_{\gamma}(V_1+Y_1+s_1A, V_1+Y_1+s_1A) \geq C_0 \vert V_1+Y_1+s_1A\vert^2$, $V_1+Y_1+s_1A\in T_{e_S}{\mathcal H}_{(\theta,\gamma(0))}$.

 \vspace{2mm}In what follows we will prove Theorem \ref{hessianatidentity}.

 We mainly deal with the case $s \not=1$. In case $s=1$ it is directly shown as
 \begin{proposition}\rm Let $\gamma$ be the geodesic of $\gamma(0) = e_S$, $\gamma'(0) = A$. Then, $\nabla d b_{\gamma}$ is positive definite over $\gamma'(0)^{\perp}={\mathfrak v}\oplus {\mathfrak z}$. More precisely, ${\mathfrak v}$ is the eigenspace of $\nabla d b_{\gamma}$ corresponding to eigenvalue $\frac{1}{2}$ and ${\mathfrak z}$ is the eigenspace correspnding to eigenvalue $1$.

 \end{proposition}In fact, by the aid of the formula in Lemma \ref{connection} it is easily seen that $\nabla d b_{\gamma}(V,V') = \frac{1}{2}\langle V,V'\rangle$, $\nabla d b_{\gamma}(Z,Z') = \langle Z,Z'\rangle$,  and $\nabla d b_{\gamma}(V,Z) = 0$, $V,V'\in {\mathfrak v}$, $Z,Z'\in {\mathfrak z}$. So, the proposition is proved.

 \vspace{2mm}
 Let $s\not=1$. We use the formula (\ref{busefuncexpression}) in Proposition \ref{busemann}.
We identify a point $p = (\exp_{\mathfrak n}(U+X), a=e^r)$ with its
  coordinates $(v^i(p), y^{\alpha}(p), a(p))$ in terms of  $U = \sum_i v^i(p) V_i$, $X = \sum_{\alpha} y^{\alpha}(p) Y_{\alpha}$, $a(p) = e^{\lambda(p)},\, \lambda(p) = r$. We introduce the following maps and the functions
  \begin{eqnarray}
 {\mathcal V} : S\times \partial S \rightarrow {\mathfrak v}&:&\hspace{2mm}{\mathcal V}(p,(v,y)) := v - U,\\ \nonumber
 {\mathcal Y} : S\times \partial S \rightarrow {\mathfrak z}&:&\hspace{2mm}{\mathcal Y}(p,(v,y)) := y - X -\frac{1}{2}[U,v],\\ \nonumber
 f(p,(v,y)) &:=& a + \frac{1}{4}\vert {\mathcal V}(p,(v,y))\vert^2,\\ \nonumber
   F(p,(v,y)) &:=& f^2(p,(v,y)) + \vert {\mathcal Y}(p,(v,y))\vert^2, \\ \nonumber
  \end{eqnarray}respectively so that
$b_{\gamma}$ can be written by
  \begin{eqnarray}\label{formbusemann}
   b_{\gamma}(p) = \log F(p,(v,y)) - \log a + C((v,y)). 
   \end{eqnarray}Here $C((v,y)) := - \log\big((1+\frac{1}{4}\vert v\vert^2)^2 +\vert y\vert^2  \big)$ is a constant function on $S$.
   Notice that for a fixed $(v,y)$ thus defined $f$, $F$, ${\mathcal V}$ and ${\mathcal Y}$ are functions of $p$ and
  expressed by $f = f(p)$, $F = F(p)$, ${\mathcal V}= {\mathcal V}(p)$ and ${\mathcal Y}={\mathcal Y}(p)$, respectively.

 Thus, the Hessian can be computed by using the formula
   \begin{eqnarray}\label{loghessian}
   \nabla d \log F(u,v) 
   = \frac{1}{F}\big(u({\tilde v}F) - (\nabla_u{\tilde v})F  \big)- \frac{1}{F^2} uF\, vF,\hspace{2mm}u,v\in T_pS
   \end{eqnarray}(${\tilde v}$ is an extension of $v$).

    Let $p = (\exp_{\mathfrak n}(U+X), e^r)$ be a point of $S$ with coordinates $\{v_i, y_{\alpha}, a = e^{\lambda}\}$. Then, straightforward computations give us via (\ref{coordinatesons})
\begin{eqnarray}\label{firstderivative}
V_i\hspace{0.5mm} F&=&  \sqrt{a}\, \frac{\partial}{\partial v_i}F - \frac{1}{2}\, \sqrt{a}\, \sum_{j,\alpha}\, \langle [V_i, V_j], Y_{\alpha}\rangle \, v_j\ \frac{\partial}{\partial y_{\alpha}}F \\ \nonumber
&=& - \sqrt{a}\hspace{0.5mm}\langle f {\mathcal V}- J_{\mathcal Y}{\mathcal V}, V_i\rangle , \\ 
\label{firstderivative-2}
Y_{\alpha}\hspace{0.5mm} F &=& e^{\lambda}\, \frac{\partial}{\partial y_{\alpha}}F = - 2a\hspace{0.5mm} \langle {\mathcal Y}, Y_{\alpha}\rangle, 
A\hspace{0.5mm} F = \frac{\partial}{\partial \lambda}F = a \frac{\partial}{\partial a}F = 2 a f
\end{eqnarray}
   and for the ${\mathfrak v}$-, ${\mathfrak z}$-valued functions ${\mathcal V}$, ${\mathcal Y}$
\begin{eqnarray}\label{VY}
 V_i({\mathcal V}) &=& - \sqrt{a} V_i, \hspace{2mm}
 V_i({\mathcal Y}) = - \frac{\sqrt{a}}{2} [V_i,{\mathcal V}],\\ 
 Y_{\alpha}({\mathcal V}) &=&0,\hspace{2mm}
 Y_{\alpha}({\mathcal Y}) = - a Y_{\alpha}
 \end{eqnarray}
which are derived from $\displaystyle{
   \frac{\partial}{\partial v^i} {\mathcal V} = - V_i,\hspace{2mm} \frac{\partial}{\partial y^{\alpha}} {\mathcal V} = 0, 
     \frac{\partial}{\partial v^i} {\mathcal Y} = - \frac{1}{2}[V_i, v], \hspace{2mm} \frac{\partial}{\partial y^{\alpha}} {\mathcal Y} = - Y_{\alpha}}$.


   \begin{lemma}\label{componenthessian}
   \rm \cite{Satoh}\hspace{2mm}The components of the Hessian $\nabla d b_{\gamma}$ with respect to a basis $\{ V_i, Y_{\alpha}, A\}$ of ${\mathfrak s}={\mathfrak v}\oplus{\mathfrak z}\oplus{\mathfrak a}$ at a point $p = (\exp_{\mathfrak n}(U+X), a)$ are given by
    \begin{eqnarray}
   \nabla d b_{\gamma}(A,A) &=& \frac{2a}{F^2}\left(fF + aF - 2a f^2\right),\\ \nonumber
  \nabla d b_{\gamma}(A,V_i)&=& \nabla d b_{\gamma}(V_i,A) = - \frac{\sqrt{a}}{2F^2}\big\{ \big(fF + 2aF - 4a f^2\big) \langle {\mathcal V},V_i\rangle \\ \nonumber
   &+& \left( 4 a f - F\right) \langle J_{\mathcal Y}{\mathcal V}, V_i\rangle \big\},\\ \nonumber
  \nabla d b_{\gamma}(A,Y_{\alpha})&=& \nabla d b_{\gamma}(Y_{\alpha},A) = \frac{2a}{F^2} (2a f - F) \langle {\mathcal Y}, Y_{\alpha}\rangle,
  \end{eqnarray}
   \begin{eqnarray}
  \label{ijcompo}
 \hspace{10mm}\nabla d b_{\gamma}(V_i,V_j) &=& \nabla d b_{\gamma}(V_j,V_i) \\ \nonumber
 &=& \frac{1}{2} \langle V_i,V_j\rangle + \frac{a}{2F}\big(\langle{\mathcal V},V_i\rangle \langle {\mathcal V}, V_j\rangle +  \langle [{\mathcal V}, V_i], [{\mathcal V}, V_j]\rangle\big) \\ \nonumber
  &-& \frac{a}{F^2} \langle f{\mathcal V} - J_{\mathcal Y}{\mathcal V}, V_i\rangle \langle f {\mathcal V} - J_{\mathcal Y}{\mathcal V}, V_j\rangle ,   \\ \nonumber
 \label{ialphacompo}
 \nabla d b_{\gamma}(V_i,Y_{\alpha}) &=& \nabla d b_{\gamma}(Y_{\alpha},V_i) = - \frac{2a\sqrt{a}}{F^2} \langle f {\mathcal V} - J_{\mathcal Y}{\mathcal V},V_i\rangle \langle {\mathcal Y}, Y_{\alpha}\rangle \\ \nonumber
  &-& \frac{\sqrt{a}}{2F} \langle [V_i, (f-2a){\mathcal V} - J_{\mathcal Y}{\mathcal V}], Y_{\alpha}\rangle , \\ \nonumber
 \label{alphabetacompo}
  \nabla d b_{\gamma}(Y_{\alpha},Y_{\beta}) &=&\nabla d b_{\gamma}(Y_{\beta},Y_{\alpha}) \\ \nonumber &=& \frac{1}{F} (F - 2a f + 2a^2) \langle Y_{\alpha}, Y_{\beta}\rangle  - \frac{4 a^2}{F^2} \langle {\mathcal Y}, Y_{\alpha}\rangle\langle {\mathcal Y}, Y_{\beta}\rangle .
   \end{eqnarray}

     \end{lemma}

   \begin{remark}\rm   Using the above formulae  we can see that the velocity vector $V + Y + s A$ of a geodesic $\gamma$ is a null vector of $\nabla d b_{\gamma}$, one of the basic properties of the Busemann functions. However we omit the detail.
   \end{remark}

  To get the formula of $\nabla d b_{\gamma}(V_i,V_j)$ for example, we use (\ref{formbusemann}) and (\ref{loghessian})  as
  \begin{eqnarray}\label{vivj}
  \nabla d \log F(V_i,V_j)=
  \frac{1}{F}\{V_i(V_j F) - (\nabla_{V_i}V_j)F\} - \left(\frac{V_i F}{F}  \right)\left(\frac{V_j F}{F}  \right) ,
  \end{eqnarray}where $V_j F$ is computed from
  (\ref{firstderivative})  and then
   \begin{eqnarray*}
  V_i(V_j F) &=& V_i\big(- \sqrt{a}\hspace{0.5mm}\{ f\langle  {\mathcal V}, V_j\rangle - \langle J_{\mathcal Y}{\mathcal V}, V_j\rangle\}
  \big) \\ \nonumber
  &=& - \sqrt{a}\{ (V_i f) \langle  {\mathcal V}, V_j\rangle + f\, V_i \langle  {\mathcal V}, V_j\rangle - V_i\, \langle J_{\mathcal Y}{\mathcal V}, V_j\rangle\} .
  \end{eqnarray*}Here, from (\ref{VY})
   \begin{eqnarray*}V_i f &=& - 1/2\, \sqrt{a} \langle {\mathcal V}, V_i\rangle, \\
V_i \langle{\mathcal V}, V_j\rangle &=& \langle V_i({\mathcal V}), V_j\rangle = - \sqrt{a} \langle V_i, V_j\rangle, \\
V_i\, \langle J_{\mathcal Y}{\mathcal V}, V_j\rangle &=& \langle J_{V_i({\mathcal Y})}{\mathcal V}, V_j \rangle + \langle J_{\mathcal Y}V_i({\mathcal V}), V_j\rangle \\ \nonumber
&=& - \sqrt{a} \langle [V_i,V_j], {\mathcal Y}\rangle + \frac{\sqrt{a}}{2}\langle [ {\mathcal V}, V_j], [{\mathcal V},V_i]\rangle .
  \end{eqnarray*}
  Thus
\begin{eqnarray*}
V_i(V_j F) = \frac{a}{2}\langle {\mathcal V}, V_i\rangle \langle {\mathcal V}, V_j\rangle + a f \langle V_i, V_j\rangle - a \langle[V_i,V_j], {\mathcal Y}\rangle + \frac{a}{2}\langle [{\mathcal V},V_j],[{\mathcal V}, V_i]\rangle.
\end{eqnarray*}

The term $(\nabla_{V_i}V_j)F$ is given by $-a \langle[V_i,V_j], {\mathcal Y}\rangle + a f \langle V_i, V_j\rangle$ from (\ref{levicivita}). Thus,
$\frac{1}{F}\{V_i(V_j F) - (\nabla_{V_i}V_j)F\}= \frac{a}{2F} \big\{\langle {\mathcal V},V_i\rangle \langle {\mathcal V},V_j\rangle+ \langle [{\mathcal V},V_j],[V_i,{\mathcal V}]\rangle\big\}$  and then (\ref{vivj})
turns out to be
\begin{eqnarray*}
\frac{a}{2F} \big\{\langle {\mathcal V},V_i\rangle \langle {\mathcal V},V_j\rangle+ \langle [{\mathcal V},V_i],[{\mathcal V},V_j]\rangle\big\}  - \frac{a}{F^2}\langle f {\mathcal V}- J_{\mathcal Y}{\mathcal V}, V_i\rangle \langle f {\mathcal V}- J_{\mathcal Y}{\mathcal V}, V_j\rangle .
\end{eqnarray*}This formula together with $\displaystyle{- \nabla d \log a(V_i,V_j) =  \frac{1}{a} (\nabla_{V_i}V_j)a = \frac{1}{2} \langle V_i, V_j\rangle}$ yields the formula for $\nabla d b_{\gamma}(V_i,V_j)$.

To get the last formula  of Lemma \ref{componenthessian} we have
\begin{eqnarray*}
   Y_{\beta} F 
  = - 2a \langle {\mathcal Y}, Y_{\beta}\rangle, \hspace{2mm}
 Y_{\alpha}(Y_{\beta} F) 
  = 2 a^2 \langle Y_{\alpha},Y_{\beta}\rangle.
\end{eqnarray*}

The term $(\nabla_{Y_{\alpha}}Y_{\beta})F$ is given by $\langle Y_{\alpha},Y_{\beta}\rangle A F = 2a \langle Y_{\alpha},Y_{\beta}\rangle f$ 
from (\ref{levicivita}) so that
\begin{eqnarray*}\frac{1}{F}\{Y_{\alpha}(Y_{\beta} F) - (\nabla_{Y_{\alpha}}Y_{\beta})F\}= \frac{1}{F}\big\{
2 a^2\langle Y_{\alpha},Y_{\beta}\rangle - 2a f\langle Y_{\alpha},Y_{\beta}\rangle \big\}
\end{eqnarray*}
  and

  $\displaystyle{\left(\frac{Y_{\alpha}F}{F}\right)\left(\frac{Y_{\beta}F}{F}\right) = \frac{4a^2}{F^2}\langle {\mathcal Y},Y_{\alpha}\rangle\langle {\mathcal Y},Y_{\beta}\rangle}.$
 Then (\ref{vivj})
turns out to be
\begin{eqnarray*}
\frac{1}{F}(2a^2 -2 a f)\,\langle Y_{\alpha},Y_{\beta}\rangle - \frac{4a^2}{F^2}\langle {\mathcal Y},Y_{\alpha}\rangle\langle {\mathcal Y},Y_{\beta}\rangle.
\end{eqnarray*}
This formula together with $\displaystyle{\nabla d \log a(Y_{\alpha},Y_{\beta}) =  \frac{1}{a} (\nabla_{Y_{\alpha}}Y_{\beta})a = \langle Y_{\alpha}, Y_{\beta}\rangle}$ yields the formula for $\nabla d b_{\gamma}(Y_{\alpha},Y_{\beta})$. Other components of the Hessian are similarly obtained.

\begin{lemma}\label{hessiancomponents-x}\rm  Let $W, W'\in {\mathfrak v}$ and $Z,Z'\in {\mathfrak z}$.
The Hessian $\nabla d b_{\gamma}$ satisfies at $e_S$  the following;
\begin{eqnarray}\label{aacompo-1}
   \nabla d b_{\gamma}(A,A) &=& \frac{2}{F^2}\left(fF + F - 2 f^2\right),
  \\ 
  \label{aw}
  \nabla d b_{\gamma}(A, W)&=&  - \frac{1}{2F^2}\big\{ \big(fF + 2F - 4 f^2\big) \langle v,W\rangle \\ \nonumber
 &+& \left( 4 f - F\right) \langle J_yv, W\rangle \big\},\\
  \label{az}
 \nabla d b_{\gamma}(A,Z)&=&  \frac{2}{F^2} (2f - F) \langle y, Z\rangle,
 \end{eqnarray}
 \begin{eqnarray}
  \label{ww}
 \hspace{10mm}\nabla d b_{\gamma}(W,W') &=&  \frac{1}{2} \langle W, W'\rangle + \frac{1}{2F}\big(\langle v,W\rangle \langle v, W'\rangle +  \langle [v, W], [v, W']\rangle\big) \\ \nonumber
  &-& \frac{1}{F^2} \langle f v - J_yv, W\rangle \langle f v - J_yv, W'\rangle ,   \\ 
 \label{wz} \nabla d b_{\gamma}(W,Z) &=&  - \frac{2}{F^2} \langle f v - J_yv,W\rangle \langle y, Z\rangle \\ \nonumber
 &-& \frac{1}{2F} \langle [W, f v- J_yv - 2v], Z\rangle , \\ 
\label{zz-1} 
 \nabla d b_{\gamma}(Z,Z') &=& \frac{1}{F} (F - 2 f + 2) \langle Z, Z'\rangle - \frac{4}{F^2} \langle y, Z\rangle\langle y, Z'\rangle .
  \end{eqnarray}
 \end{lemma}
  \subsection{Block decomposition of $\nabla d b_{\gamma}$}

  Let $\gamma$ be a geodesic of $\gamma(0) = e_S$, $\gamma'(0) = V+ Y+ sA$, a unit vector such that $V\not=0$ and $Y\not=0$.
   \begin{proposition}\label{decompositionofs}\rm
  The Lie algebra
  ${\mathfrak s}$ admits the orthogonal decomposition;
   \begin{eqnarray}\label{orthog-1}
  {\mathfrak s} &=& {\mathfrak s}_4 \oplus {\mathfrak p} \oplus {\mathfrak q},\hspace{2mm}{\mathfrak q}={\mathfrak t}\oplus Y^{\perp}, \\ 
    {\mathfrak s}_4 &:=& {\rm Span}\hspace{0.5mm}\{ V, J_Y V, Y, A\},\\ 
       {\mathfrak p} &:=& {\rm Ker}\, (ad\hspace{0.5mm}V) \cap {\rm Ker}\, (ad\hspace{0.5mm} J_YV), \\ 
    {\mathfrak t} &:=& \{ J_{Y^{\perp}}V, J_{Y^{\perp}}J_YV\}, \\ \nonumber
    Y^{\perp}\hspace{0.5mm}&:=& \{ Z\in {\mathfrak z}\, \vert Z \perp Y \},
    \end{eqnarray}($\displaystyle{J_{Y^{\perp}}V :=\{ J_ZV,\, \vert\, Z\in Y^{\perp}\} }$).
    \end{proposition}    Refer to \cite{BTV}, p.97 for Proposition \ref{decompositionofs}  whose proof will be given in {\sc Appendix II},\, {\sc Part I}. The subspace $\gamma'(0)^{\perp}$ is therefore decomposed into
  \begin{eqnarray*}
 (V+Y+sA)^{\perp} = {\mathfrak s}_4^0\oplus{\mathfrak p}\oplus {\mathfrak t}_1\oplus {\mathfrak t}_2\oplus {\mathfrak t}_3 \oplus (Y^{\perp})_J \oplus (Y^{\perp})_J^{\perp}.
  \end{eqnarray*}Here ${\mathfrak s}_4^0 := {\mathfrak s}_4\cap (V+Y+sA)^{\perp}$.  Due to the notation of \cite{BTV} and (\ref{decompoq}), (\ref{spaceq})
  \begin{eqnarray}{\mathfrak q} &=& {\mathfrak q}_0\oplus \cdots\oplus {\mathfrak q}_{\ell}, \\
  \label{t-1}
   {\mathfrak q}_{\ell} &=& {\mathfrak t}_1\oplus (Y^{\perp})_J,\hspace{2mm}{\mathfrak t}_1 := J_{(Y^{\perp})_J}V
  \end{eqnarray}and
  \begin{eqnarray}
   {\mathfrak q}_0  \oplus \cdots \oplus  {\mathfrak q}_{\ell-1} &=& {\mathfrak t}_2\oplus {\mathfrak t}_3 \oplus (Y^{\perp})_J^{\perp},\hspace{2mm} \\
   \label{t2t3}
   {\mathfrak t}_2 := J_{(Y^{\perp})_J^{\perp}}V,& &\hspace{-2mm}{\mathfrak t}_3 := J_{(Y^{\perp})_J^{\perp}}J_YV.
  \end{eqnarray}

  Then we have an orthogonal decomposition;
  \begin{eqnarray}\label{adapteddecompo}
 (V+Y+sA)^{\perp} = {\mathfrak s}_4^0\oplus
 {\mathfrak p}\oplus \left(\oplus_{j=0}^{\ell-1} {\mathfrak q}_j\right)\oplus {\mathfrak q}_{\ell}.
  \end{eqnarray}
 \begin{lemma}\label{orthogonality}\rm With respect to the decomposition (\ref{adapteddecompo})
  \begin{eqnarray*}
 ({\rm i})\hspace{2mm} \nabla d b_{\gamma}({\mathfrak s}_4^0,{\mathfrak p}) &=& 0, \hspace{2mm}({\rm ii})\hspace{2mm}\nabla d b_{\gamma}({\mathfrak s}_4^0,{\mathfrak q}_j) = 0,\ 0\leq j \leq \ell, 
  \nonumber \hspace{2mm}\\
  ({\rm iii})\hspace{2mm}\nabla d b_{\gamma}({\mathfrak p}, {\mathfrak q}_j) &=& 0,\, 0\leq j \leq \ell, 
 \, ({\rm iv})\hspace{1mm}\nabla d b_{\gamma}({\mathfrak q}_i,{\mathfrak q}_j) = 0, \hspace{1mm}i\not=j,  0\leq i<j \leq \ell.  
  \end{eqnarray*}so that
  \begin{eqnarray*}
  \nabla d b_{\gamma} = \nabla d b_{\gamma}\vert_{{\mathfrak s}_4^0} \oplus
  \nabla d b_{\gamma}\vert_{\mathfrak p} \oplus \left(\oplus_{j=0}^{\ell-1}\nabla d b_{\gamma}\vert_{{\mathfrak q}_j}\right)\oplus \nabla d b_{\gamma}\vert_{{\mathfrak q}_{\ell}}.
  \end{eqnarray*}
  \end{lemma}
  \begin{proof}\rm
  From the orthogonality of ${\mathfrak s}_4^0$ and ${\mathfrak p}$, and definition of ${\mathfrak p}$ it is obvious to see (i), since $[V,W]=[J_YV,W]=0$ for $W\in{\mathfrak p}$. (iii) is similarly shown.  For verifying (ii) it suffices to prove that $\nabla d b_{\gamma}(V, {\mathfrak q}_j) = \nabla d b_{\gamma}(J_YV, {\mathfrak q}_j) = \nabla d b_{\gamma}(Y, {\mathfrak q}_j) = 0$ and $\nabla d b_{\gamma}(A, {\mathfrak q}_j) = 0$, since $\gamma'(0)=V+Y+sA\in {\mathfrak s}_4$ is a null vector of $\nabla d b_{\gamma}$.

  Let $W + Z\in {\mathfrak q}_j$.  $W$ is written for $j$ of $\mu_j > -1$ by $W = J_{Z_1}{\hat V} + J_{Z_2}J_{\hat Y}{\hat V}$ with $Z, Z_a\in L_j$, $a=1,2$ and for $j=\ell$ of $\mu_{\ell}=-1$\, $W = J_{Z_1}{\hat V}$ with $Z,Z_1\in (Y^{\perp})_J=L_{\ell}$, respectively. Then from the formulae of Lemma \ref{hessiancomponents-x} 
  together with Lemma \ref{ljyeperp} below we see (ii). (iv) is also shown by the aid of Lemma \ref{ljyeperp}.

  \end{proof}
  \begin{lemma}\label{ljyeperp}\rm
  Let $Z\in L_j \left(\subset Y^{\perp}\right)$. Then
  \begin{eqnarray}\label{bracketlj}
  [v, J_Z{\hat V}],\hspace{1mm}[v,J_ZJ_{\hat Y}{\hat V}],\hspace{1mm}
    [J_yv, J_Z{\hat V}], \hspace{1mm}[J_yv,J_ZJ_{\hat Y}{\hat V}]\in {\rm Span}\, \{Z, KZ\} \subset  L_j,
  \end{eqnarray}more precisely
  \begin{eqnarray*}
  [v, J_Z{\hat V}] &=& \frac{2\vert V\vert}{\chi_{\infty}}\big\{(1-s) Z - \vert Y\vert KZ
    \big\},\\ \nonumber
  [v, J_ZJ_{\hat Y}{\hat V}] &=& \frac{2\vert V\vert}{\chi_{\infty}}\big\{\vert Y\vert Z + (1-s)KZ\big\}{\red{.}}
  \end{eqnarray*}
  \end{lemma}
  \begin{proof}\rm One has
  \begin{eqnarray*}
  [V, J_Z{\hat V}] &=& \vert V\vert [{\hat V}, J_Z{\hat V}] = \vert V\vert Z
  \end{eqnarray*}and by using (\ref{bracket})\, namely the formula $ [J_ZU, J_{Z}V] = - \vert Z\vert^2[U,V] - 2 \langle U, J_ZV\rangle_{\mathfrak n} Z$
    \begin{eqnarray*}
  [J_YV, J_Z{\hat V}] &=& \vert Y\vert\vert V\vert [J_{\hat Y}{\hat V}, J_Z{\hat V}]
   =\vert Y\vert\vert V\vert (-[J_{\hat Y}{\hat V}, J_{\hat Y}(J_{\hat Y}J_Z{\hat V})]) \\ \nonumber
 &=& \vert Y\vert\vert V\vert \big\{-\left(-\vert Y\vert^2 [{\hat V}, J_{\hat Y}J_Z{\hat V}]- 2\langle {\hat V},J_{\hat Y}J_{\hat Y}J_Z {\hat V}\rangle {\hat Y}\right)\big\}  \\ \nonumber
 &=&\vert Y\vert\vert V\vert \left( [{\hat V}, J_{\hat Y}J_Z{\hat V}] - 2\langle {\hat V}, J_Z{\hat V}\rangle {\hat Y}\right) \\ \nonumber
  &=&  \vert Y\vert\vert V\vert\, [{\hat V}, J_{\hat Y}J_Z{\hat V}] = - \vert Y\vert\vert V\vert\, [{\hat V}, J_ZJ_{\hat Y}{\hat V}] \\ \nonumber
  &=& - \vert Y\vert\vert V\vert\, K(Z)
  \end{eqnarray*}so that (\ref{bracketlj}) is shown, since $v$ and $J_yv$ are a linear combination of $V$ and $J_YV$, respectively. Moreover
  \begin{eqnarray*}
   [V, J_ZJ_{\hat Y}{\hat V}] = \vert V\vert [{\hat V}, J_ZJ_{\hat Y}{\hat V}] = \vert V\vert K(Z)
  \end{eqnarray*}and
  \begin{eqnarray*}
   [J_YV, J_ZJ_{\hat Y}{\hat V}] = \vert Y\vert\vert V\vert [J_{\hat Y}{\hat V}, J_ZJ_{\hat Y}{\hat V}] 
   = \vert Y\vert\vert V\vert Z.
  \end{eqnarray*}Therefore the later parts of  (\ref{bracketlj}) are similarly obtained.
  \end{proof}

  Thus, for asserting the positive definiteness of $\nabla d b_{\gamma}$ it suffices from Lemma \ref{orthogonality} to show the positive definiteness to each subspace of (\ref{adapteddecompo}).

         \begin{lemma}\rm For $W, W'\in {\mathfrak p}$
     \begin{eqnarray*}
     \nabla d b_{\gamma}(W,W') = \frac{1}{2} \langle W, W'\rangle.  
   \end{eqnarray*}
   Henceforth, each $W$ in ${\mathfrak p}$ is an eigenvector of $\nabla d b_{\gamma}$ corresponding to eigenvalue $\frac{1}{2}$. \end{lemma}

   This is a direct consequence of (\ref{ww}), since $W, W' \in ({\rm Ker}\, ad V)\cap ({\rm Ker}\, ad J_YV)$ and $v$, $J_yv$ are vectors linearly spanned by $V$, $J_YV$.

\subsection{Auxiliary formulae}
 The value of ${\mathcal V}$ and ${\mathcal Y}$ at
$e_S$ are
\begin{eqnarray*}
 {\mathcal V}(e_S) = v\, =\, \frac{2}{\chi_{\infty}}\{(1-s)V+J_YV\}, \hspace{2mm}  {\mathcal Y}(e_S) =  y = \frac{2}{\chi_{\infty}} Y,
\end{eqnarray*}since   $U = 0$, $X=0$ and $a= 1$ at $e_S$. Here $v$ and $y$ are the coordinates of $[\gamma]\in \partial S$ defined by (\ref{vy}).  Note ${\mathcal V}(e_S) = v\in {\rm Span}\, \{V, J_YV\}$. 
Moreover,
\begin{eqnarray}\label{ff}
f(e_S) = 
\frac{2(1-s)}{\chi_{\infty}},\hspace{2mm}
F(e_S)= f^2(e_S) + \vert {\mathcal Y}\vert^2(e_S) 
= \frac{4}{\chi_{\infty}}.
\end{eqnarray}
As before, we let $\chi_{\infty} = (1-s)^2 +\vert Y\vert^2 = \lim_{t\rightarrow\infty}\chi(t)$, the limit of the function $\chi(t)$. From (\ref{ff}) we have
\begin{eqnarray*}\label{fv}
f(e_S) v = \frac{4(1-s)}{\chi^2_{\infty}}\, \left( (1-s)V + J_YV\right),\hspace{2mm}
J_y v  
= \frac{4}{\chi^2_{\infty}}\{(1-s) J_YV - \vert Y\vert^2 V\}
\end{eqnarray*}so that at $e_S$
\begin{eqnarray*}\label{auxiliary} \hspace{5mm} f(e_S) v - J_yv= \frac{4}{\chi_{\infty}}V\hspace{2mm}{\rm and}\hspace{2mm}\displaystyle{4f - F = 4 \frac{(1-2s)}{\chi_{\infty}} .}\end{eqnarray*}


   \subsection{An adapted basis}

    \vspace{2mm}
   First we take an orthonormal basis of the Lie subalgebra ${\mathfrak s}_4$ of ${\mathfrak s}$;
        \begin{eqnarray}\label{linearcombi}
   W^{{\mathfrak s}_4}_0 &:= & V+Y+sA,\\ \nonumber
   W^{{\mathfrak s}_4}_1 &:= & \vert V\vert J_{\hat Y}{\hat V} + s {\hat Y} - \vert Y\vert A,
   \\ \nonumber 
   W^{{\mathfrak s}_4}_2 &:= &  s {\hat V} - \vert Y\vert J_{\hat Y}{\hat V} - \vert V\vert A,
   \\ \nonumber
  W^{{\mathfrak s}_4}_3  &:= & \vert Y\vert {\hat V} + s J_{\hat Y}{\hat V} - \vert V\vert {\hat Y}.
     \end{eqnarray}
              \begin{proposition}\label{s4}\rm
      $W^{{\mathfrak s}_4}_1$ is an eigenvector of $\nabla d b_{\gamma}$ of eigenvalue $1$ at $e_S$ 
       and $W^{{\mathfrak s}_4}_2, W^{{\mathfrak s}_4}_3$ are the eigenvectors of $\nabla d b_{\gamma}$ of eigenvalue $\frac{1}{2}$ at $e_S$.
     \end{proposition}
    \begin{note}\rm The above basis corresponds to the eigenvectors of $R_{V+Y+sA}$. Namely, $W^{{\mathfrak s}_4}_1$ is the eigenvector with eigenvalue $-1$ and $W^{{\mathfrak s}_4}_2, W^{{\mathfrak s}_4}_3$ are eigenvectors with eigenvalue $-\frac{1}{4}$.  Refer to \cite{BTV}.
    \end{note}
       Now decompose, as before, $Y^{\perp} \subset {\mathfrak z}$ into
    \begin{eqnarray*}\label{decompy}
      Y^{\perp} = (Y^{\perp})_J \oplus  (Y^{\perp})_J^{\perp}.
    \end{eqnarray*}

     \vspace{2mm}
       Put $\dim (Y^{\perp})_J = k_1$ and $\dim (Y^{\perp})_J^{\perp} = k_2$. Then $\dim Y^{\perp} = k_1 + k_2$ and $k = \dim {\mathfrak z} = k_1 + k_2 + 1$.

   The orthogonal decomposition of $Y^{\perp}$ induces, therefore, a decomposition of the subspace ${\mathfrak t}= \{J_{Y^{\perp}}V, J_{Y^{\perp}}J_YV\}$ as $\displaystyle{    {\mathfrak t} = {\mathfrak t}_1 \oplus {\mathfrak t}_2 \oplus {\mathfrak t}_3.}$  

    Note that ${\mathfrak t}_1 = J_{(Y^{\perp})_J} V = J_{(Y^{\perp})_J} J_YV$ is  $J_Y$-invariant.  By the aid of the bases $\{Z^J_{\alpha}\}$ and  $\{Z^{J^{\perp}}_{\beta}\}$ of $(Y^{\perp})_J$ and $(Y^{\perp})_J^{\perp}$, we provide an orthonormal basis for each of the subspaces ${\mathfrak t}_a$, $a =1,2,3$, respectively as follows;
      \begin{eqnarray*}\label{t1}
 & & \big\{W^{{\mathfrak t}_1}_1=J_{Z^J_{1}}{\hat V}, \dots, W^{{\mathfrak t}_1}_{k_1}=J_{Z^J_{k_1}}{\hat V}\big\}, \\
 & &\label{t2} \big\{W^{{\mathfrak t}_2}_1 =J_{Z^{J^{\perp}}_1}{\hat V}, \dots, W^{{\mathfrak t}_2}_{k_2}= J_{Z^{J^{\perp}}_{k_2}}{\hat V}\big\},\\
& &\label{t3}   \big\{W^{{\mathfrak t}_3}_1 =J_{Z^{J^{\perp}}_1}J_{\hat Y}{\hat V}, \dots, W^{{\mathfrak t}_3}_{k_2}= J_{Z^{J^{\perp}}_{k_2}}J_{\hat Y}{\hat V}\big\}.
   \end{eqnarray*} 
   Thus, $\dim {\mathfrak t}_1 = k_1,\,  \dim {\mathfrak t}_2 = \dim {\mathfrak t}_3 = k_2$ and $\dim {\mathfrak t} = k_1 + 2 k_2$.

\vspace{2mm}
   The components of the Hessian $\nabla d b_{\gamma}$ 
   with respect to $V_1 = {\hat V}, V_2 = J_{\hat Y}{\hat V}, Y_1 = {\hat Y}, A$ are given by the following. At $e_S$   we have $a = 1$, ${\mathcal V} = v$, ${\mathcal Y} = y$ and from this, we have the following.

    \begin{lemma}\label{component3}\rm
   \begin{eqnarray*}
  \nabla d b_{\gamma}(A, V_2) &=& - \frac{1}{2F^2}\big\{ \big(fF + 2F - 4f^2\big) \langle v,V_2\rangle \\ \nonumber
   &+& \left( 4 f - F\right) \langle J_yv, V_2\rangle \big\}=-\frac{1}{2}\vert Y\vert\vert V\vert, \\ \nonumber
  \nabla d b_{\gamma}(V_1, V_2)  &=& \frac{1}{2F}\big(\langle v,V_1\rangle \langle v, V_2\rangle +  \langle [v, V_1], [v, V_2]\rangle\big)=0, \\ \nonumber
  \nabla d b_{\gamma}(V_2, V_2)  &=&\frac{1}{2F}\big(\langle v,V_2\rangle \langle v, V_2\rangle +  \langle [v, V_2], [v, V_2]\rangle\big) + \frac{1}{2}=\frac{1}{2}(1+\vert V\vert^2),\\ \nonumber
   \nabla d b_{\gamma}(V_2,Y_1) &=& - \frac{2}{F^2} \langle f v - J_yv,V_2\rangle \langle y, Y_{1}\rangle \\ \nonumber
  &-& \frac{1}{2F} \langle [V_2, (f-2) v - J_yv], Y_{1}\rangle=\frac{1}{2} s \vert V\vert.
   \end{eqnarray*} 
   \end{lemma}

      \begin{lemma}\label{eigen-1}\rm
      \begin{eqnarray*}
      \nabla d b_{\gamma}(W^{\mathfrak s}_1,W^{\mathfrak s}_1) &=& 1,\hspace{2mm}\hspace{2mm} \nabla d b_{\gamma}(W^{\mathfrak s}_1,W^{\mathfrak s}_i) = 0, \, i = 2,3, \\ \nonumber
        \nabla d b_{\gamma}(W^{\mathfrak s}_i, W^{\mathfrak s}_j) &=&  \frac{1}{2} \delta_{ij}, \hspace{2mm}i, j = 2,3. 
     \end{eqnarray*}
     \end{lemma}

    \begin{proof}\rm   Since $W^{{\mathfrak s}_4}_1 = \vert V\vert J_{\hat Y}{\hat V} + s {\hat Y} - \vert Y\vert A$, using the formulae of Lemma \ref{component3} we verify that
         \begin{eqnarray*}
      \nabla d b_{\gamma}(W^{\mathfrak s}_1,W^{\mathfrak s}_1) 
   &=& \vert V\vert^2\left(\frac{1}{2}\vert V\vert^2+ \frac{1}{2}\right) + 2s\vert V\vert\left(\frac{s}{2}\vert V\vert  \right)- 2\vert V\vert\vert Y\vert\left(-\frac{\vert V\vert\vert Y\vert}{2}\right) \\ \nonumber
     &+& s^2\big\{\frac{1}{2}(1+s^2-\vert Y\vert^2)\big\} - 2 s\vert Y\vert(-s \vert Y\vert) + \vert Y\vert^2\cdot\frac{1}{2}(1-s^2 + \vert Y\vert^2) \\ \nonumber
     &=& 1.
      \end{eqnarray*} Others are similarly obtained.
\end{proof}


 \vspace{2mm}
   It is concluded from Lemma \ref{eigen-1} 
  that $W_1^{\mathfrak s}$ is an eigenvector of $\nabla d b_{\gamma}$ corresponding to eigenvalue $1$, and $W_2^{\mathfrak s}$ and $W_3^{\mathfrak s}$ are eigenvectors corresponding to eigenvalue $\frac{1}{2}$.

\subsection{Positive definiteness over ${\mathfrak q}_{\ell} = {\mathfrak t}_1\oplus (Y^{\perp})_J$}\label{positive}

The aim of this subsection is to show the following. 

\begin{lemma}\label{positive-1}\rm $\nabla d b_{\gamma}$ over ${\mathfrak t}_1\oplus (Y^{\perp})_J$ is positive definite.
\end{lemma}

Decompose ${\mathfrak t}_1\oplus(Y^{\perp})_J$ orthogonally into
\begin{eqnarray*}
 {\mathfrak t}_1\oplus(Y^{\perp})_J = \oplus_{a=1}^{r_{\ell}}\, {\mathfrak h}_a,\hspace{2mm} {\mathfrak h}_a := {\rm Span}\, \{W_{2a-1}, W_{2a}, Z_{2a-1}, Z_{2a}\} 
\end{eqnarray*} $(\hspace{1mm}\dim (Y^{\perp})_J = 2 r_{\ell}\hspace{1mm})$. Then, it follows immediately
\begin{eqnarray*}
 \nabla d b_{\gamma}\vert_{{\mathfrak t}_0^1\oplus(Y^{\perp})_J } = \oplus_{a=1}^{r_{\ell}} \nabla d b_{\gamma} \vert_{ {\mathfrak h}_a}
\end{eqnarray*}as a direct sum of the quadratic forms $\nabla d b_{\gamma} \vert_{ {\mathfrak h}_a}$.  We will  show $\displaystyle{\nabla d b\vert_{ {\mathfrak h}_a} > 0}$,\, $a = 1,\cdots,r_{\ell}$ and conclude thus $\nabla d b_{\gamma} > 0$ over ${\mathfrak t}_1\oplus (Y^{\perp})_J$.

Instead of showing $\nabla d b_{\gamma} > 0$ we rather consider the eigenproblem for $ {\mathcal T}=-{\mathcal S} $, the endomorphism induced from $\nabla d b_{\gamma}$. Here ${\mathcal S}$ is the shape operator of a horosphere centered at $[\gamma]\in \partial S$. Then we restrict the eigenproblem over ${\mathfrak h}_1$ and write the endomorphism ${\mathcal T}_{\hspace{1mm}\vert{\mathfrak h}_1}$ as a $4\times 4$-matrix given by
\begin{eqnarray*}
B =\left(
\begin{array}{cc|cc}
 a & 0 & c & d \\
 0 & a & - d & c \\ \hline
 c & - d & b &0 \\
 d & c & 0 & b
 \end{array}\right),
\end{eqnarray*}where
\begin{eqnarray*}
\label{abcd}\displaystyle{a = \frac{1}{2}(1+\vert V\vert^2)},\, \displaystyle{b = \frac{1}{2}(1+s^2+\vert Y\vert^2)},\, \displaystyle{c = \frac{1}{2}s\vert V\vert}, \,\displaystyle{d = \frac{1}{2}\vert Y\vert\vert V\vert}.
\end{eqnarray*} 
 Other matrices for ${\mathfrak h}_a$, $a >1$ are the same.
Let $\lambda$ be an eigenvalue of $B$. Then, $\lambda$ satisfies
\begin{eqnarray*}
 (a-\lambda)(b-\lambda) = c^2 + d^2
\end{eqnarray*}
 whose solutions are
\begin{eqnarray*}
\lambda = \frac{1}{2}\big\{a+b \pm \sqrt{(a-b)^2 + 4\{c^2+d^2\}} \big\} = 1,\, \frac{1}{2} .
\end{eqnarray*}
\begin{lemma}\label{eigenvalue}\rm
The eigenvalues of $B$ and hence of ${\mathcal T}$ are $1$ and $\displaystyle{\frac{1}{2}}$.



Let $\lambda = 1/2$ be the eigenvalue of $\nabla d b_{\gamma}$ over ${\mathfrak q}_{\ell}$. The 
eigenvectors corresponding to $\lambda =1/2$ are  vectors of the form
\begin{eqnarray}
U_1= \vert V\vert^2 Z + J_{Z}\left(J_YV- sV  \right),\hspace{2mm}Z\in (Y^{\perp})_J.
\end{eqnarray}
Let $\lambda = 1$ be the eigenvalue. Then, the 
eigenvectors corresponding to $\lambda =1$ are  vectors of the form
\begin{eqnarray}
U_2= \left(\vert V\vert^2-1\right) Z + J_{Z}\left(J_YV- sV  \right),\hspace{2mm}Z\in (Y^{\perp})_J.
\end{eqnarray}
\end{lemma}
\begin{remark}\rm
As shown in \cite{BTV}, the eigenvalues of $R_{V+Y+sA}\vert_{{\mathfrak t}_1\oplus(Y^{\perp})_J}$ and the corresponding eigenspaces are
\begin{eqnarray*}
- \frac{1}{4}&:&\hspace{2mm}\vert V\vert^2 Z + J_{Z}\left(J_YV- sV  \right),\hspace{2mm}Z\in (Y^{\perp})_J,\\ \nonumber
-1&:&\hspace{2mm}\left(\vert V\vert^2-1\right) Z + J_{Z}\left(J_YV- sV  \right),\hspace{2mm}Z\in (Y^{\perp})_J.
\end{eqnarray*}Thus, the eigenspaces of $\nabla d b_{\gamma}$ over ${\mathfrak t}_1\oplus(Y^{\perp})_J$ are exactly same as those of the Jacobi operator.
\end{remark}

As was verified above,  it is concluded that $\nabla d b_{\gamma}$ is positive definite over ${\mathfrak t}_1\oplus(Y^{\perp})_J$.

\subsection{Positive definiteness over ${\mathfrak q}_j$, $j=0,\dots,\ell-1$}
\vspace{2mm}
We will deal in this subsection with the positive definiteness of the Hessian over the subspace ${\mathfrak t}_2\oplus {\mathfrak t}_3 \oplus (Y^{\perp})_J^{\perp} = \oplus_{j=0}^{\ell-1}\ {\mathfrak q}_j$. It suffices from Lemma \ref{orthogonality} to show $\nabla d b_{\gamma}\vert_{{\mathfrak q}_j} > 0$ for each $j = 0,1,\cdots, \ell-1$. Here ${\mathfrak q}_j = {\rm Span}\, \{J_{L_j}{\hat V}, J_{L_j}J_{\hat Y}{\hat V}, L_j\}$.

\vspace{2mm}
Let $j=0$ with $\mu_0 = 0$. Then, $L_0 =\{Z\in (Y^{\perp})_J^{\perp}\,\vert\, KZ=0\}$, since $K^2Z= 0\, \Leftrightarrow\, KZ= 0$ and also $J_{L_0}{\hat V}\perp J_{L_0}J_{\hat Y}{\hat V}$ from (\ref{perpendicular}). Let $\{Z_a,\, a= 1,\cdots, r_0\}$ be an orthonormal basis of $L_0$ so that ${\mathfrak q}_0$ is decomposed orthogonally into
\begin{eqnarray*}
{\mathfrak q}_0 = \oplus_{a=1}^{r_0}\, {\mathfrak f}_a,\hspace{2mm}{\mathfrak f}_a := {\rm Span}\, \{W_a=J_{Z_a}{\hat V},\, W_a'=J_{Z_a}J_{\hat Y}{\hat V}, Z_a\}.
\end{eqnarray*}
It is easily derived similarly as in the proof of Lemma \ref{orthogonality} that $\nabla d b_{\gamma}\vert_{{\mathfrak q}_0} = \oplus_{a=1}^{r_0} \nabla d b_{\gamma}\vert_{{\mathfrak f}_a}$, from Lemma \ref{ljyeperp}.


In what follows we will prove $\nabla d b_{\gamma}\vert_{{\mathfrak f}_a} > 0$. 
With respect to the basis $\{ W_a, W_a', Z_a\}$ of 
${\mathfrak f}_a$, we have the following. 
\begin{lemma}\rm
\begin{eqnarray*}\nabla d b(W_a, W_a) &=& \frac{1}{2} +\frac{1}{2\chi_{\infty}}\vert V\vert^2 (1- s)^2,
\\ \nonumber
\nabla d b(W'_a, W'_a) &=& \frac{1}{2} +\frac{1}{2\chi_{\infty}}\vert V\vert^2\vert Y\vert^2, \\ \nonumber
\nabla d b(Z_a, Z_a) &=& \frac{1}{2} (1+s^2+\vert Y\vert^2), \\ \nonumber
\nabla d b(W_a, Z_a) &=& \frac{1}{2}\vert V\vert s, \,
\nabla d b(W'_a, Z_a) = -\frac{1}{2}\vert V\vert\vert Y\vert, \\ \nonumber
\nabla d b(W_a, W'_a) &=& \frac{1}{2}\vert V\vert^2 \frac{(1-s)\vert Y\vert}{\chi_{\infty}}. \\ \nonumber
\end{eqnarray*}\end{lemma}

Let $w W_a + w' W_a' + z Z_a\in {\mathfrak f}_a$($w,w',z\in {\Bbb R}$) be a non-zero vector. Then
 \begin{eqnarray*}
 & &
\nabla d b_{\gamma}(w W_a + w' W_a' + z Z_a, w W_a + w' W_a' + z Z_a)
\\ \nonumber
&=& \big\{\frac{1}{2} +\frac{1}{2\chi_{\infty}}\vert V\vert^2 (1- s)^2\big\} w^2 + \big\{\frac{1}{2} + \frac{1}{2\chi_{\infty}}\vert V\vert^2\vert Y\vert^2\big\} (w')^2 \\ \nonumber
&+& \frac{1}{2}(1+s^2+\vert Y\vert^2) z^2 + 2\hspace{0.5mm}\frac{1}{2}\vert V\vert^2 \frac{(1-s)\vert Y\vert}{\chi_{\infty}} w\, w' \\ \nonumber
&+& 2\hspace{0.5mm} \frac{1}{2}\ s\ \vert V\vert w\,z + 2\hspace{0.5mm} \frac{1}{2}\left(-\vert Y\vert\ \vert V\vert \right)w'\,z \\ \nonumber
&=& \frac{1}{2}(w^2+(w')^2) +   \frac{1}{2\chi_{\infty}} \vert V\vert^2 \big\{(1-s) w + \vert Y\vert w'\big\}^2
\\ \nonumber
&+& \frac{1}{2}(1+s^2+\vert Y\vert^2) z^2 + 2\, \frac{1}{2}\vert V\vert z\left(s w - \vert Y\vert w'\right)
 \end{eqnarray*}
 which is reduced to
 \begin{eqnarray*}
 & &\frac{1}{2}(w^2+(w')^2) +   \frac{1}{2\chi_{\infty}} \vert V\vert^2 \big\{(1-s) w + \vert Y\vert w'\big\}^2 \\ \nonumber
 &+& \frac{1}{2} z^2 + \frac{1}{2} s^2 z^2 + \frac{1}{2} \vert Y\vert^2 z^2 + 2\, \frac{1}{2}\vert V\vert w\hspace{0.5mm}(s z)  + 2\, \frac{1}{2}\vert V\vert w'\left(- \vert Y\vert z\right) \\ \nonumber
 &=& \frac{1}{2\chi_{\infty}} \vert V\vert^2 \big\{(1-s) w + \vert Y\vert w'\big\}^2 
 + \frac{1}{2}\left(w+ s\vert V\vert z\right)^2\\
 &+& \frac{1}{2}\left(w' - \vert Y\vert\vert V\vert z\right)^2 + \frac{1}{2}\big[1+(s^2 + \vert Y\vert^2)(1-\vert V\vert^2)\big]\ z^2.
 \end{eqnarray*}This is obviously positive, since $V$, $Y$ are non-zero vectors, so
we have the following.
\begin{proposition}\rm $\nabla d b_{\gamma}$ is positive definite over ${\mathfrak q}_0$.
\end{proposition}

\vspace{2mm}For the case\, $0< j < \ell$ with $-1 < \mu_j < 0$\,  define on the subspace ${\mathfrak q}_j$ an almost complex structure $\displaystyle{{\hat K} :=   \frac{1}{\sqrt{\vert \mu_j\vert}} K
}$, Hermitian with respect to the inner product $\langle \cdot,\cdot\rangle$ so that we take an orthonormal basis $\{Z_j\},\, j=1,\cdots, 2r_{j}$ such that
$Z_{2a}:={\hat K}Z_{2a-1}$\, $a=1,\cdots, r_j
$.

Decompose ${\mathfrak q}_j$ into
\begin{eqnarray}
{\mathfrak q}_j &=& \nonumber \oplus_{a=1}^{r_j}\ {\mathfrak b}_a,\hspace{2mm}\dim L_j = 2 r_j,\\ 
\label{nonorthog}
{\mathfrak b}_a &:=& {\rm Span}\, \{W_{2a-1}, W_{2a}, W'_{2a-1}, \, W'_{2a}, Z_{2a-1}, Z_{2a}\},
\end{eqnarray}where
$W_{2a-1}:=J_{Z_{2a-1}}{\hat V}, W_{2a}:=J_{Z_{2a}}{\hat V}, W'_{2a-1}:=J_{Z_{2a-1}}J_{\hat Y}{\hat V}, \, W'_{2a}:=J_{Z_{2a}}J_{\hat Y}{\hat V}$.
Similarly as in case $j=0$ with $\mu_0 = 0$, it is shown that $\nabla d b_{\gamma}\vert_{{\mathfrak q}_j} = \oplus_{a=1}^{r_j} \nabla d b_{\gamma}\vert_{{\mathfrak b}_a}$. Notice that  the basis of (\ref{nonorthog}) is not orthonormal, since  $\displaystyle{\langle W_{2a-1},W'_{2a}\rangle = - \sqrt{\vert \mu_j\vert}}$ from (\ref{notorthogonal}).


\begin{lemma}\rm On ${\mathfrak b}_1= {\rm Span}\,\{W_1, W_2, W_1', W_2', Z_1, Z_2\}$ the components of the Hessian $\nabla d b_{\gamma}$ are
shown below.
\begin{eqnarray*}\nabla d b(W_1, W_1)&=& \nabla d b(W_2, W_2) \\ \nonumber
&=&\frac{1}{2} + \frac{1}{2\chi_{\infty}} \vert V\vert^2 \big\{ (1-s)^2  - \vert Y\vert^2 \mu_i
\big\}, \\ \nonumber
&=& \frac{1}{2}(1+\vert V\vert^2)  - \frac{1}{2\chi_{\infty}} \vert V\vert^2\vert Y\vert^2(1+\mu_i),
\\ \nonumber
\nabla d b(W_1, W_2) &=& 0,
\\ \nonumber
\nabla d b(W_1', W_1') &=& \nabla d b(W_2', W_2')=\frac{1}{2} +\frac{1}{2\chi_{\infty}} \vert V\vert^2 \big\{ \vert Y\vert^2  - (1-s)^2 \mu_i
\big\} \\ \nonumber
&=&\frac{1}{2}(1+\vert V\vert^2)  -\frac{1}{2\chi_{\infty}} \vert V\vert^2(1-s)^2 (1+ \mu_i), \\ \nonumber
\nabla d b(W_1', W_2') &=& 0,
\\ \nonumber
\nabla d b(W_1,W_1') &=& \nabla d b(W_2,W_2') = \frac{1}{2\chi_{\infty}} \vert V\vert^2 (1-s)\vert Y\vert (1 + \mu_i), \\ \nonumber
\nabla d b(W_1,W_2') &=&
-\nabla d b(W_2,W_1') = \frac{1}{2}\langle W_1, W_2'\rangle  + \frac{1}{2} \vert V\vert^2 (-\sqrt{\vert\mu_i\vert})\\ \nonumber
&=& - \frac{1}{2}(1 + \vert V\vert^2)\sqrt{\vert\mu_i\vert}.
\end{eqnarray*}
\begin{eqnarray*}\nabla d b(W_1, Z_1) &=& \nabla d b(W_2, Z_2) =\frac{1}{2} s\  \vert V\vert,
\\ \nonumber
\nabla d b(W_1, Z_2) &=& - \nabla d b(W_2, Z_1)= \frac{1}{2} \vert Y\vert\vert V\vert \sqrt{\vert \mu_i\vert },
\\ \nonumber
\nabla d b(W_1', Z_1) &=&\nabla d b(W_2', Z_2) = - \frac{1}{2} \vert Y\vert\vert V\vert, \\ \nonumber
\nabla d b(W_1', Z_2) &=&-\nabla d b(W_2', Z_1)= \frac{1}{2} s\, \vert V\vert \sqrt{\vert \mu_i\vert }
, 
\\ \nonumber
\nabla d b(Z_i, Z_j) &=& \frac{1}{2} (1+s^2+ \vert Y\vert^2) \delta_{ij},\,  i,j=1,2.
\end{eqnarray*}
\end{lemma}
\begin{note}\rm
From this lemma
the Hessian $\nabla d b_{\gamma}\vert_{{\mathfrak q}_j}$ is Hermitian with respect to ${\hat K}$.
\end{note}

Now we show the positive definiteness of $\nabla d b_{\gamma}$ on ${\mathfrak b}_1$.
Let $W = u_1 W_1 + u_2 W_2 + v_1 W_1' + v_2 W_2' + w_1 Z_1 + w_2 Z_2 \in {\mathfrak b}_1$ be a non zero vector. Then, $\nabla d b_{\gamma}(W,W)$ is written in terms of a quadratic form as
\begin{eqnarray*}\nabla d b_{\gamma}(W,W)
&=&a (u_1^2 +  u_2^2) + b(v_1^2+v_2^2) + t(w_1^2+w_2^2)\\ \nonumber
&+& 2 c(u_1v_1+u_2v_2) + 2 d(u_1v_2-u_2v_1)\\ \nonumber
&+&2 f(u_1w_1+u_2w_2) + 2 g(u_1w_2-u_2w_1)\\ \nonumber
&+&2 h(v_1w_1+v_2w_2) + 2 \ell (v_1w_2-v_2w_1),
 \end{eqnarray*} where
 \begin{eqnarray*}
 a &=& \frac{1}{2}+ \frac{1}{2\chi_{\infty}}\vert V\vert^2\{(1-s)^2 + \vert Y\vert^2\cos^2\theta_i\}, \\ \nonumber
 b &=& \frac{1}{2}+ \frac{1}{2\chi_{\infty}}\vert V\vert^2\{\vert Y\vert^2 + (1-s)^2\cos^2\theta_i\}, \\ \nonumber
 c&=& \frac{1}{2\chi_{\infty}}\vert V\vert^2(1-s)\vert Y\vert \sin^2\theta_i, \, d= - \frac{1}{2}(1+\vert V\vert^2) \cos \theta_i, \\ \nonumber
  f &=& \frac{1}{2} s \vert V\vert,\hspace{2mm}  g= \frac{1}{2} \vert Y\vert\vert V\vert \cos \theta_i,\\ \nonumber
 h&=& - \frac{1}{2} \vert Y\vert\vert V\vert, \hspace{2mm}\ell = \frac{1}{2} s \vert V\vert \cos \theta_i , \\ \nonumber
 t&=& \frac{1}{2}(1+s^2+\vert Y\vert^2).
 \end{eqnarray*}Set, for convenience, $\sqrt{\vert \mu_i\vert} = \cos \theta_i$, $0<\theta_i<\frac{\pi}{2}$, so $\mu_i = - \cos^2\theta_i$. Then, the above quadratic form can be described by the Hermitian form in terms of a Hermitian matrix $H$;
 \begin{eqnarray*}
 \nabla d b(W,W) &=& ({\overline u},{\overline v},{\overline w})H\hspace{1mm}^t(u,v,w), \, {\rm where}\\ \nonumber
 H&=&\left(\begin{array}{ccc}
 a&c-id&f-ig\\
 c+id&b&h-i\ell\\
 f+ig&h+i\ell&t
 \end{array}
 \right),
  \end{eqnarray*}$u= u_1+iu_2, v=v_1+iv_2,w=w_1+iw_2 \in {\Bbb C}$. In order to assert the positive definiteness of the Hessian it suffices to show $\det H^{(2)} =\displaystyle{\left\vert\begin{array}{cc}a&c -id\\ c+id& b\end{array}\right\vert > 0}$ and  $\det H^{(3)} = \det H>0$, since $a = \det H^{(1)} > 0$.
  \begin{lemma}\label{positivehermitematrix}\rm
  \begin{eqnarray*}
  \det H^{(2)}&=& \frac{1}{4}(1+\vert V\vert^2)\sin^2\theta_i,\\ \nonumber
  \det H^{(3)} &=& \frac{1}{4}\sin^2\theta_i - \frac{\vert V\vert^4\vert Y\vert^2}{8\chi_{\infty}} \sin^4\theta_i
  > \frac{1}{8\chi_{\infty}}\sin^4 \theta_i\left(2 \chi_{\infty}-\vert V\vert^4\vert Y\vert^2\right).
  \end{eqnarray*}
 \end{lemma}
 \begin{proof}
 \rm
 One sees
 \begin{eqnarray*}\det H^{(2)} =
 ab - \vert c +id\vert^2
= \frac{1}{4}(1+\vert V\vert^2)\sin^2\theta_i.
  \end{eqnarray*}
 For $\det H^{(3)}$ one has
 \begin{eqnarray*}
 \det H^{(3)} = \det H^{(2)} t - a \vert h+i\ell\vert^2 - b\vert f+ig\vert^2
 + 2 Re(c+id)(h+i\ell)(f-ig).
 \end{eqnarray*}By straight computations one obtains
 \begin{eqnarray*}
 & &a \vert h+i\ell\vert^2 + b\vert f+ig\vert^2 \\ \nonumber
 &=& \frac{\vert V\vert^2}{8}(s^2+\vert Y\vert^2)(1+\cos^2\theta_i)   \\ \nonumber
 &+& \frac{\vert V\vert^4}{8\chi_{\infty}}\big\{\vert Y\vert^2\big(s^2+(1-s)^2\big)(1+\cos^4\theta_i)\\ \nonumber
 &+& 2\big(\vert Y\vert^4+s^2(1-s)^2\big)\cos^2\theta_i\big\}
 \end{eqnarray*}and

 \begin{eqnarray*}
 Re (c+id)(f-ig)(h+i\ell)
 &=& \frac{\vert V\vert^2}{4}\big\{-\frac{\vert V\vert^2}{2\chi_{\infty}}
 s(1-s)\vert Y\vert^2 \sin^4\theta_i \\ \nonumber
 &+& \frac{1}{2}(1+\vert V\vert^2)(s^2+\vert Y\vert^2)\cos^2\theta_i\big\}.
 \end{eqnarray*}Thus, one has
 \begin{eqnarray*}
 \det H^{(3)}&=&\frac{1}{8}(1+\vert V\vert^2)\sin^2\theta_i\, (1+s^2+\vert Y\vert^2) - \frac{\vert V\vert^2}{8}(s^2+\vert Y\vert^2)(1+\cos^2\theta_i) \\ \nonumber
 &-& \frac{\vert V\vert^4}{8\chi_{\infty}}\big\{\vert Y\vert^2\big(s^2+(1-s)^2\big)(1+\cos^4\theta_i)
 + 2\big(\vert Y\vert^4+s^2(1-s)^2\big)\cos^2\theta_i\big\}\\ \nonumber
 &+& 2 \frac{\vert V\vert^2}{4}\big\{-\frac{\vert V\vert^2}{2\chi_{\infty}}
 s(1-s)\vert Y\vert^2 \sin^4\theta_i 
 + \frac{1}{2}(1+\vert V\vert^2)(s^2+\vert Y\vert^2)\cos^2\theta_i\big\}.
 \end{eqnarray*}

 Then,  a slight computation gives us 
\begin{eqnarray*}
\det H^{(3)}
&= &\left(\frac{1}{4}- \frac{1}{4}\cos^2\theta_i - \frac{\vert V\vert^2}{4}(s^2+\vert Y\vert^2)\cos^2\theta_i\right)\\ \nonumber
&+&\left(-\frac{\vert V\vert^4}{8\chi_{\infty}}\vert Y\vert^2(1+\cos^4\theta_i) + \frac{\vert V\vert^4}{4\chi_{\infty}}\vert Y\vert^2\cos^2\theta_i  \right)\\ \nonumber
&+& \frac{\vert V\vert^2}{4}(s^2+\vert Y\vert^2)\cos^2\theta_i \\ \nonumber
&=& \frac{1}{4}- \frac{1}{4}\cos^2\theta_i - \frac{\vert V\vert^4\vert Y\vert^2}{8\chi_{\infty}}(1-\cos^2\theta_i)^2\\ \nonumber
&=& \frac{1}{4}\sin^2\theta_i - \frac{\vert V\vert^4\vert Y\vert^2}{8\chi_{\infty}}\sin^4\theta_i.
\end{eqnarray*}

 \end{proof}

  Summarizing the arguments from subsection 4.5 we can conclude that the Hessian of the Busemann function associated to $\gamma$ is positive definite, provided $V\not=0, Y\not=0$ for $\gamma'(0)=V+Y+sA $.

\vspace{2mm}
     We will give an argument for $\nabla d b_{\gamma} > 0$ in the following non-generic cases; (i) $V\not=0$, $Y=0$, (ii) $V=0$, $Y\not=0$ and (iii) $V = 0$, $Y=0$, i.e., $s = -1$.

         \vspace{2mm}
     Case: $V\not=0$, $Y=0$; Let $\gamma$ be the geodesic of $\gamma(0) = e_s$, $\gamma'(0)=V+sA = \vert V\vert{\hat V}+ sA$, $\vert V\vert^2+s^2=1$.
     Let $W_1:= \vert V\vert{\hat V}+ sA$ and $W_2:= s{\hat V}- \vert V\vert A$. Then, one has $\displaystyle{{\mathfrak s} = {\mathfrak s}_2\oplus {\mathfrak p}\oplus {\mathfrak q}\oplus {\mathfrak z}     }$, $\displaystyle{{\mathfrak s}_2 := {\rm Span}\,\{W_1^{\mathfrak s},W_2^{\mathfrak s}\} }$, ${\mathfrak p} := \{U\in {\mathfrak v}\, \vert\, [V,U]=0, \langle U,V\rangle = 0\}$, ${\mathfrak q}:= J_{\mathfrak z}{\hat V}$ so $\gamma'(0)^{\perp} = {\Bbb R} W_2^{\mathfrak s}\oplus {\mathfrak p} \oplus {\mathfrak q}\oplus {\mathfrak z}$.

     It is shown by a simple computation that any vector of
    $ {\Bbb R} W_2^{\mathfrak s}\oplus {\mathfrak p} $ is an eigenvector corresponding to eigenvalue $\frac{1}{2}$.

    The Hessian $\nabla d b_{\gamma}$ is positive definite over ${\mathfrak q}\oplus{\mathfrak z}$. In fact, ${\mathfrak q}\oplus{\mathfrak z}$ splits into ${\mathfrak q}\oplus{\mathfrak z} = \oplus_{a=1}^k {\mathfrak b}_a$, ${\mathfrak b}_a := {\rm Span}\, \{J_{Z_a}{\hat V}, Z_a\}$ and fulfills that $\nabla d b_{\gamma}({\mathfrak b}_a,{\mathfrak b}_b) = 0$ for any distinct $a, b$ and $\nabla d b_{\gamma}\vert_{{\mathfrak b}_a} > 0$, $a=1,\dots, k$. Thus, the Hessian $\nabla d b_{\gamma}$ is positive definite over $\gamma'(0)^{\perp}$.

    \vspace{2mm}
    Case: $V=0$, $Y\not=0$;  Let $\gamma$ be the geodesic of $\gamma(0) = e_S$, $\gamma'(0) = Y + sA =\vert Y\vert {\hat Y}+ sA$, $\vert Y\vert^2+s^2=1$. It is concluded that
   on the space $\gamma(0)^{\perp} = {\mathfrak v}\oplus {\Bbb R}(s{\hat Y}- \vert Y\vert A)\oplus Y^{\perp}$ the Hessian $\nabla d b_{\gamma}$ is positive definite. Moreover, ${\mathfrak v}$ is the eigenspace corresponding to eigenvalue $\displaystyle{\frac{1}{2}}$ and $Y^{\perp}$ is the eigenspace corresponding to eigenvalue $1$.

    \vspace{2mm}Case: $V =0$, $Y=0$, i.e., $s = -1$;  In this case one can apply the formulae of Lemma \ref{componenthessian} to obtain that the Hessian $\nabla d b_{\gamma}$ is positive definite over $(-A)^{\perp} = {\mathfrak v}\oplus{\mathfrak z}$, since at $e_S$ ${\mathcal V}={\mathcal Y} = 0$.  The space ${\mathfrak v}$ is the eigenspace corresponding to eigenvalue $\displaystyle{\frac{1}{2}}$ and ${\mathfrak z}$ is the eigenspace corresponding to eigenvalue $1$.

    Therefore, from the above arguments Theorem \ref{hessianatidentity} is completely verified.

\subsection{Spectral properties of $\nabla d b_{\gamma}$}

\vspace{2mm}As a by-product of Theorem \ref{hessianatidentity} 
we can
compare the eigenspaces of $(\nabla d b_{\gamma})_{e_S}$ corresponding to the eigenvalues with the eigenspaces of $R_{\gamma'(0)}$ corresponding to the eigenvalues. Let ${\mathfrak f}:= {\mathfrak s}_4^0\oplus {\mathfrak p}\oplus {\mathfrak t}_1  \oplus (Y^{\perp})_J$ be the subspace of $\gamma'(0)^{\perp}$. Then, we can assert the commutativity of ${\mathcal S}(t)\vert_{\mathfrak f}$ with $R_{\gamma'(t)}\vert_{\mathfrak f}$ at $t=0$. Moreover we can show that the commutativity is valid for all $t$ as follows.
 \begin{theorem}\label{eigenproblemcommute}\rm
 Let $V+Y+sA\in {\mathfrak s}$ be a unit vector. Assume $V\not=0$, $Y\not=0$, a generic assumption on $V$, $Y$. Let ${\mathfrak s} = {\mathfrak s}_4\oplus {\mathfrak p}\oplus {\mathfrak t}_1 \oplus {\mathfrak t}_2 \oplus {\mathfrak t}_3 \oplus (Y^{\perp})_J\oplus (Y^{\perp})_J^{\perp}
$ be an admissible decomposition of ${\mathfrak s}$ associated with $V+Y+sA$.

Then,
 (i)\hspace{2mm}on the subspace ${\mathfrak f}$ 
  the shape operator ${\mathcal S}(t)$ of the horosphere ${\mathcal H}_{(\theta, \gamma(t))}$
 commutes with $R_{\gamma'(t)}$ ($\gamma$ is a geodesic of $\gamma(0) = e_S$, $\gamma'(0)=V+Y+sA$, $[\gamma]=\theta$) so that ${\mathcal S}(t)\vert_{\mathfrak f}$ shares common eigenspaces with $R_{\gamma'(t)}\vert_{\mathfrak f}$ and
  (ii)\hspace{2mm}each eigenvalue of ${\mathcal S}(t)\vert_{\mathfrak f}$ is negative constant and moreover
 (iii)\hspace{2mm}if $\lambda <0$ is an eigenvalue of  ${\mathcal S}(t)\vert_{\mathfrak f}$, then $- \lambda^2$ is an eigenvalue of $R_{\gamma'(t)}\vert_{\mathfrak f}$  and vice versa.
\end{theorem}

In fact, from (\ref{geodesicdecompo}) the vectors $V,\, Y$ for $\gamma'(0)=V+Y+sA$ are generic if and only if $V(t)$,$Y(t)$ for $\gamma'(t)=V(t)+Y(t)+A(t)$ are generic and we have moreover $K_{V(t),Y(t)}= K_{V,Y}$ as shown in 4.3,\hspace{0.5mm}\cite{BTV}.
It can be checked further that the admissible decomposition of ${\mathfrak s}$ at $t=0$ gives also admissible one at any $t$. For the detailed argument for all $t$ refer to \cite{IKPS-2}.

 \begin{remark}\rm
 The eigenvalues and eigenspaces of $R_{V+Y+sA}$ are given completely in \cite{BTV}. On the subspace ${\mathfrak f}$ one has

 (1)\hspace{2mm}the eigenvalues and eigenspaces of $R_{V+Y+sA}\vert_{{\mathfrak s}_4^0}$ are
 \begin{eqnarray*}
  0, & &\hspace{2mm}{\Bbb R}(V+Y+sA),\\ \nonumber
  -\frac{1}{4}, & &\hspace{2mm}{\Bbb R}(sV- J_YV -\vert V\vert^2)A)\oplus {\Bbb R}(\vert Y\vert^2 V+ s J_YV - \vert V\vert^2 Y),\\ \nonumber
  -1, & &\hspace{2mm}{\Bbb R}(J_YV + s Y - \vert Y\vert^2 A).
   \end{eqnarray*}

    (2)\hspace{2mm}$R_{V+Y+sA}\vert_{{\mathfrak p}}$ has a single eigenvalue, that is $- \frac{1}{4}$, provided ${\mathfrak p} \not= \{0\}$.

   (3)\hspace{2mm} If ${\mathfrak q} \not= \{0\}$ and further $(Y^{\perp})_J  \not= \{0\}$, the eigenvalues and eigenspaces of $R_{V+Y+sA}\vert_{{\mathfrak t}_1\oplus (Y^{\perp})_J}$ are
      \begin{eqnarray*}
   - \frac{1}{4},& &\hspace{2mm}\vert V\vert^2 X + J_X(J_Y V - s V),\hspace{4mm} X\in (Y^{\perp})_J, \\ \nonumber
      - 1, & &\hspace{2mm}(\vert V\vert^2-1) X + J_X(J_YV - sV),\hspace{4mm} X \in (Y^{\perp})_J.
   \end{eqnarray*}
\end{remark}

  \section{Appendix I; another proof of Theorem   \ref{damekricci-2}}

  \begin{proposition}\rm
\label{rank}\rm
Let $S$ be a Damek-Ricci space. Let $\gamma$ be an arbitray geodesic of $S$. Then, there exists no non-trivial perpendicular, parallel Jacobi vector field 
 along $\gamma$.
\end{proposition}
As a direct consequence we obtain
\begin{corollary}\rm
Let $S$ be a Damek-Ricci space. Then every geodesic of $S$ is of rank one. 
\end{corollary}
To verify the above proposition we present the sectional
curvature formula as
\begin{proposition}\label{planesection}
\rm Let $U + X + r A$, $V + Y$ be vectors in $\mathfrak{s}$ of unit
norm, perpendicular to each other. Then, the sectional curvature
$K(P)$ of a plane section $P = \{ U+X+r A, V+Y\}$ is represented by
\begin{eqnarray*}
K(P) &=& - \frac{3}{4}\ \left\vert [U,V]+ r Y   \right\vert^2 - \frac{1}{4} r^2 \vert V\vert^2 - \frac{1}{4} r^2 \vert Y\vert^2 \\ \nonumber
&+& \frac{1}{4} \vert V \vert^2 \vert X \vert^2 + \frac{1}{4} \vert U \vert^2 \vert Y \vert^2 - \langle J_XU,J_YV\rangle + \frac{1}{2} \langle J_XV,J_YU\rangle  \\ \nonumber
&-& \Big(\frac{1}{2} \vert X\vert^2 \vert V\vert^2 +
\frac{1}{2} \vert Y\vert^2 \vert U\vert^2 +\frac{1}{4} \vert U\vert^2 \vert V\vert^2 + \vert X\vert^2 \vert Y\vert^2 \\ \nonumber
&-& \langle U,V\rangle \langle X,Y\rangle - \frac{1}{4}\langle U,V\rangle^2 - \langle X,Y\rangle^2 \Big).
\end{eqnarray*}
\end{proposition}This formula is obtained from the formula given in \cite{BTV}, p. 85. We remark that the metric of a Damek-Ricci space is left invariant. Any formulae represented by using  left invariant vector
fields are valid at any point of $S$.

\vspace{2mm}
\begin{lemma}\label{nonpositivecurvature}\rm
Let $P$ be a plane section at $e_S$ 
 spanned by orthonormal vectors $\{ X_1 =
aA+ b (U_1 + Y_1), X_2 = U_2 + Y_2 \}$,  $a, b\in {\Bbb R}$, $U_1, U_2 \in \mathfrak{v}$, $Y_1, Y_2\in \mathfrak{z}$.
 Then 
\begin{eqnarray}\label{sectionalformula}
K(P) &=& - \frac{a^2}{4} - \frac{3}{4} \big\vert a Y_2 +b[U_1,U_2] \big\vert^2 \\ \nonumber
&-&\frac{3}{4}\, b^2 \langle Y_1,Y_2\rangle^2 - \frac{3}{4}\, b^2 T(P),
\end{eqnarray}where
\begin{eqnarray}\label{t}
T(P) := \vert Y_1\vert^2 \vert Y_2\vert^2 + 2 \langle J_{Y_1}U_1, J_{Y_2}U_2\rangle + \frac{1}{3},
\end{eqnarray}which fulfills
\begin{eqnarray}\label{t1}
T(P) &\geq& \vert Y_1\vert^2 \vert Y_2\vert^2 - 2 \vert Y_1\vert \vert Y_2\vert \vert U_1\vert \vert U_2\vert + \frac{1}{3}\\ \nonumber
&\geq& \big(\sqrt{3} \vert Y_1\vert \vert Y_2\vert - \frac{1}{\sqrt{3}}\big)^2 \geq 0
\end{eqnarray}from which it follows $K(P) \leq 0$.
\end{lemma}
\begin{proof}\rm
To obtain (\ref{sectionalformula}), (\ref{t}) 
we set $U = b U_1$, $X = b Y_1$, $r = a$, $V = U_2$ and $Y = Y_2$ in Proposition \ref{planesection} to have
\begin{eqnarray*}
K(P) &=& - \frac{3}{4}\vert [b U_1,U_2] + a Y_2\vert^2 - \frac{1}{4}a^2\vert U_2\vert^2 - \frac{1}{4}a^2\vert Y_2\vert^2 \\ \nonumber
 &+& \frac{1}{4}\vert U_2\vert^2\vert b Y_1\vert^2 + \frac{1}{4}\vert b U_1\vert^2 \vert Y_2\vert^2 \\ \nonumber
 &-& \langle J_{bY_1} bU_1, J_{Y_2} U_2\rangle + \frac{1}{2}  \langle J_{bY_1} U_2, J_{Y_2} bU_1\rangle \\ \nonumber
 &-&\big( \frac{1}{2}\vert bY_1\vert^2\vert U_2\vert^2 + \frac{1}{2}\vert Y_2\vert^2\vert b U_1\vert^2 + \frac{1}{4}\vert bU_1\vert^2\vert U_2\vert^2 + \vert bY_1\vert^2\vert Y_2\vert^2 \\ \nonumber
 &-& \langle bU_1,U_2\rangle \langle bY_1,Y_2\rangle - \frac{1}{4} \langle bU_1,U_2\rangle^2 - \langle bY_1,Y_2\rangle^2\big).
\end{eqnarray*}For this we make use of (\ref{algebrastr-a})
and the orthonormality of $U_i + Y_i$, $i = 1,2$, in particular  $\displaystyle{\langle U_1,U_2\rangle + \langle Y_1,Y_2\rangle  = 0}$ to obtain (\ref{sectionalformula}), (\ref{t}).
We show the second inequality of (\ref{t1}) as follows.  It suffices for this to show
\begin{eqnarray}
 \vert Y_1\vert^2 \vert Y_2\vert^2 - 2 \vert Y_1\vert \vert Y_2\vert \vert U_1\vert \vert U_2\vert + \frac{1}{3} \geq 3 \vert Y_1\vert^2 \vert Y_2\vert^2 - 2 \vert Y_1\vert \vert Y_2\vert + \frac{1}{3}
\end{eqnarray}and hence suffices to show
\begin{eqnarray}\label{inequality-1}
-\vert Y_1\vert^2 \vert Y_2\vert^2 + \vert Y_1\vert\vert Y_2\vert \geq \vert Y_1\vert\vert Y_2\vert \vert U_1\vert\vert U_2\vert.
\end{eqnarray}Let $A_1$, $A_2$ be the left, right hand side of (\ref{t1}), respectively.  Notice $A_1 > 0$, because
$\displaystyle{\vert Y_1\vert\vert Y_2\vert(1-\vert Y_1\vert\vert Y_2\vert) > 0}$. Therefore, we may compute $A_1^2 - A_2^2$ as
\begin{eqnarray*}
A_1^2 - A_2^2 &=& \big(\vert Y_1\vert^4 \vert Y_2\vert^4 - 2
\vert Y_1\vert^3\vert Y_2\vert^3 + \vert Y_1\vert^2\vert Y_2\vert^2\big) - \vert Y_1\vert^2\vert Y_2\vert^2 \vert U_1\vert^2\vert U_2\vert^2\\ \nonumber
&=& \vert Y_1\vert^2 \vert Y_2\vert^2 \big\{\vert Y_1\vert^2 \vert Y_2\vert^2 - 2 \vert Y_1\vert \vert Y_2\vert + 1 -(1-\vert Y_1\vert^2)(1-\vert Y_2\vert^2)   \big\}\\ \nonumber
&=& \vert Y_1\vert^2 \vert Y_2\vert^2 \big(\vert Y_1\vert - \vert Y_2\vert  \big)^2 \geq 0
\end{eqnarray*}which implies that the
sectional curvature is represented by a sum of four
non-positive terms and thus Lemma \ref{nonpositivecurvature} is proved. Refer to \cite{Boggino} for this lemma.

\end{proof}
\begin{lemma}\label{condflatsection}\rm Let $P$ be a plane section given in Lemma \ref{nonpositivecurvature}. Then, $K(P) = 0$ if and only if $a = 0$, $b = 1$ and
\begin{eqnarray}[U_1,U_2] = 0, \hspace{2mm}\langle U_1,U_2 \rangle = 0,\hspace{2mm}\langle Y_1,Y_2 \rangle = 0\hspace{2mm}{\rm and}\hspace{2mm}
J_{Y_1} U_1 = -J_{Y_2} U_2,
\end{eqnarray}
and moreover
$\displaystyle{\vert U_i\vert = \frac{\sqrt{2}}{\sqrt{3}}}$ and $\displaystyle{\vert Y_i\vert = \frac{1}{\sqrt{3}}
}$, $i = 1,2$.
\end{lemma}
Formula (\ref{sectionalformula}) indicates that $K(P) = 0$ implies $a = 0$ and hence $b = 1$, $[U_1,U_2] = 0$, $\langle Y_1,Y_2\rangle = 0$ and $T(P) = 0$ and vice versa.
 Then, the lemma is obtained.


\begin{proof} \rm of Proposition \ref{rank}. Let $\gamma$ be a geodesic of $S$ of $\gamma(0) = e_S,\, \gamma'(0) = s A + b(V+Y)$, $s^2 + b^2 =1$, $\vert V + Y\vert = 1$. Then,  we can assert that there exists no plane section field  along $\gamma$ whose sectional curvature is zero for any $t$. In fact, if, contrarily there exists a flat plane section field $P(t)$, generated by an orthonormal basis\, $\{\gamma'(t), \xi(t)\}$, where $\xi(t) = U(t) + Z(t) + a(t)A$ is a vector field along $\gamma$, perpendicular and $\vert \xi(t)\vert = 1$, then from Lemma \ref{geod} we have
\begin{eqnarray}\label{velocity}
\gamma'(t) &=& \frac{\sqrt{h(t)}}{\chi(t)}
\big[(1-s\theta(t))^2 - \theta^2(t) \vert Y\vert^2  \big] \vert V\vert {\hat V} \\ \nonumber
&+& \frac{2 \sqrt{h(t)}}{\chi(t)} \theta(t)\cdot (1-s\theta(t))\vert V\vert\vert Y\vert J_{\hat Y}{\hat V} \\ \nonumber
&+& h(t) \vert Y\vert {\hat Y} + \big(\log h(t)\big)' A\\ \nonumber
\end{eqnarray}
 Refer to (Theorem 2, p. 94, \cite{BTV}), for more detail.

    Since $K(P(t)) = 0$ at $t = 0$, Lemma \ref{condflatsection} implies $s = 0$ and hence $b = 1$ so that  the $A$-component of (\ref{velocity}) is 
$\big(\log h(t)\big)' $,
 where by setting $s=0$ \begin{eqnarray}h(t) = \frac{1-\theta^2(t)}{\chi(t)}, \hspace{2mm}
 \chi(t) = 1 + \theta^2(t) \vert Y\vert^2.\end{eqnarray} Since we assume $K(P(t)) = 0$ for any $t$, it follows from Lemma \ref{condflatsection} that
the component $\displaystyle{\left(\log h(t)\right)' }= \frac{h'(t)}{h(t)}$ must vanish for every $t$. 
However, 
 $(\log h(t))' = 0$ only when $t = 0$, since $\theta(t) = \tanh t/2$. This is a contradiction.
\end{proof}

 \section{Appendix II; A proof of Proposition   \ref{decompositionofs}}

 Proposition \ref{decompositionofs} is verified as follows.  Since
 \begin{eqnarray*}
 {\mathfrak s}_4 \oplus {\mathfrak p}\oplus {\mathfrak q} \subset {\mathfrak s},\hspace{2mm} {\mathfrak q} = {\mathfrak t} \oplus Y^{\perp}
 \end{eqnarray*}and \begin{eqnarray*}
 {\rm Span}\, \{ V, J_YV\} \oplus {\mathfrak p}\oplus {\mathfrak t} \subset {\mathfrak v},
  \end{eqnarray*}
   it suffices to show
  \begin{eqnarray}  \dim {\mathfrak v} \leq
 2 +\dim{\mathfrak p} + \dim{\mathfrak t} .
 \end{eqnarray}
   For this, we make use of the orthogonal decomposition
 defined at (\ref{J2condition-1}) 
 \begin{eqnarray*}
  Y^{\perp} = (Y^{\perp})_J \oplus (Y^{\perp})_J^{\perp},
 \end{eqnarray*}
 together also with the direct sum   decomposition given at (\ref{t-1}) and (\ref{t2t3})
 \begin{eqnarray}
 {\mathfrak t} = {\mathfrak t}_1 \oplus {\mathfrak t}_2\oplus {\mathfrak t}_3=J_{ (Y^{\perp})_J}V \oplus J_{ (Y^{\perp})_J^{\perp}}V \oplus J_{ (Y^{\perp})_J^{\perp}}J_YV,
 \end{eqnarray}
 \begin{eqnarray}\label{dimensioform}
 \dim {\mathfrak t} &=&  k_1 + 2 k_2,\\ \nonumber
 k_1 &:=& \dim (Y^{\perp})_J = \dim J_{(Y^{\perp})_J}V,
\\ \nonumber
k_2 &:=& \dim (Y^{\perp})_J^{\perp} = \dim J_{(Y^{\perp})_J^{\perp}}V = J_{(Y^{\perp})_J^{\perp}}J_YV.
 \end{eqnarray}

 We have certainly the orthogonal decomposition relative to the vector $V$ as
 $\displaystyle{ {\mathfrak v} = {\rm Ker}\hspace{0.5mm} ad(V) \oplus J_{\mathfrak z}V}$,
where $J_{\mathfrak z}V =\{ J_ZV\, \vert\, Z\in {\mathfrak z}\} 
 ={\rm Ker}\hspace{0.5mm} ad(V)^{\perp}$ and have, replacing $V$ by $J_YV \not=0$, another decomposition as
$\displaystyle{
  {\mathfrak v} = {\rm Ker}\hspace{0.5mm} ad(J_YV) \oplus J_{\mathfrak z}J_YV}$.
Decompose ${\mathfrak v}$ into
 \begin{eqnarray}
  {\mathfrak v} = {\rm Ker}\hspace{0.5mm} ad(V) \oplus J_{\mathfrak z}V = {\Bbb R}V \oplus {\Bbb R}J_YV \oplus {\mathfrak p} \oplus {\mathfrak p}_1 \oplus J_{Y^{\perp}}V
  \end{eqnarray}and similarly into
    \begin{eqnarray}
  {\mathfrak v} = {\rm Ker}\hspace{0.5mm} ad(J_YV) \oplus J_{\mathfrak z}J_YV = {\Bbb R}V \oplus {\Bbb R}J_YV \oplus {\mathfrak p} \oplus {\mathfrak p}_2 \oplus J_{Y^{\perp}}J_YV,
  \end{eqnarray}where ${\mathfrak p}_i$, $i = 1,2$ are the orthogonal complement of  ${\mathfrak p}$ in ${\rm Ker}\hspace{0.5mm} ad(V)$ and ${\rm Ker}\hspace{0.5mm} ad(J_YV)$, respectively.  

   \begin{claim}\rm The map
  \begin{eqnarray}h:\hspace{2mm}{\mathfrak p}_1 & \, &\rightarrow (Y^{\perp})_J^{\perp}; \\ \nonumber
  U &\, &\mapsto  h(U) := \frac{1}{\vert V\vert^2 \vert Y\vert^2} [J_YV,U]
  \end{eqnarray}  is injective.
  \end{claim}

  \begin{proof}\rm From the above orthogonal decompositions we have
 \begin{eqnarray}
  {\mathfrak p}_1 \oplus J_{Y^{\perp}}V ={\mathfrak p}_2 \oplus J_{(Y^{\perp})_J}J_YV\oplus J_{(Y^{\perp})_J^{\perp}}J_YV,
    \end{eqnarray}since $Y^{\perp} = (Y^{\perp})_J \oplus (Y^{\perp})_J^{\perp}$. Then, any $U\in {\mathfrak p}_1$ is written as
       \begin{eqnarray}\label{sumformula}
    U = U_1 + U_2' + U_2'',\hspace{2mm} U_1 \in {\mathfrak p}_2,\, U_2' \in J_{(Y^{\perp})_J}J_YV, U_2'' \in J_{(Y^{\perp})_J^{\perp}}J_YV.
    \end{eqnarray}Hence, there exists a $Z\in (Y^{\perp})_J$ such that $U_2' = J_ZJ_YV$. Here $U_2'$ is written also as $U_2' = J_{Z'}V$ for a $Z'\in (Y^{\perp})_J$ from definition of $(Y^{\perp})_J$. Moreover, $U'' = J_{Z''}J_YV$ for a $Z''\in (Y^{\perp})_J^{\perp}$. Take an inner product of (\ref{sumformula}) with the vector $U_2' = J_{Z'}V$. Then $\langle U, U_2'\rangle = \langle U, J_{Z'}V\rangle = 0$, since ${\mathfrak p}_1 \perp J_{Y^{\perp}} V$. On the other hand $\langle U_1, J_{Z'}V\rangle = \langle U_1, J_{Z}J_YV\rangle =0$, since ${\mathfrak p}_2 \perp J_{Y^{\perp}}J_YV$. Also we have $\langle U_2'', U_2' \rangle =0$ from (\ref{algebrastr-a}) and consequently $\vert U_2'\vert^2=0$, namely $U_2' = 0$.
      Therefore $U$ is written as $U = U_1 + U_2''$.

    Since ${\mathfrak p}_2 \subset {\rm Ker}\, ad(J_YV)$, \begin{eqnarray}[J_YV,U] &=& [J_YV, U_1 + U_2''] 
  = [J_YV,  U_2''] \\ \nonumber
  &=& [J_YV, J_{Z''}J_YV] = \vert J_YV\vert^2 Z'' =\vert Y\vert^2\vert V\vert^2 Z''
  \end{eqnarray} 
  and henceforth
  $\displaystyle{Z'' = \frac{1}{\vert V\vert^2 \vert Y\vert^2}[J_YV, U]}$ which is just $h(U) $.

  Now suppose $h(U) = Z''$ is zero vector. Then, $[J_YV,U] = 0$, that is, $U \in {\rm Ker}\, ad(J_YV)$. By the way, $U\in{\mathfrak p}_1\subset {\rm Ker}\, ad(V)$ so that $U \in {\mathfrak p}={\rm Ker} ad(V)\cap {\rm Ker} ad(J_YV)$ and consequently $U = 0$, since ${\mathfrak p}\cap {\mathfrak p}_1 = \{0\}$.
  \end{proof}

  From this lemma $p_1 = \dim {\mathfrak p}_1 \leq  k_2 = \dim J_{(Y^{\perp})_J^{\perp}}J_YV$ so that from (\ref{dimensioform})
     \begin{eqnarray}\dim {\mathfrak v} &=&  2 + p + p_1  + \dim (J_{Y^{\perp}}V)
  = 2 + p+ p_1 + k_1 + k_2 \\ \nonumber
  &\leq& 2 + p + k_1 + 2 k_2 \\ \nonumber
  &=& 2 + p + \dim (J_{Y^{\perp}}V \oplus J_{Y^{\perp}}J_YV) = 2 + p + \dim{\mathfrak t}.
\end{eqnarray}
Hence the desired decomposition is obtained.

\vspace{3mm}
\begin{center}{\sc Part II}
\end{center}

\section{Rank of geodesics and eigenspace  corresponding to eigenvalue 0}
In this section, we prove Theorem \ref{positivehessian}.

Let $x \in M$ and $\theta\in \partial M$. Let $\gamma$ be a geodesic of $M$ satisfying $\gamma(0) = x$, representing $\theta$ and $b= b_{\gamma}$ be the Busemann function associated with $\gamma$.

  First we show
  \begin{eqnarray}\label{inequality-1}\displaystyle{\rm{rank} \, \gamma\hspace{1mm} -1 \leq \dim ({\Bbb E}_{\it b})_{\gamma(t)}}, \,\forall t\in {\Bbb R}\end{eqnarray} and then by exploiting double inductive arguments \begin{eqnarray}\label{inequality-2}\displaystyle{\rm{rank} \, \gamma\hspace{1mm}-1 \geq \dim ({\Bbb E}_{\it b})_{\gamma(t)}}, \,\forall t\in {\Bbb R}.\end{eqnarray}

Let $r = \rm{rank}\, \gamma$.  Let $\{ v_1(t),\cdots, v_n(t) \}$ be an orthonormal basis consisting of parallel vector fields along $\gamma$ such that $\gamma'(t)=v_1(t)$ and $v_j(t)$, $j = 2,\cdots,r$ are Jacobi vector fields. It suffices to show $\displaystyle{{\mathcal S}(t) v_j(t) = 0 }$ for $j = 2,\cdots,r$ in terms of the shape operator ${\mathcal S}(t)$ of the horosphere ${\mathcal H}_{(\theta,\gamma(t))}$. Then, the inequality (\ref{inequality-1}) is proved.
Define, for this, a set of perpendicular Jacobi fields $\displaystyle{\{ Y_j = Y_j(t)\hspace{2mm}\vert\, j= 2,\cdots,n,\, t\in {\Bbb R}\}
}$ satisfying $Y_j(0) = v_j(0)$ and $\vert Y_j(t)\vert \leq C$, $\forall t \geq 0$ for some $C> 0$,\, $j= 2,\cdots, n$. Note these Jacobi vector fields are stable.  Here $Y_j(t) = v_j(t)$, $t\in {\Bbb R}$ for $j = 2,\cdots,r$, since each $v_j(t)$ is a parallel Jacobi field. Thus $\{ Y_2=Y_2(t), \cdots, Y_n=Y_n(t)\}$  defines a stable, perpendicular Jacobi tensor ${\mathcal J} = {\mathcal J}(t)$ along $\gamma$ by
\begin{eqnarray}
{\mathcal J}(t)\, &:&\, \gamma(t)^{\perp} \rightarrow \gamma(t)^{\perp}, \\ \nonumber
 & & \sum_{j=2}^{n} c^i v_i(t) \mapsto \sum_{j=2}^{n} c^i Y_i(t)
\end{eqnarray}from which the shape operator ${\mathcal S}(t)$ of the horosphere ${\mathcal H}_{(\theta,\gamma(t))}$ at $\gamma(t)$ is described by
$\displaystyle{{\mathcal S}(t) = {\mathcal J}'(t)\, {\mathcal J}^{-1}(t)  }$. Refer to \cite{ISKyushu, Kn2002} for construction of the stable Jacobi tensor and its relation with the shape operator ${\mathcal S}(t)$.
By definition of ${\mathcal J}(t)$ we see for each $j = 2,\cdots, r$\, ${\mathcal J}(t) v_j(t) = v_j(t)$ so that ${\mathcal J}'(t) v_j(t) = 0$ and ${\mathcal J}^{-1}(t)v_j(t) = v_j(t)$ and therefore ${\mathcal S}(t)v_j(t) = {\mathcal J}'(t)\, {\mathcal J}^{-1}(t) v_j(t) = {\mathcal J}'(t) v_j(t) = 0$ so we get (\ref{inequality-1}).

\vspace{2mm}
Now we will show (\ref{inequality-2}) as follows.

For this we fix $t = 0$ without loss of generality and verify that for $u\in T_{\gamma(0)} {\mathcal H}_{(\theta,\gamma(0))}$ satisfying ${\mathcal S}(0)u= 0$ there exists a parallel vector field $u(t)$ along $\gamma$ such that $u(0)=u$ which satisfies $R(t) u(t) = 0$, $\forall t\in {\Bbb R}$. This proves (\ref{inequality-2}).

Let $\displaystyle{\mathcal{H}_{(\theta,\gamma(t))} = \{ y\in M \vert \ b_{\theta}(y) = b_{\theta}(\gamma(t))\}}$, $t\in {\Bbb R}$, be a horosphere centered at $\theta$ and passing through $\gamma(t)$.
The restriction of the Hessian to a horosphere $\mathcal{H}_{(\theta,\gamma(t))} $ gives minus signed second fundamental form $h$, as seen at (\ref{secondff}), Section 2. Namely, we have $\displaystyle{(\nabla d b_{\theta})_{\vert \mathcal{H}_{(\theta,\gamma(t))}}(\cdot,\cdot) = - \langle {\mathcal S}(t)\cdot, \cdot \rangle }$ with respect to the operator $\displaystyle{{\mathcal S}(t)}$ which is negative semi-definite from the positive semi-definiteness of $\nabla d b_{\theta}$.

Let $u(t)$ be a parallel vector field along $\gamma(t)$ such that $u(0) = u$. Similarly for $v\in T_{\gamma(0)}{\mathcal H}_{(\theta,\gamma(0))}$ we define a parallel vector field $v(t)$ along $\gamma(t)$ such that $v(0) = v$.

From the Riccati equation we have along $\gamma(t)$
\begin{eqnarray}\label{riccati-1}
\langle {\mathcal S}(t)u(t), v(t) \rangle' + \langle {\mathcal S}(t)u(t), {\mathcal S}(t)v(t) \rangle + \langle R(t)u(t), v(t) \rangle = 0
\end{eqnarray}or
\begin{eqnarray}\label{riccati-2}
\langle {\mathcal S}'(t)u(t), v(t) \rangle + \langle {\mathcal S}(t)u(t), {\mathcal S}(t)v(t) \rangle + \langle R(t)u(t), v(t) \rangle = 0.
\end{eqnarray}

In order to obtain Proposition \ref{proposition-2} below we need to take the derivative of ${\mathcal S}(t)$.  Since $\displaystyle{{\mathcal S}'(t) u(t) = \frac{\nabla}{d t}\left({\mathcal S}(t)u(t)\right) }$, ${\mathcal S}'(t)$ is an endomorphism of $T_{\gamma(t)}\mathcal{H}_{(\theta,\gamma(t))}$.
  For any integer $k \geq 2$\hspace{2mm}
$\displaystyle{{\mathcal S}^{(k)}(t) u(t) = \left(\frac{\nabla}{d t}\right)^k\left({\mathcal S}(t)u(t)\right) }$ is also an endomorphism of $T_{\gamma(t)}\mathcal{H}_{(\theta,\gamma(t))}$.
\begin{proposition}\label{proposition-2}\rm
Let $(M,g)$ be a Hadamard manifold. Let $x\in M$ and
$\theta\in
\partial M$ and $\gamma$ be a geodesic satisfying $\gamma(0) = x$
and $[\gamma] = \theta$.

If there exists a non zero $u\in T_x \mathcal{H}_{(\theta,x)}$ satisfying ${\mathcal S}(0) u = 0$, then
\begin{eqnarray}
 \left(\langle {\mathcal S}(t)u(t), v(t) \rangle^{(k)}\right)_{\vert t= 0} &=& 0, \\ \nonumber
  \left(\langle R(t)u(t), v(t) \rangle^{(k)}\right)_{\vert t=0} &=& 0
\end{eqnarray}for any $v\in T_x \mathcal{H}_{(\theta,x)}$ and any integer $k \geq 0$.
\end{proposition}

An inductive proof will be given in the next section.
\begin{theorem}\rm
 Let $(M,g)$ be an analytic Hadamard manifold.  Let $x\in M$ and $\theta\in \partial M$ and $\gamma$ be a geodesic satisfying $\gamma(0) = x$ and $[\gamma] = \theta$.  Suppose that every Busemann function $b_{\theta}$ of $M$ is analytic.

 \vspace{1mm}If there exists a non zero $u\in T_x \mathcal{H}_{(\theta,x)}$ satisfying ${\mathcal S}(0) u = 0$, then $\displaystyle{{\mathcal S}(t)u(t) = 0}$ and $\displaystyle{R(t)u(t) = 0}$ for any $t\in \mathbb{R}$ so that $u(t)$ is a parallel Jacobi field along $\gamma$.
\end{theorem}

This theorem is a direct consequence of Proposition \ref{proposition-2}.

\vspace{1mm}
Let $\{ e_i(t); i = 2,\dots,n\}$ be a parallel, perpendicular orthonormal fields along $\gamma$, $n = \dim X$.   Then from Proposition \ref{proposition-2} we have $\displaystyle{\left(\langle {\mathcal S}(t)u(t), e_i(t) \rangle^{(k)}\right)_{\vert t= 0} = 0}$ for any $k \geq 0$.
 Since $(X,g)$ is analytic and every Busemann function is analytic on $M$, the Hessian $\nabla d b_{\theta}$ is also analytic on $M$ so that ${\mathcal S}(t)$ is analytic along the geodesic $\gamma$ and hence $\displaystyle{\langle {\mathcal S}(t)u(t), v(t) \rangle} = 0$, that is, ${\mathcal S}(t) u(t) = 0$ for $-\infty < t < \infty$. Similarly we have $R(t) u(t) = 0$ for $-\infty < t < \infty$.

 \vspace{2mm}
 From this theorem the inequality (\ref{inequality-2}) is derived.

 \section{Proof of Proposition \ref{proposition-2}}

We assume that there exists a non-zero vector $u\in T_{\gamma(0)}{\mathcal H}_{(\theta,\gamma(0))} =\gamma'(0)^{\perp}$ which satisfies ${\mathcal S}(0)u = 0$.
\begin{lemma}\rm Under this assumption it holds
\begin{eqnarray}
\langle {\mathcal S}(0)^{(j)}u,u\rangle &= 0,\, j= 0, \cdots,4,\\ \nonumber
\langle R(0)^{(j')}u,u\rangle &= 0,\, j'= 0,\cdots,3,
\end{eqnarray}
and moreover for any $v\in T_{\gamma(0)}{\mathcal H}_{(\theta,\gamma(0))}=\gamma'(0)^{\perp}$
\begin{eqnarray}
\langle {\mathcal S}'(0)u,v\rangle = 0, \hspace{2mm} 
\langle R'(0)u,v\rangle = 0. 
\end{eqnarray}
\end{lemma}

\begin{proof}\rm
\begin{assertion}\label{suu-a}
\begin{eqnarray}\label{suu-1}
\langle {\mathcal S}(0)u,u \rangle = 0,\hspace{2mm}\langle {\mathcal S}'(0)u,u \rangle = 0.
\end{eqnarray}\rm The first equality is obvious. The second one follows from the negative semi-definiteness of ${\mathcal S}(t)$ for any $t$. In fact, set $f(t) := \langle {\mathcal S}(t)u(t),u(t) \rangle$. Then $f(0) = 0$. 
Suppose $f'(0) = \langle {\mathcal S}'(0)u,u\rangle \not= 0$.
 Then, we may assume $f'(0) = \langle {\mathcal S}(t)u(t),u(t) \rangle'_{\vert t=0} =: a > 0$ without loss of generality. Then, since $f(t)$ is analytic, there exists $t_0 > 0$ such that $f'(t) > \frac{a}{2}$ for $0 < t < t_0$. From Maclaurin expansion we have
$\displaystyle{f(t) = f(0) + f'(\theta t) t > \frac{a}{2} t > 0}$, for any $0 < t < t_0$ ($0 < \theta < 1$). This is a contradiction, because $\langle {\mathcal S}(t) \cdot, \cdot\rangle $ is negative semi-definite for any $t$.
\end{assertion}
\end{proof}

The following is immediate from the assumption above.
\begin{assertion}\label{suv-a}\rm
\begin{eqnarray*}
\langle {\mathcal S}(0)u,v\rangle = 0,\hspace{2mm}\forall v\in T_{\gamma(0)}{\mathcal H}.
\end{eqnarray*}
\end{assertion}

\vspace{2mm}Therefore, the Riccati equation implies then
\begin{assertion}\label{ruu-a}\rm
\begin{eqnarray}\label{ruu-0}
\langle R(0)u, u\rangle= 0, \hspace{2mm}\langle R'(0)u, u\rangle= 0.
\end{eqnarray}
\end{assertion}
In fact, set $v=u$ in the equation (\ref{riccati-2})  to get \begin{eqnarray}\label{riccati-x}
\langle {\mathcal S}'(t)u(t), u(t) \rangle + \langle {\mathcal S}(t)u(t), {\mathcal S}(t)u(t) \rangle + \langle R(t)u(t), u(t) \rangle = 0.
\end{eqnarray} Then, as (\ref{suu-a}) and (\ref{suv-a}) show, the first and the second terms of the equation (\ref{riccati-x}) at $t = 0$ vanish so that
 the third term $\langle R(0)u, u\rangle$ must be zero. From the negative semi-definiteness of the Jacobi operator $R(t)$ one obtains $\langle R'(0)u, u\rangle = 0$ similar to the
 proof of Assertion \ref{suu-a}. 
 Since $R(t)\leq 0$, $\forall t$, (\ref{ruu-a}) implies
\begin{assertion}\rm
\begin{eqnarray}\label{ruv-0}
 \langle R(0)u,v\rangle = 0,\hspace{2mm}\forall v\in T_{\gamma(0)} {\mathcal H}_{(\theta,\gamma(0))}.
\end{eqnarray}

\end{assertion}
By using these facts and applying the Riccati equation, one has
\begin{assertion}\rm
\begin{eqnarray}\label{suu-b}
 \langle {\mathcal S}''(0)u,u \rangle = 0, \hspace{2mm}
\hspace{2mm} \langle {\mathcal S}'''(0)u,u \rangle = 0.
\end{eqnarray}
This is shown by the following argument. Differentiate the equation (\ref{riccati-x}) with respect to $t$ at $t=0$.  
Then the second term reduces to $2\langle {\mathcal S}'(0)u, {\mathcal S}(0)u\rangle$ which vanishes from Assertion \ref{suv-a}. The third term vanishes from Assertion \ref{ruu-a} so one gets the first equality of (\ref{suu-b}). From the negative semi-definteness of ${\mathcal S}(t)$ together with the Maclaurin expansion the second one is obtained.

\end{assertion}
\begin{assertion}\label{suv-b}\rm
\begin{eqnarray}\label{suv-1}
 \langle {\mathcal S}'(0)u,v \rangle = 0, \hspace{2mm}\forall v\in T_{\gamma(0)}{\mathcal H}_{(\theta,\gamma(0))}
\end{eqnarray}i.e., ${\mathcal S}'(0)u = 0$.
\end{assertion}
\begin{proof}\rm
In the equation (\ref{riccati-2})
 one set $t = 0$, where $v(t)$ is a parallel vector field along $\gamma$ such that $v(0)=v$. Then, the second term $\langle {\mathcal S}(0)u, {\mathcal S}(0)v\rangle$ vanishes from (\ref{suv-a}) and third term vanishes from (\ref{ruv-0}) so the first term must vanish.
\end{proof}

From the Riccati equation one has
\begin{assertion}
\begin{eqnarray}\label{ruu-2}
 \langle R''(0)u, u\rangle = 0, \hspace{2mm}\langle R'''(0)u, u\rangle = 0.
\end{eqnarray}
\end{assertion}
In fact one has by differentiating twice (\ref{riccati-2}) and setting $v(t) = u(t)$
\begin{eqnarray*}
\langle {\mathcal S}'''(t)u(t),u(t)\rangle + \langle {\mathcal S}(t)u(t), {\mathcal S}(t)u(t)\rangle'' + \langle R(t)u(t),u(t)\rangle'' = 0.
\end{eqnarray*}Here, the first term $\langle {\mathcal S}'''(0)u,u\rangle$ vanishes from (\ref{suu-b}) and the second term becomes
\begin{eqnarray*}
\langle {\mathcal S}(t)u(t),{\mathcal S}(t)u(t)\rangle'' = 2 \langle {\mathcal S}''(t)u(t),{\mathcal S}(t)u(t)\rangle + \langle {\mathcal S}'(t)u(t),{\mathcal S}'(t)u(t)\rangle
\end{eqnarray*}which also vanishes at $t=0$ from  Assertion \ref{suv-b}. So, one obtains the first equality of (\ref{ruu-2}). One has immediately then by the Maclaurin expansion
\begin{eqnarray*}\label{ruu-3}
 \langle R'''(0)u, u\rangle = 0
\end{eqnarray*}from the negative semi-definiteness of $R(t)$.

 \vspace{2mm}
 Now, we will investigate the following.

\begin{assertion}\label{assertionruv-b}\rm
\begin{eqnarray*}\label{ruv-b}
 \langle R'(0)u, v\rangle = 0,\hspace{2mm}\forall v\in T_{\gamma(0)}{\mathcal H}_{(\theta,\gamma(0))},
\end{eqnarray*}that is, $R'(0)u = 0$.
\end{assertion}

\begin{proof}\rm
Suppose that there exists a $v$ such that $\langle R'(0)u, v\rangle \not= 0$. Then, one may assume without loss of generality $\langle R'(0)u, v\rangle > 0$.  Now one can consider the quadratic form $\langle R(t)w(t), w(t)\rangle$, $w(t) = (u(t)+xv(t))\in T_{\gamma(t)}{\mathcal H}_{(\theta,\gamma(0))}$, $x\in \mathbb{R}$.
Since
\begin{eqnarray*}
\langle R(t)w(t), w(t)\rangle &=& \langle R(t)u(t), u(t)\rangle + 2 x \langle R(t)u(t), v(t)\rangle + x^2 \langle R(t)v(t), v(t)\rangle \\ \nonumber
&=& {\hat f}(t) + 2 {\hat g}(t) x + {\hat h}(t) x^2
\end{eqnarray*}where

\begin{eqnarray}
 {\hat f}(t) :=\langle R(t)u(t), u(t)\rangle,
 \hspace{2mm}{\hat g}(t)  :=\langle R(t)u(t), v(t)\rangle\hspace{2mm}{\rm and}\hspace{2mm}{\hat h}(t) :=\langle R(t)v(t), v(t)\rangle
\end{eqnarray} are functions with respect to $t$.  One has then from the previous arguments by using the Maclaurin  expansion
\begin{eqnarray*}
{\hat f}(t) = \frac{{\hat f}^{(4)}(\theta t)}{4!} t^4,\, 0<\theta<1,
\end{eqnarray*}since ${\hat f}^{(j)}(0) = 0$ for $j = 0, \cdots, 3$ and similarly

\begin{eqnarray*}
 {\hat g}(t) = {\hat g}'(\theta' t) t,\, 0<\theta' < 1,
\end{eqnarray*} since ${\hat g}(0) = 0$. The discriminant $D = D(t)$ of $x$ is given, therefore, by
\begin{eqnarray*}
 D(t) = \left({\hat g}(t)\right)^2 - {\hat f}(t) {\hat h}(t) = \left\{{\hat g}'(\theta' t)\right\}^2  t^2 - \frac{{\hat f}^{(4)}(\theta t)}{4!} t^4\cdot {\hat h}(t).
 \end{eqnarray*}One may assume ${\hat h}(0) < 0$, since, if, otherwise, ${\hat h}(0) = 0$, then $v = v(0)$ and hence $w = u + y v$, $y\in {\Bbb R}$ is also an eigenvector of $R(0)$ with eigenvalue zero so that $\langle R'(0)w,w\rangle = 2 y \langle R'(0)u,v\rangle$ must vanish and this leads to a contradiction. Thus, one has
 \begin{eqnarray*}
  D(t) = \left\{{\hat g}'(\theta' t)\right\}^2  t^2 \left\{ 1 - \frac{{\hat f}^{(4)}(\theta t)}{4!}/\big\{{\hat g}'(\theta' t)\big\}^2\,  t^2\cdot h(t)\right\}.
 \end{eqnarray*}Here, for a $t$ of sufficiently small $\vert t\vert$  $\vert{\hat f}^{(4)}(\theta t)\vert \leq C$, for a $C> 0$ from the continuity of ${\hat f}(t)$ and moreover ${\hat g}'(\theta' t) > 0$ from the assumption in proving argument. Therefore, there exists some $t$ of sufficiently small $\vert t\vert$ for which $D(t) > 0$ so that one can choose an $x$ such that $\langle R(t)w(t),w(t) \rangle = {\hat f}(t) + 2 x {\hat g}(t) + x^2 {\hat h}(t) > 0$ for such  $t$. This is a contradiction, since $\langle R(t)w(t),w(t) \rangle \leq 0$ for any $w(t)$ and one can conclude $\langle R'(0)u,v\rangle = 0$ for any $v$.
\end{proof}

 \begin{assertion}\rm
 \begin{eqnarray*}\label{suv-d}
 \langle {\mathcal S}^{(4)}(0)u, u\rangle = 0.
 \end{eqnarray*}
\end{assertion}

This follows from the  Riccati equation argument 
\begin{eqnarray*}
\langle {\mathcal S}'(t)u(t),u(t)\rangle^{(3)} + \langle  {\mathcal S}(t)u(t), {\mathcal S}(t)u(t)\rangle^{(3)}
+\langle R(t)u(t),u(t)\rangle^{(3)} =0.
\end{eqnarray*}

 The second term reduces  at $t=0$ by Leibniz' formula to
 \begin{eqnarray*}
 \langle  {\mathcal S}(t)u(t), {\mathcal S}(t)u(t)\rangle^{(3)}\vert_{t=0} = 2\langle  {\mathcal S}(0)u, {\mathcal S}'''(0)u\rangle + 6 \langle  {\mathcal S}'(0)u, {\mathcal S}''(0)u\rangle = 0.
 \end{eqnarray*}Furthermore at $t=0$
 $\displaystyle{\langle R(t)u(t),u(t)\rangle^{(3)}\vert_{t=0} = \langle R'''(0)u,u\rangle =0}$ from (\ref{ruu-2}). Thus the lemma is proved.

\begin{proposition}\rm Let $k$ be a positive integer. Suppose for this integer $k$

\noindent
(i)\,
\begin{eqnarray}\label{suu0to2k}\displaystyle{\langle {\mathcal S}(t)u(t), u(t)\rangle^{(j)}_{\vert t=0} = 0}, j = 0,\cdots, 2k,
\end{eqnarray}  (ii)\, \begin{eqnarray}\label{suv0tok-1}\displaystyle{\langle {\mathcal S}(t)u(t), v(t)\rangle^{(j_1)}_{\vert t=0} = 0,} \hspace{2mm}v\in T_{\gamma(0)} \mathcal{H}_{(\theta,\gamma(0))},\hspace{2mm}
 j_1 = 0,\dots, k-1\end{eqnarray} 
 and moreover \\
 (iii) \begin{eqnarray}\label{ruu0to2k-1}\displaystyle{\langle R(t)u(t), u(t)\rangle^{(j_2)}_{\vert t=0} = 0},\,  j_2 = 0,\cdots, 2k-1\end{eqnarray}and  (iv)
\begin{eqnarray}\displaystyle{\langle R(t)u(t), v(t)\rangle^{(j_3)}_{\vert t=0} = 0,}\hspace{2mm}v\in T_{\gamma(0)} \mathcal{H}_{(\theta,\gamma(0))}, \hspace{2mm} j_3 = 0,\cdots, k-1.\end{eqnarray} 
 Then, it is concluded

\vspace{2mm}\noindent
(a)\,
\begin{eqnarray}\label{suu2k+1}\langle {\mathcal S}(t)u(t), u(t)\rangle^{(j)}_{\vert t=0} = 0, \hspace{2mm}j= 2k+1, 2k+2
\end{eqnarray}  (b)\, \begin{eqnarray}\label{suvk}\displaystyle{\langle {\mathcal S}(t)u(t), v(t)\rangle^{(k)}_{\vert t=0} = 0,}\hspace{2mm}\forall v\in T_{\gamma(0)}{\mathcal H}_{\theta,\gamma(0)}
\end{eqnarray}
and moreover
 (c) \begin{eqnarray}\label{ruu2k}
 \langle R(t)u(t), u(t)\rangle^{(j_1)}_{\vert t=0} = 0, \hspace{2mm}j_1= 2k, 2k+1
 \end{eqnarray}and  (d)
\begin{eqnarray}\label{ruvk}\displaystyle{\langle R(t)u(t), v(t)\rangle^{(k)}_{\vert t=0} = 0, \hspace{2mm}\forall v\in T_{\gamma(0)}{\mathcal H}_{\theta,\gamma(0)}.}
\end{eqnarray} 

\end{proposition}

\vspace{1mm}\noindent
{\bf Proof.}\hspace{2mm} First we show 
$\displaystyle{\langle {\mathcal S}(t)u(t),u(t)\rangle^{(2k+1)}_{\vert t=0} = \langle {\mathcal S}^{(2k+1)}(0)u,u\rangle= 0}$.
 For this we suppose $a:= \langle {\mathcal S}(t)u(t),u(t)\rangle^{(2k+1)}_{\vert t=0} = \langle {\mathcal S}^{(2k+1)}(0)u,u\rangle  >0$ without loss of generality.  Let $f(t) :=\langle {\mathcal S}(t)u(t),u(t)\rangle$. Then by using the Maclaurin expansion  the function $f(t)$ is written from (i) as
\begin{eqnarray*}
 f(t) :=  \frac{1}{(2k+1)!} f^{(2k+1)}(\theta t) t^{(2k+1)},\hspace{2mm}0 <\theta <1
\end{eqnarray*}for $t$ of sufficiently small $\vert t\vert$. There exists, then, a $t_0> 0$ such that $f^{(2k+1)}(t) \geq \frac{a}{2} > 0$ and hence $f(t)$ must be positive for $0< t < t_0$ . This is a contradiction and hence we obtain (\ref{suu2k+1}),
since ${\mathcal S}(t)$ is negative semi-definite.  In case of $a < 0$, we may take a negative $t$ of sufficiently small $\vert t\vert$.

 Now we will show
\begin{eqnarray*}
\displaystyle{\langle {\mathcal S}(t)u(t), v(t)\rangle^{(k)}_{\vert t=0} = 0}
\end{eqnarray*} for any $v\in T_{\gamma(0)} \mathcal{H}_{(\theta,\gamma(0))}$. For this we suppose similarly as in proof of Assertion \ref{assertionruv-b}
\begin{eqnarray*}\displaystyle{\langle {\mathcal S}(t)u(t), v(t)\rangle^{k}_{\vert t=0} >0}
\end{eqnarray*}and consider the quadratic form $\langle {\mathcal S}(t)w(t), w(t)\rangle$, $w(t) = u(t)+xv(t)$ for any $x\in {\Bbb R}$. Then
\begin{eqnarray}\label{quadratic}
\langle {\mathcal S}(t)(w(t), w(t)\rangle &=& f(t) + 2 x g(t) + x^2 h(t),\\ \nonumber
g(t) &:=& \langle {\mathcal S}(t)u(t),v(t)\rangle, \\ \nonumber
h(t) &:=& \langle {\mathcal S}(t)v(t),v(t)\rangle.
\end{eqnarray}Here  $f(t)$ is written by
\begin{eqnarray*} f(t)= \frac{1}{(2k+2)!} f^{(2k+2)}(\theta_1 t) t^{(2k+2)},\, 0 < \theta_1 < 1
\end{eqnarray*}
from (\ref{suu0to2k}) and (\ref{suu2k+1}). Moreover  $g(t)$ is expanded from (\ref{suv0tok-1}) as
\begin{eqnarray*}
 g(t) = \frac{1}{k!} g^{(k)}(\theta_2 t) t^k,\, 0 < \theta_2 < 1.
\end{eqnarray*} From (\ref{suvk}) we may assume $g^{(k)}(\theta_2 t) \not= 0$ for $t > 0$ of sufficiently small $\vert t\vert$. Since $h(0) \leq 0$, we can assume $h(0) < 0$ (if $h(0) = 0$ is supposed, then $v(0)$ is an eigenvector of eigenvalue zero so for any $x$\hspace{2mm} $u + x v$ is also an eigenvector of ${\mathcal S}(0)$ corresponding to eigenvalue zero and consequently $\langle {\mathcal S}(0)(u+xv),u+xv\rangle = 0$. However, $\langle {\mathcal S}(0)(u+xv), u+xv\rangle = \langle {\mathcal S}(0)u,u\rangle + 2 x \langle {\mathcal S}(0)u,v\rangle + x^2 \langle {\mathcal S}(0)v,v\rangle = 2 x \langle {\mathcal S}(0)u,v\rangle$. This is a contradiction, since (\ref{suv0tok-1}) is assumed in the proposition and we assumed (\ref{suvk})).
Return back to the main argument.
Compute the discriminant $D(t)=g^2(t) - f(t) h(t)$ of (\ref{quadratic}) as
\begin{eqnarray*}
 D(t) 
 &=& \left(\frac{1}{k!}\right)^2\big\{ g^{(k)}(\theta t) t^k\big\}^2 - \, \frac{h(t)}{(2k+2)!} f^{(2k+2)}(\theta_1 t) t^{(2k+2)} \\ \nonumber
 &=&  \left(\frac{1}{k!}\right)^2\big\{ g^{(k)}(\theta t)\big\}^2 t^{2k}\Big[1 -  \left(\frac{1}{k!}\right)^{-2}\big\{ g^{(k)}(\theta t)\big\}^{-2}\, \frac{h(t)}{(2k+2)!} f^{(2k+2)}(\theta_1 t) t^{2}\Big].
\end{eqnarray*}Since
\begin{eqnarray*}\Big\vert\left(\frac{1}{k!}\right)^{-2}\big\{ g^{(k)}(\theta t)\big\}^{-2}\, \frac{h(t)}{(2k+2)!} f^{(2k+2)}(\theta_1 t) t^{2}\Big\vert \leq c\hspace{0.5mm}t^2,
\end{eqnarray*}for $t$ of sufficiently small $\vert t\vert$, the terms in the parenthesis tend to $1$ as $t \rightarrow 0$. Consequently we have $D(t) > 0$ for any $t$ of sufficiently small $\vert t\vert$ so that the quadratic function (\ref{quadratic}) with $h(t) < 0$ has a positive value for a certain $x\in {\Bbb R}$.
This is a contradiction so that (\ref{suvk}) is verified.

We will next show (\ref{ruu2k}) in the following way.

First we show  the first equality of (\ref{ruu2k})
by differentiating $2k$ times the Riccati equation (\ref{riccati-2}) at $t=0$ as
\begin{eqnarray}\label{certainequation}
& &\langle {\mathcal S}'(t)u(t),u(t)\rangle^{(2k)}\vert_{t=0} \\ \nonumber
&+& \langle {\mathcal S}(t)u(t),{\mathcal S}(t)u(t)\rangle^{(2k)}\vert_{t=0} \\ \nonumber
&+& \langle R(t)u(t),u(t)\rangle^{(2k)}\vert_{t=0} = 0.
\end{eqnarray}The first term, as reduced to $\langle {\mathcal S}^{(2k+1)}(0)u,u\rangle = 0$ vanishes from (\ref{suu2k+1}). To assert $\langle {\mathcal S}(t)u(t),{\mathcal S}(t)u(t)\rangle^{(2k)}\vert_{t=0} = 0$ we apply Leibniz formula
\begin{eqnarray*}
\langle {\mathcal S}(t)u(t),{\mathcal S}(t)u(t)\rangle^{(2k)} &=& \sum_{r=0}^{2k} \left(\begin{array}{c}2k\\r\end{array}\right)\langle {\mathcal S}^{(r)}(t)u(t), {\mathcal S}^{(2k-r)}(t)u(t)\rangle\\ \nonumber
&=& 2\,\sum_{r=0}^{k-1} \left(\begin{array}{c}2k\\r\end{array}\right)\langle {\mathcal S}^{(r)}(t)u(t), {\mathcal S}^{(2k-r)}(t)u(t)\rangle \\ \nonumber
& &+\left(\begin{array}{c}2k\\k\end{array}\right)\langle {\mathcal S}^{(k)}(t)u(t), {\mathcal S}^{(k)}(t)u(t)\rangle.
\end{eqnarray*}When $t = 0$, all the terms of right hand vanish from the assumption (\ref{suv0tok-1}) together with (\ref{suvk}). Thus, the second term of (\ref{certainequation}) vanishes. Therefore the third term vanishes, namely  the first equality of (\ref{ruu2k}) is proved. The second equality of (\ref{ruu2k}) is shown by using the Maclaurin expansion, similar to the proof of the first equality of (\ref{suu2k+1}). 

  Now we will show
\begin{eqnarray*}
\displaystyle{\langle R(t)u(t), v(t)\rangle^{(k)}_{\vert t=0} = 0}
\end{eqnarray*} for any $v\in T_{\gamma(0)} \mathcal{H}$.  However, this is verified by using an argument just similar to the argument given for proving (\ref{suvk}) so we omit a proof.

\vspace{2mm}To see finally $\displaystyle{\langle S(t)u(t), u(t)\rangle^{(2k+2)}_{\vert t=0} = 0}$ we differentiate the Riccati equation $2k + 1$ times
\begin{eqnarray*}
 \langle {\mathcal S}'(t)u(t),u(t) \rangle ^{(2k+1)} +  \langle {\mathcal S}(t)u(t),{\mathcal S}(t)u(t) \rangle^{(2k+1)} +  \langle R(t)u(t),u(t) ^{(2k+1)} = 0.
\end{eqnarray*}Here the third term vanishes at $t = 0$ from (\ref{ruu2k}).
We see that the second term vanishes at $t = 0$, since  we can assume ${\mathcal S}^{(j)}(0) u = 0$ for $j = 0,\cdots, k$ from the assumption (\ref{suv0tok-1}) and (\ref{suvk}) to observe by the aid of the Leibniz formula
\begin{eqnarray*}
\langle {\mathcal S}(t)u(t),{\mathcal S}(t)u(t) \rangle^{(2k+1)}_{\vert t=0}  = 2 \sum_{r= 0}^{k}  \left(\begin{array}{c} 2k+1\\ r\end{array}\right) \langle {\mathcal S}^{(r)}(0)u, {\mathcal S}^{(2k+1- r)}(0) u\rangle = 0.
\end{eqnarray*}Therefore, $\langle {\mathcal S}^{(2k+2)}(0)u,u \rangle = 0$, that is, the second equality of (\ref{suu2k+1}) is obtained.

\section{visibility axiom and the Hessian of Busemann function}

First we  will prove Proposition \ref{visibilitycriterion}.

Let $b_{\theta}$ be the Busemann function associated with $\theta \in \partial M$. Let $p\in M$ and $\gamma : {\Bbb R} \rightarrow M$ be a geodesic of $\gamma(0) = p$, $[\gamma] = \theta$. So $b_{\theta}$ may be represented by $b_{\theta}(\cdot) = b_{\gamma}(\cdot)$ neglecting an additive constant so $b_{\theta}$ satisfies $b_{\theta}(p) = 0$ and $b_{\theta}(\gamma(t)) = -t$ for any $t$. Now, let $\gamma_1$ be a geodesic of $\gamma_1(0) = p$ and $[\gamma_1] \not= \theta$. We will show $\lim_{t\rightarrow\infty} b_{\theta}(\gamma_1(t)) = +\infty$. We see  $\gamma_1'(0) \not= \gamma'(0)$ at $p$ from the non-positive curvature assumption and $[\gamma_1] \not= \theta$.

  Further, we can assume ${\gamma}_1'(0) \not= - {\gamma}'(0)$, since, if equality ${\gamma}_1'(0) = - {\gamma}'(0)$ holds, then $\gamma_1(t) = \gamma(-t)$ for any $t$, and hence $b_{\theta}(\gamma_1(t)) = b_{\theta}(\gamma(-t)) = t \rightarrow +\infty$ as $t \rightarrow +\infty$, so we get the desired result.

Let $\varphi$ be the angle between $\nabla b_{\theta}$ and $\gamma_1'(0)$, that is, $\cos \varphi = \langle (\nabla b_{\theta})_p, \gamma_1'(0)\rangle$. Since $\nabla b_{\theta}$ coincides with $- \gamma'(0)$ at $p$, we consider the following  cases; case (i): $0 < \varphi < \pi/2$,\, case (ii): $\varphi
 = \pi/2$, that is, $\gamma_1'(0)$ is tangent to the horosphere ${\mathcal H}_{(\theta,\gamma(0))}$ at $p$ and case (iii) $\pi/2 < \varphi < \pi$.

Define a function $q$ of $t$ by $\displaystyle{q(t) := b_{\theta}(\gamma_1(t)).}$
 Note $q(0) = 0$. Since $[\gamma_1] \not= \pm\, \theta$, $q(t)$ is strictly convex for any $t$. In fact, $q''(t)$ is written by
 \begin{eqnarray*}q''(t) = (\nabla d
b_{\theta})_{{\gamma}_1(t)}(\gamma'_1(t),\gamma'_1(t))
\end{eqnarray*}
and this value of the Hessian is positive.  If
the value is zero, otherwise, then, the velocity vector
$\gamma'_1(t)$ must be a null vector of the Hessian at the point $\gamma_1(t)$, that is,  
orthogonal to a horosphere centered at $\theta$. Thus $\gamma_1$ must be asymptotically equivalent to the geodesic $\gamma$ so that  $\gamma_1$ belongs to $\theta$ or $-\theta = [\gamma^-]$. But this possibility is excluded as before and hence  we have
$q''(t) = (\nabla d
b_{\theta})_{{\gamma}_1(t)}({\gamma}_1'(t),{\gamma}_1'(t)) >
0$. Then, $q(t)$ is strictly convex.

Consider case (i):\hspace{2mm}Since $0 < \varphi < \pi/2$, $0 < \cos \varphi < 1$.
We have

 $\displaystyle{q'(t) = < (\nabla b_{\theta})_{{\gamma}_1(t)}, \gamma'_1(t) >}$
so that $q'(0) = \cos \varphi > 0$. 
Therefore, there exists $t_0 > 0$ such that $q'(t) > \varepsilon >
0$ for an $\varepsilon > 0$ and any $t$ of $0\leq t \leq t_0$.  
Since $q''(t) > 0$ at any $t$, $q'(t)$ is strictly increasing in any
$t$ so that we conclude $q'(t) \geq \varepsilon >0$ for any $t
\geq t_0$ and hence for all $t > 0$ and also we may assume
$q(t_0) > 0$. 
Let $n > 0$ be an
 arbitrary positive number. Then,  from the strict convexity of $q(t)$\hspace{2mm} 
 $\displaystyle{
 q(t_0)  < \left(1- \frac{1}{n}\right)\, q(0) + \frac{1}{n}\, q(n \, t_0) = \frac{1}{n} q(n t_0)}$
 that is, $q(n t_0) > n q(t_0)$, in other words, $b_{\theta}(\gamma_1(n t_0)) > n\, b_{\theta}(\gamma_1(t_0))$. Notice $b_{\theta}(\gamma_1(t_0)) = q(t_0) > 0 $. So we get $\lim_{t\rightarrow\infty} b_{\theta}(\gamma_1(t)) = \infty$ from the continuity of $b_{\theta}$.

\vspace{2mm}Case (ii):\hspace{2mm}In this case $\gamma_1$ is tangent to the
horosphere ${\mathcal H}_{(\theta,\gamma(0))}$ at $t = 0$. We see then  $q'(0) = 0$. So,
from $q''(t)>0$ it follows that
for any $t$ there exists $a$ of $0<a<1$ such that
$\displaystyle{q(t) =  \frac{1}{2} q''(a t) \hspace{0.5mm}t^2 > 0}$.
This means that $q(t_0) > 0$ for some $t_0 > 0$ so that we conclude that $\lim_{t\rightarrow\infty} b(\gamma_1(t)) = \infty$, similarly as in case (i).

\vspace{2mm}
Case (iii):\hspace{2mm}We have $\cos \varphi < 0$. By the way, $q'(t) = \langle \nabla b_{\theta}, \gamma'_1(t)\rangle = \cos \varphi(t)$, where $\varphi(t)$ is the angle $\angle_{\gamma(t)}(\nabla b_{\theta}, \gamma_1'(t))$ between $\nabla b_{\theta}$ and $\gamma_1'(t)$.
Then
\begin{eqnarray}
\cos \varphi(t) - \cos\varphi(0) = \int_0^t q''(t) dt = \int_0^t \nabla d b_{\theta}(\gamma'_1(t), \gamma'_1(t))dt.
\end{eqnarray}
Here $\varphi(0) = \varphi$ and then $-1 < \cos \varphi(0) < 0$.  Thus, from the assumption of Corollary \ref{visibilitycriterion} there exists a $t_0 > 0$ such that $\cos \varphi(t_0)$ is positive, since

$\displaystyle{\cos \varphi(t_0) = \cos \varphi + \int_0^{t_0} \nabla d b_{\theta}(\gamma'_1(t), \gamma'_1(t))dt}$.
From the continuity of $\varphi(t)$ there exists a $t_1 > 0$ such that $\cos \varphi(t_1) = 0$. 
So, it is concluded by applying case (ii) that $b_{\theta}(\gamma_1(t))$ tends to $+\infty$ as $t \rightarrow \infty$.

\begin{proof} \rm of Proposition \ref{leasteigenvalue}.
\,
Note that the least eigenvalue $\lambda=\lambda(p)$ is a positive continuous function on $S$, since  $\nabla d b_{\theta}(\cdot,\cdot) > 0$.  Let $q(t) = b_{\theta}(\gamma_1(t))$ as before. Then,  $\displaystyle{q'(t) = \langle \nabla b_{\theta}, \gamma_1'(t) \rangle}$
 $\displaystyle{ = \cos\varphi(t)}$,\,  $\displaystyle{\varphi(t) = \angle_{\gamma_1(t)}(\nabla b_{\theta}, \gamma_1'(t))}$ and $q''(t) = \nabla d b_{\theta}(\gamma_1'(t), \gamma_1'(t))$.
 Since $\nabla d b_{\theta}(\nabla b_{\theta}, v) = 0$ for any vector $v$ tangent to a horosphere centered at $\theta$,  we have the equality
\begin{eqnarray*}\label{secondderiv}
 (\cos \varphi(t))' = \nabla d b_{\theta}(\gamma_1'(t)^{T}, \gamma_1'(t)^{T})
\end{eqnarray*}where $\displaystyle{\gamma_1'(t)^{T} := \gamma_1'(t) - \langle\nabla b_{\theta}, \gamma_1'(t)\rangle \nabla b_{\theta}}$ is tangential to the horosphere.  Then, by using $\lambda(t):=\lambda(\gamma_1(t))$ we have
\begin{eqnarray*}
  \nabla d b_{\theta}(\gamma_1'(t)^{T}, \gamma_1'(t)^{T}) \geq \lambda(t)\hspace{0.5mm} \vert \gamma_1'(t)^{T}\vert^2
  =\lambda(t) \left(1 - \cos ^2\varphi(t)\right)
\end{eqnarray*}so that
\begin{eqnarray*}
 (\cos \varphi(t))' \geq \lambda(t) \left(1 - \cos ^2\varphi(t)\right).
\end{eqnarray*}Let $y(t) := q'(t) = \cos \varphi(t)$. Then, this is reduced to $y'(t) \geq \lambda(t)(1- y^2(t))$, or equivalently to
\begin{eqnarray}\label{angleequation}
\displaystyle{ \frac{1}{1-y^2(t)}\, y'(t) \geq \lambda(t)}.
\end{eqnarray}

Here, $-1 < y(t) < 1$ for any $t$. In fact, if we assume, otherwise, $y(t) = \pm 1$ for a $t_0$,  then this means that $\varphi(t_0) = 0$ or $\pi$  and consequently either $[\gamma_1] = \theta$, which is excluded, or $[\gamma_1] = -\theta$ from which $b_{\theta}(\gamma_1(t)) = t \to + \infty$, as $t\to \infty$. But
this case has been considered.

Now from (\ref{angleequation}) we obtain
\begin{eqnarray*}
\log \sqrt{\frac{1+y(t)}{1-y(t)} } \geq \log \sqrt{\frac{1+y(t_1)}{1-y(t_1)} }\, \int_{t_1}^t \lambda(s) ds
\end{eqnarray*} for any $t, t_1$ $(t_1 < t)$, from which
\begin{eqnarray}\label{visibility}
\frac{1+y(t)}{1-y(t)}  \geq \frac{1+y(t_1)}{1-y(t_1)} \, \exp\left\{2\, \int_{t_1}^t \lambda(s) ds\right\}.
\end{eqnarray}
 We set $\displaystyle{F(t;t_1) :=
\frac{1+y(t_1)}{1-y(t_1)} \, \exp\left\{2\, \int_{t_1}^t \lambda(s) ds\right\}}$. Then, (\ref{visibility}) is written as
\begin{eqnarray*}
y(t) \geq \frac{F(t;t_1)- 1}{F(t;t_1)+1}.
\end{eqnarray*}
From the assumption of Proposition \ref{leasteigenvalue}, 
 namely from $\displaystyle{\lim_{t\rightarrow\infty}\int_{t_1}^t \lambda(s) ds = +\infty }$, it holds $F(t,t_1) - 1 > 0$ for a sufficiently large $t$ and thus $y(t) = \cos \varphi(t)> 0$,
that is,  $0 < \varphi(t) < \frac{\pi}{2}$ for
the angle $\varphi(t)= \angle_{\gamma_1(t)}\left(\nabla b_{\theta}, \gamma'_1(t)\right)$. Since the Hessian $\nabla d b_{\theta}$ is positive definite, we can apply the argument  in proving Proposition \ref{visibilitycriterion} in case (i) to obtain Proposition  \ref{leasteigenvalue}. 

\end{proof}

 Moreover we have the following result.
\begin{corollary}{\rm Let $(M,g)$ be a Hadamard manifold having positive definite Hessian $\nabla d b_{\theta}$, for any $\theta\in\partial M$.
If the least eigenvalue of $\nabla d b_{\theta}$ restricted to a horosphere centered at $\theta\in\partial M$ has a uniform lower bound from below over $M$, then $(M,g)$ fulfills visibility axiom.}
\end{corollary}
}

\vspace{2mm}A Damek-Ricci space $S$ satisfies the visibility axiom from Theorem \ref{positivedefiniteness}.   
The visibility of $S$ is also derived directly from 
the exact form of the Busemann function in Proposition \ref{busemann}.

\section*{Acknowledgements}
 \noindent {This research was supported by Basic Science Research Program through
the National Research Foundation of Korea(NRF)
 funded by the Ministry of Education (NRF-2016R1D1A1B03930449).}

\end{document}